\documentclass[11pt,reqno]{amsart}

\usepackage[dvipsnames]{xcolor}
\usepackage{amsmath}
\usepackage{amsthm}
\usepackage{graphicx}
\usepackage{upgreek}
\usepackage{hyperref}
\usepackage{mathrsfs}
\usepackage{bm,bbm}
\usepackage{amsfonts}
\usepackage{amssymb}
\usepackage{csquotes}
\usepackage{stmaryrd}
\usepackage{ulem}
\usepackage{tikz-cd}

\newcommand{\sredm}[1]{\ifmmode\text{\xout{\ensuremath{\displaystyle \textcolor{red}{#1}}}}\else\sout{\textcolor{red}{#1}}\fi}

\newcount\Comments  
\definecolor{darkgreen}{rgb}{0,0.5,0}
\newcommand{\kibitz}[2]{\ifnum\Comments=1\textcolor{#1}{#2}\fi}

\usepackage{enumerate}

\usepackage{soul}

\usepackage[margin=1.1in]{geometry}
\numberwithin{equation}{section}

\newtheorem{theorem}{Theorem}[section]
\newtheorem{lemma}{Lemma}[section]

\newtheorem{proposition}{Proposition}[section]
\newtheorem{corollary}{Corollary}[section]

\newtheorem{definition}{Definition}[section]
\newtheorem{remark}{Remark}[section]

\newcommand{\noi}{\noindent}

\newcommand{\fra}{\mathfrak{a}}

\newcommand{\E}{\mathbb{E}}
\newcommand{\R}{\mathbb{R}}

\newcommand{\N}{\mathbb{N}}

\newcommand{\eps}{\varepsilon}
\newcommand\numberthis{\addtocounter{equation}{1}\tag{\theequation}}

\newcommand{\one}{\boldsymbol{\mathbbm{1}}}

\newcommand{\vr}{\varrho}
\newcommand{\al}{\alpha}

\newcommand{\EE}{{\mathbb E}}
\newcommand{\FF}{{\mathbb F}}
\newcommand{\PP}{{\mathbb P}}

\newcommand{\calA}{{\mathcal A}}
\newcommand{\calB}{{\mathcal B}}
\newcommand{\calC}{{\mathcal C}}

\newcommand{\calM}{{\mathcal M}}
\newcommand{\calN}{{\mathcal N}}

\newcommand{\calP}{{\mathcal P}}
\newcommand{\calQ}{{\mathcal Q}}

\newcommand{\calS}{{\mathcal S}}
\newcommand{\calT}{{\mathcal T}}

\newcommand{\skp}{\vspace{\baselineskip}}

\newcommand{\pl}{\partial}

\newcommand\iy{\infty}

\newcommand{\A}{\mathbb{A}}

\newcommand{\limn}{\lim_{n\to\iy}}

\newcommand{\1}{\boldsymbol{\mathbbm{1}}}

\DeclareMathOperator*{\argmin}{arg\,min}

\newcommand{\TT}{\tilde T}
\newcommand{\mT}{m^{(T)}}
\newcommand{\mTT}{m^{(\TT)}}
\newcommand{\vT}{v^{(T)}}
\newcommand{\vTT}{v^{(\TT)}}

\newcommand{\Deltadot}{\Delta}

\title{Analysis of the Finite-State Ergodic Master Equation}
\author[A. Cohen]{Asaf Cohen }
\address{Department of Mathematics\\
University of Michigan\\
Ann Arbor, MI 48109\\
United States
}
\email{shloshim@gmail.com }

\author[E. Zell]{Ethan Zell}
\address{Department of Mathematics\\
University of Michigan\\
Ann Arbor, MI 48109\\
United States
}
\email{ezell@umich.edu}

\thanks{A. Cohen acknowledges the financial support of Research supported by the National Science Foundation (DMS-2006305). }

\def\namedlabel#1#2{\begingroup
    #2%
    \def\@currentlabel{#2}%
    \phantomsection\label{#1}\endgroup
}

\date{\today}

\begin{document}
\maketitle

\begin{abstract}

Mean field games model equilibria in games with a continuum of  players as limiting systems of symmetric $n$-player games with weak interaction between the players. We consider a finite-state, infinite-horizon problem with two cost criteria: discounted and ergodic. Under the Lasry--Lions monotonicity condition we characterize the stationary ergodic mean field game equilibrium by a mean field game system of two coupled equations: one for the value and the other for the stationary measure. This system is linked with the ergodic master equation. Several discounted mean field game systems are utilized in order to set up the relevant discounted master equations. We show that the discounted master equations are smooth, uniformly in the discount factor. Taking the discount factor to zero, we achieve the smoothness of the ergodic master equation. 
\end{abstract}
\skp
\noi{\bf Keywords:}    
Mean-field games, master equation, finite state control problem, Markov chains 

\noi{\bf AMS Classification:} 
60J27, 
35M99, 
91A15 

\tableofcontents

\section{Introduction}\label{sec1}

In this paper we study a stationary mean field game (MFG) with ergodic cost and finite state space. Under the Lasry--Lions monotonicity condition, we characterize the value of the game and the stationary distribution by a system of two coupled equations. The MFG is also characterized by a single partial differential equation, called the {\it master equation}:
\begin{align}
\tag{ME}\label{ME}
\varrho = H(x,\Delta_x U_0(\cdot,\eta))+\sum_{y,z\in[d]}\eta_y D^\eta_{yz} U_0(x,\eta)\gamma^*_z(y,\Delta_y U_0(\cdot,\eta))+F(x,\eta),
\end{align} 
where $x\in[d]:=\{1,\ldots,d\}$, $\eta$ belongs to the $(d-1)$-dimensional simplex, $\Delta_x U_0(\cdot,\eta):=(U_0(y,\eta)-U_0(x,\eta))_{y\in[d]}$ is the vector of differences, and  $D^\eta_{yz}U_0(x,\eta)$ stands for  a directional derivative on the $(d-1)$-dimensional simplex (given explicitly in Section \ref{sec11}). The data are the Hamiltonian $H$ and the function $F$ that represents a running cost part that is independent of the control; $\gamma^*$ is the unique optimal rate selector for the Hamiltonian, and specifically, satisfies $\gamma^*(y,p)=D_pH(y,p)$. The solution to this system is the couple $(\varrho,U_0)$, where $\varrho\in\R$ is the {\it value} of the ergodic MFG and $U_0$ is a real-valued function, which is referred to as the {\it potential function}.


The main result of this work is the well-posedness of the ergodic master equation \eqref{ME} in the classical sense. We prove existence and uniqueness of the solution, and inherit regularity properties from the related discounted master equation. Specifically, we go through various forward-backward linearized discounted systems to obtain well-posedness for a discounted master equation and the regularity of its solution, uniformly in the discount factor. Then we use a \textit{vanishing discount} approach to obtain the main result.


\subsection{Background} MFGs were introduced by Lasry and Lions in \cite{Lasry2006, LasryLions2} and by Caines, Huang, and Malham\'e in \cite{Huang2006} as limiting systems of $n$-player stochastic differential games of mean field type. Since their introduction, the theory of MFGs has greatly developed and results are summarized in the books \cite{Achdou_book, Bensoussan2013, CardaliaguetDelarueLasryLions, CarmonaDelarue_book_I, CarmonaDelarue_book_II}.

Analogous to a Nash equilibrium in a finite player game, MFGs also have a notion of equilibrium. 
From a stochastic control point of view, the MFG equilibrium is described as a fixed point of a best response map, which sends a given flow of measures to the distribution of a controlled state-dynamics. Alternatively, it is characterized as the solution to a coupled system of differential equations, called the {\it MFG system}. One of the equations is a Hamilton--Jacobi--Bellman  equation and the other is Kolmogorov's forward equation. 
%

In order to establish the connection between the many-player games and the respective MFG, there are two major problems of interest: (1) proving the well-posedness of the MFG system and (2) showing convergence of the $n$-player system/equilibria to the MFG system/equilibria. At the intersection of these two fundamental questions is the so-called master equation. This is a (partial) differential equation on the space of player distributions that characterizes the value of the MFG; as such, the  solution of the master equation satisfies certain properties that relates it to the solution of the MFG system from which it is derived. Cardaliaguet, Delarue, Lasry, and Lions developed a method in \cite{CardaliaguetDelarueLasryLions} that utilizes the master equation in order to prove the convergence of the $n$-player game to its mean field limit. Subsequently, the master equation approach was further used by \cite{bay-coh2019} and \cite{cec-pel2019} for finite state, finite horizon games. Later, \cite{BCCD} and \cite{BCCD-arxiv} utilized the master equation method for finite state MFGs with common noise while \cite{JEP_2021__8__1099_0} introduced a weak notion of master equation solutions that relied on a Lasry--Lions monotonicity assumption. While the Lasry--Lions monotonicity assumption is standard in MFG literature, novel conditions appear in \cite{gan-mes-mou-zha2022} and \cite{mou-zha2019}; the former \cite{gan-mes-mou-zha2022} offers an alternative \textit{displacement monotonicity} condition that allows analysis of non-separable Hamiltonians. Then in \cite{2022arXiv220110762M}, Zhang and Mou give an \textit{anti-monotonicity} condition which allows for model data outside of both the Lasry--Lions monotonicity and the displacement monotonicity conditions.

The master equation approach to proving convergence of the $n$-player system improves on techniques that directly compare the MFG system and the many-player Nash system, since the latter provides convergence only for the finite horizon problem, see e.g., \cite{Gomes2013}. For more about classical solutions of master equations, we refer the reader to: \cite{gan-swi2015, buc-li-pen-rai2017, ben-jam-yam2019}, and for weak/viscosity solutions: \cite{mou-zha2019, wu-zha2020, bur-ign-rep2020}. The master equation is also used to obtain fluctuation and large deviation results as in \cite{bay-coh2019, cec-pel2019, del-lac-ram2019, del-lac-ram2020} and for numerical schemes as in \cite{cha-cri-del2019}.

In ergodic MFGs the cost criterion is a long-term average cost. The seminal work \cite{LasryLions} by Lasry and Lions introduced the notion of stationary MFGs as limits of $n$-player games with ergodic costs. Some details were considered later in \cite{MR3127148}\footnote{When we mention the term {\it ergodic} we mean that the performance criteria are formulated as long-time average costs, the term {\it stationary} means that the system considered is independent of time.}. Since then, ergodic MFGs have been studied in various contexts, e.g. discrete-time, \cite{MR4064670}, \cite{MR3059616}, and diffusion processes \cite{MR3597374}. In \cite{car-por}, Cardaliaguet and Porretta studied the mean field system with discounted and ergodic cost criteria on the torus, formulating a weak solution to their ergodic master equation. We, on the other hand, work with the finite-state case, and take the analysis one step further and establish existence and uniqueness of a classical solution on the simplex, and moreover, show that the solution has Lipschitz partial derivatives. In a forthcoming paper, we will use this regularity to prove convergence of the $n$-player game to its mean field limit in the case of Markovian closed-loop controls. In \cite{Gomes2013}, Gomes, Mohr, and Souza introduced the {\it ergodic MFG system} in the finite state case:
\begin{align}\tag{MFG$_0$}
\label{erg_MFG}
\begin{cases}
-\bar\varrho + H(x,\Delta_x \bar u)+F(x,\bar\mu)=0,\\
\sum_{y\in[d]}\bar\mu_y\gamma^*_x(y,\Delta_y \bar u)=0.
\end{cases}
\end{align}
The unique solution is the triplet $(\bar \varrho,\bar  u, \bar\mu)$, where $\bar\varrho\in\R$ is the  value of the game, $\bar u\in\calM_0:=\{\eta\in\R^d:\sum_{x\in[d]}\eta_x=0\}$ is the potential vector, and $\bar\mu$ is the stationary distribution of the dynamics of the MFG.  The proof provided in \cite{Gomes2013} that ensures existence and uniqueness for the system \eqref{erg_MFG} relies on contractive assumptions that this work avoids. Therefore, we will provide an alternate proof of existence and uniqueness by going through a stationary discounted MFG system and applying a vanishing discount argument. Later as part of the main result, we will prove the connection between the ergodic MFG system \eqref{erg_MFG} and the ergodic master equation \eqref{ME}.


\subsection{Outline of the Paper and the Main Results}

After outlining the MFG models, we define mean field equilibria (MFE) for both the discounted and the ergodic problems via a best response map. Specifically, a MFE is a fixed point in the following sense. First, consider a single, representative player that optimally responds to a ``continuum of other players," formulated as a flow of distributions over the states. We have a MFE when the representative player's best response strategy causes the player's distribution law to agree with the flow of measures. We continue by setting up the MFG systems associated with each MFE, and prove that a MFE naturally characterizes the solution to the corresponding MFG system (and vice versa). The relationship between the MFG systems' solutions and MFE are specifically outlined in Propositions \ref{prop:statDisc_erg_mfg} and \ref{prop:MFG} and rely on probabilistic arguments.\footnote{Note that we distinguish between the notions of MFE and MFG system and establish the relationship between the two. Gomes et al.~\cite{Gomes2013} did not define the mean field equilibria in the probabilistic way as we do (neither in the finite horizon nor in the ergodic case). What they refer to as the mean field game equilibrium is actually the MFG system. In the stationary context, they focused on the trend of the finite horizon MFG system to the stationary MFG system. 
} Later, we show that the MFG systems are intimately related to their corresponding master equations in a sense we will make concrete in Proposition \ref{prop:MEr} and Theorem \ref{prop:ME}.

In order to provide existence and uniqueness results for the ergodic master equation, we first establish the regularity of the discounted master equations' solutions, uniformly in the discount factor. As briefly mentioned in the prior section, we expand on a technique developed in \cite{car-por} that takes advantage of the unique solution to the {\it stationary discounted MFG system} and introduce linearized systems of ordinary differential equations (ODEs) centered around that solution. We take the analysis further, improving our bounds on the solutions of the discounted MFG systems and using this to establish uniform regularity in the discounted master equation. Then by a vanishing discount argument, we obtain existence and regularity for the ergodic master equation. In the rest of this subsection we guide the reader through the stages of the proof.

In the finite horizon case, the authors of \cite{CardaliaguetDelarueLasryLions} constructed a solution to the master equation from the MFG system and established regularity of this solution by using a linearized system around the MFG system. Apparently, centering the linearized discounted systems around the discounted MFG systems is not enough in order to achieve regularity properties, uniformly in the discounted factor. The first to realize this were Cardlaiaguet and Porretta, and in \cite{car-por} they showed that the proper way to obtain uniform properties is to first build the discounted linearized systems around the {\it stationary discounted system}. This {\it stationary} discounted system takes the initial distribution to be stationary (in the discounted MFG) and as a result, the solution to this system is constant in time. We show in Proposition \ref{prop:statDisc_erg_mfg} that the stationary discounted MFG system is well-posed and converges to the ergodic MFG system \eqref{erg_MFG} as the discount factor goes to zero. We also prove the correspondence between the ergodic MFG system and ergodic MFE.

After the stationary systems are well understood, we introduce several linearized ODE systems. The linearized system approach provides an important avenue to assess the regularity of the solution to the discounted master equation. Specifically, we use several linearized MFG systems, showing first their well-posedness. Because the differences of discounted MFG solutions satisfy a particular linearized system, we are able to produce a bound on the difference between the solutions that is independent of the discount factor. This is one result in Proposition \ref{prop:MFG}. We then explicitly show that the unique solution to a particular linearized system (namely equation \eqref{lin_MFG_r_around_MFG^r}) is, in a sense, a derivative of the MFG system and that this solution is regular. Then, we can pass along this regularity to the discounted master equation solution. We set up the discounted master equations and establish their regularity, uniformly in the discount factor around zero. This is the content of Proposition \ref{prop:MEr}.

Once we establish regularity for the solution of the discounted master equation, we show that a family of solutions (indexed by the discount factor) is sequentially compact given the discount factor is small enough. Therefore, for any sequence of vanishing discounts, we can find a further subsequence which has a limit. The limit we find will be a solution to the ergodic master equation and will inherit the regularity of the solution to the discounted master equation. The value of the solution, namely  $\vr$, is unique and the potential function $U_0$ is unique up to an additive constant.  
Furthermore, the ergodic master equation and the ergodic MFG system are connected as follows. Let $(\bar\vr,\bar u,\bar\mu)$ the unique solution of $\eqref{erg_MFG}$, then $\vr=\bar\vr$ and $U_0(\cdot,\bar\mu) = \bar u_\cdot+k\vec{1}$, where $\vec{1}=(1,\ldots,1)\in\R^d$. These results are summarized in the main theorem of the paper, Theorem \ref{prop:ME}.

Throughout the paper, we borrow different techniques from the finite horizon case to establish our results, though there are several major differences in our approach that we now highlight. In particular, the so-called \textit{duality} of the MFG system plays a more central role in obtaining several estimates. That is, the coupled equations in our MFG systems are, in some sense, dual to one other. As we will see, treating the MFG value as a test function to integrate against, one can derive useful estimates for both the MFG system and the linearized MFG system.  Aside from the duality estimates, the analysis of finite horizon MFGs heavily relies on Gronwall's inequality, see e.g., \cite{CardaliaguetDelarueLasryLions, bay-coh2019, cec-pel2019}. Since we aim to obtain estimates that are uniform in the discount factor, we need to use other techniques. For this, we utilize machinery from the theory of Markov chains and the theory of differential equations. Namely, in Section \ref{sec:useful_lemmas} we introduce several lemmas which provide estimates on: (1) the measures from the MFG system when the initial conditions are allowed to vary and (2) the dynamics of the value from the MFG system. 


\subsection{Organization of the Paper} 

We conclude this section with an overview of some notation and basic definitions. The remainder of the paper is organized as follows. In Section \ref{sec2}, we describe the models of interest in detail, list the assumptions on the models, introduce the relevant MFG systems and master equations, and list our results. Section \ref{sec:useful_lemmas} begins our analysis by providing several preliminary lemmas that give estimates on solutions to the finite state analogs of a linear equation in divergence form and a viscous transport equation. Then, Section \ref{sec:stationary} provides the proof for Proposition \ref{prop:statDisc_erg_mfg} that deals with the stationary MFG systems. Next, Section \ref{sec:preparation} prepares the ground for the analysis of the difference between solutions to the discounted MFG system and for the results about the discounted master equations. It begins with an analysis of linearized systems which are relevant for the master equation analysis. Section \ref{sec:proof_prop_MFG} proves the results for the discounted MFG system; in particular, it recognizes that the differences of solutions to the discounted MFG system satisfy a linearized equation of the form studied in the prior section and is able to make use of the lemmas from that section. Next, Section \ref{sec:proof_MEr} makes use of all the prior work to prove Proposition \ref{prop:MEr}; this is the main result for the discounted master equation. Finally, Section \ref{sec:pf_thm} utilizes the uniform regularity established in the previous section to prove Theorem \ref{prop:ME}, which is the main result of the paper and deals with  the ergodic master equation.

\subsection{Notation and Basic Definitions}\label{sec11}
We denote the ray $[0,\infty)$ by $\R_+$ and the set of positive integers by $\N$. The parameter $d\in\N$ is fixed throughout and we denote the set $[d]:=\{1,2,\dots,d\}$. For $p=1,2$, we denote the $p$-norm on $\R^d$ by $|\cdot|_p$. Let $\1_A$ be the indicator function corresponding to the event $A$; that is, the function that is $1$ on $A$ and $0$ on its complement. The parameters $x,y$, and $z$ represent states and belong to the state-space $[d]$. Moreover, $\sum_{y,y\ne x}$ stands for $\sum_{y\in[d],y\ne x}$. Let
\begin{align}\notag
\calP([d])&:=\Big\{\eta\in\R^d_+:\sum_{x\in[d]}\eta_x = 1\Big\}.
\\\notag
\end{align} For any subset $\A\subseteq \R_+$, we define
$\calQ[\A]$ as the set of $d\times d$ transition matrices with rates in $\A$; that is, $\alpha \in \calQ[\A]$ if for every distinct $x,y\in [d]$, $\alpha_{xy}\in \A$ and $\sum_{z\in [d]} \alpha_{xz}=0$. As a matter of simplifying certain notationally intensive expressions, we introduce a few convenient notations. For any $x\in[d]$ and  $p\in\R^d$, set:
\begin{align}\notag
\Delta_x p:=(p_y-p_x)_{y\in[d]}. 
\end{align} Even more compactly,
\[
|\Deltadot p|:=\Big(\sum_{y\in[d]}|\Delta_y p|^2 \Big)^{1/2}.
\] Set $e_{yz}:=e_z-e_y$ where $(e_i:i\in[d])$ is the standard basis of $\R^d$. Let $\vec{1}\in\R^d$ be the vector of all ones. For any $b\in \R^d$, define 
\[
\langle b\rangle := \vec{1}\Big(\frac{1}{d}\sum_{x\in [d]} b_x\Big).
\] For a sufficiently differentiable function $K$ mapping $\R^d\ni p \mapsto K(p)\in\R$, we denote its gradient by $D_p K$ or $\nabla_p K$ and the Hessian matrix by $D_{pp}^2 K$. We also use derivatives on the simplex as follows. When $K:\calP([d]) \to \R$ is a measurable function, we say that $K$ is differentiable if for any distinct $y,z\in [d]$, there exists a function $D^\eta_{yz} K : \calP([d])\to\R$ such that: 
\[
D^\eta_{yz} K (\eta) = \lim_{h\to 0^+} \frac{K(\eta + he_{yz})-K(\eta)}{h}.
\] 
Denote $D^\eta K=(D^\eta_{yz} K)_{y,z\in [d]}$ and $D^\eta_y K := (D^\eta_{yz} K)_{z\in [d]}$. The set of continuously differentiable functions mapping $\R_+$ to some Banach space $X$, is denoted by $\calC^1(\R_+,X)$. Similarly, the set of all continuously differentiable functions $K:\calP([d])\to\R$ is denoted by $\calC^1(\calP([d]))$. For $m\in\calM_0 := \{m\in\R^d \mid \sum_{x\in [d]} m_x =0\}$, we define $\frac{\pl}{\pl m} K : \calP([d])\to\R^d$ and subsequently note that:
\begin{align}\label{simplex_witchcraft}
\Big[\frac{\pl}{\pl m } K(\eta)\Big]_y := m\cdot D^\eta_y K(\eta) = m\cdot D^\eta_1 K(\eta),
\end{align} where for a matrix $A$, we denote by $A_y$ its $y$-th row. For further discussion of derivatives on the simplex $\calP([d])$ and justification of these computations, we refer the reader to \cite[Section 2.1]{cec-pel2019}. 

Lastly, for $K: [d]\times \calP([d])\to\R$ we say that $K$ is {\it Lasry--Lions monotone} if for all $\eta,\hat\eta\in\calP([d])$,
\[
\sum_{x\in [d]} (K(x,\eta) - K(x,\hat\eta)) (\eta_x-\hat\eta_x) \geq 0.
\]

\section{The model and the main results}
\label{sec2}

In Section \ref{sec:mfg_descrip}, we describe the discounted and ergodic models of interest. The assumptions about the models are listed in Section \ref{sec:n_assumptions}. In Section \ref{sec:MFG_sys}, we introduce the corresponding discounted and ergodic MFG games. That is, we provide the MFG systems and  the relationship between the ergodic MFG system and the stationary discounted MFG systems and the mean field equilibria, described by a fixed point of a best response map. Then, we introduce the discounted and ergodic master equations in Section \ref{sec:MEs}. We provide there the connection to their respective MFGs and state the main results of the paper, Theorem \ref{prop:ME}, that deals with the well-posedness and regularity of the ergodic master equation. 

\subsection{Description of the MFGs}\label{sec:mfg_descrip} Throughout the paper, we have in the background a fixed filtered probability space $(\Omega, \FF,\PP)$. We consider two finite-state continuous-time MFGs: discounted and ergodic. The games share much of the same structure and differ only by their cost criteria. The state space is $[d]$, where $d$ is an integer greater than one. Informally, in an MFG a continuum of players aim to minimize their own individual cost. The initial distribution of the players along the states is $\mu_0\in\calP([d])$. Rigorously, equilibria in MFGs are defined via a fixed point of a best response mapping as we now illustrate. 

We consider a representative player that uses Markovian controls taking values in a set of rates $\A\subseteq \R_+$, also called the {\it action set}. A {\it Markovian control} is  a measurable function $\alpha:\R_+\times [d]\to \R^d$, such that for any $y\in[d]$, $y\ne x$, $\alpha_y(t,x)\in\A$ is the rate of transition at time $t$ to move from state $x$ to state $y$; and $\alpha_x(t,x)=-\sum_{y,y\ne x}\alpha_y(t,x)$. A Markovian control $\tilde\alpha$ that is independent of the time variable $t$ is called a {\it stationary Markovian control}. Denote the set of all Markovian controls (resp., stationary Markovian controls) by $\calA$ (resp., $\calA_s$). Occasionally, we refer to $\alpha\in\calA$ as a measurable mapping $\alpha:\R_+\to\calQ[\A]$, and to $\alpha\in\calA_s$ as an element of $\calQ[\A]$. 

The dynamics of the representative player $X=X^\al$ are given by:
\[
X_t = X_0 + \int_0^t \int_{\A^{d}} \sum_{y\in [d]} (y-X_{s^-})\boldsymbol{1}_{\{\xi_y\in(0,\al_y(s,X_{s^-}))\}} \calN(ds,d\xi), \qquad t\in\R_+,
\] where  
$\PP\circ(X_0)^{-1}$ is a fixed distribution from $\calP([d])$. For simplicity throughout the paper, we assume that $X_0=x_0\in[d]$ is fixed. Also, $\calN$ is a Poisson random measure with intensity measure $\nu$ given by
\begin{equation}\label{n_intensity_measure}
\nu(E) := \sum_{y\in [d]}\text{Leb}(E\cap \A^{d}_y),
\end{equation} with $\A^{d}_y := \{u\in\Xi \mid u_x = 0 \text{ for all }x\neq y\}$, and where $\text{Leb}$ is the Lebesgue measure on $\R$.  The Poisson random measure is the idiosyncratic noise.

To model the continuum of players, we consider a flow of measures, given by the deterministic and measurable function $\mu:\R_+\to\calP([d])$, satisfying the initial condition $\mu(0)=\PP\circ(X_0)^{-1}$. The representative player responds to the flow of measures $\mu$ by minimizing its private cost. In the discounted case, with discount factor $r>0$, the cost to the representative player using the control $\al$ and with the other players distributed at time $t$ according to $\mu(t)$ is:
\begin{align}\notag
J_r(\al,\mu):=\E\Big[\int_0^\iy e^{-rt}\left\{f(X_t, \al(t,X_t))+F(X_t,\mu(t))\right\}dt\Big],
\end{align} where $f$ and $F$ are real-valued measurable functions, such that $f(x,a)$ does not depend on $a_x$. A few times, we will consider the cost beginning at a particular time $t_0\in\R_+$ and where $X_{t_0}=x\in [d]$. In this case, we denote:
\[
J_r(t_0,x,\al,\mu) := \E\Big[\int_{t_0}^\iy e^{-rt}\left\{f(X_t, \al(t,X_t))+F(X_t,\mu(t))\right\}dt\Big| X_{t_0} = x \Big].
\]

\begin{definition}[Discounted Mean Field Equilibrium (MFE)]\label{def:disc_MFE} Let $(\tilde \al,\tilde\mu) \in \calA \times \calC^1(\R_+,\calP([d]))$ and define a jump process $(X^{\tilde \al}_t)_{t\geq 0}$ for the representative player on $[d]$. We say that $(\tilde \al,\tilde\mu)$ is a discounted MFE (starting at $x_0\in [d])$ if:
\begin{itemize}
    \item $X_0^{\tilde \al}=x_0$;
    \item $\PP\circ (X^{\tilde \al}_t)^{-1} = \tilde\mu_t$;
    \item $J_r(\tilde\al,\tilde\mu)=\argmin_{\al'\in\calA} J_r(\al',\tilde\mu)$.
\end{itemize}
\end{definition}

In the ergodic case, the cost is a long-time average:
\begin{align}\notag
J_0(\al,\mu):=\limsup_{T\to\iy}\frac{1}{T}\E\Big[\int_0^T \left\{f(X_t, \al(t,X_t))+F(X_t,\mu(t))\right\}dt\Big].
\end{align} 
As in \cite[Definition~2.1]{MR3597374}, we define a stationary ergodic MFE as a fixed point to an ergodic control problem in the following sense:
\begin{definition}[Stationary Ergodic MFE]\label{def:stat_erg_MFE} Let $(\tilde \al,\tilde\mu) \in \calA_s \times \calC^1(\R_+,\calP([d]))$ and define a jump process $(X^{\tilde \al}_t)_{t\geq 0}$ for the representative player on $[d]$. We say that $(\tilde \al,\tilde\mu)$ is a stationary ergodic MFE (starting at $x_0\in [d])$ if:
\begin{itemize}
    \item $X_0^{\tilde \al}=x_0$;
    \item $\PP\circ (X^{\tilde \al}_t)^{-1} = \tilde\mu_t$;
    \item $J_0(\tilde\al,\tilde\mu)=\argmin_{\al'\in\calA_s} J_0(\al',\tilde\mu)$.
\end{itemize}
\end{definition} We will demonstrate in Proposition \ref{prop:statDisc_erg_mfg} that the unique solution of the ergodic MFG system \eqref{erg_MFG} naturally characterizes a stationary ergodic MFE and vice versa.

\subsection{Assumptions}\label{sec:n_assumptions} 

For the MFGs, we introduce the Hamiltonian $H:[d]\times\R^d\to\R$, given by
\begin{align}\notag
H(x,p):= \inf_{a}\Big\{f(x,a)+a\cdot p\Big\},
\end{align} 
where the infimum is taken over {\it rate vectors} $a\in\R^d$, such that for any $y\ne x$, $a_y\in\A$ and $a_x=-\sum_{y,y\ne x}a_y$. 
We now make a few standard assumptions on the model. 
\begin{itemize}
    \item[($A_1$)] The action space is  $\A:=[\mathfrak{a}_l,\mathfrak{a}_u]$, for some $0<\mathfrak{a}_l < \mathfrak{a}_u<\iy$.
\end{itemize} 

\begin{itemize}
    \item[($A_2$)] There exists a unique optimal rate selector for the Hamiltonian. Namely, there is a measurable function $\gamma^*$ that maps any $(x,p)\in[d]\times\R^d$, to a unique rate vector $a$, satisfying  for any $y\ne x$, $a_y\in\A$ and $a_x=-\sum_{y,y\ne x}a_y$: 
\begin{align}\notag
 \gamma^*(x,p):=\argmin_{a}\Big\{f(x,a)+a\cdot p\Big\}.
\end{align} 
\end{itemize} We note that it is enough for $f$ to be strictly convex (in $\al$) to ensure that $\gamma^*$ has a unique minimizer. Moreover, a sufficient condition for $\gamma^*$ to be globally Lipschitz is that $f$ is uniformly convex. Whenever $H$ is differentiable, \begin{align}
    \label{gamma_H}
\gamma^*(x,p) = D_p H(x,p),
\end{align}
see e.g., \cite[Proposition 1]{Gomes2013}.

\begin{itemize}
    \item[($A_3$)] For any $x\in[d]$, the mean field cost $F(x,\cdot)\in\calC^1(\calP([d]))$ and $D^\eta F(x,\cdot)$ is Lipschitz.
\end{itemize} As a result, both $F(x,\cdot)$ and $D^\eta F(x,\cdot)$ are bounded since they are defined on compact sets. 
\begin{itemize}
    \item[($A_4$)] 
    The mean field cost is monotone in the Lasry--Lions sense:
    \begin{equation}
        \sum_{x\in [d]} (F(x,\eta) - F(x,\hat\eta))(\eta_x-\hat\eta_x) \geq 0,
    \end{equation}
\end{itemize} for all $\eta,\hat\eta\in\calP([d])$. Since $F(x,\cdot)$ is $\calC^1$ regular, and Lasry--Lions monotone, this is equivalent to having for all $m\in\calM_0$ and all $\eta\in\calP([d])$ that:
\begin{equation}\label{monotone_2}
    \sum_{x\in [d]} m_x D^\eta_1 F(x,\eta) \cdot m \geq 0.
\end{equation} Finally, we require that: 
\begin{itemize}
    \item[($A_5$)] The Hamiltonian $H$ is $\calC^2$ with respect to $p$. Moreover, on any compact set $[-K,K]$, we will assume $D_p H$ and $D^2_{pp}H$ are Lipschitz in $p$ and there exists $C_{2,H}>0$ such that:
\[
D^2_{pp} H(x,p) \leq -C_{2,H}.
\]
\end{itemize} We note that it is enough to assume these properties for compact sets since essentially, any $p$ argument of $H$ (and of $\gamma^*$) will be uniformly bounded (see, for instance, \cite[Remark 1]{cec-pel2019} and Lemma \ref{lem:sol_bound} in the sequel). 

Throughout, we will denote the Lipschitz constant of a given function $g$ as $C_{L,g}$ and a bound as $C_g$. 

For an example that satisfies the assumption above, the reader is referred to \cite[Example 1]{cec-pel2019}.

\subsection{The MFG Systems} \label{sec:MFG_sys} Below, we introduce two discounted MFG systems, which are essential to obtain the main result. We start with the discounted MFG system that is comprised of a Hamilton--Jacobi--Bellman equation, whose solution is the value of the game, and a Kolmogorov's equation. The latter describes the flow of the distribution of players for some fixed initial data $\mu_0$. 
Namely, fix an initial distribution of players $\mu_0=(\mu_{0,x}:x\in[d])\in\calP([d])$. The {\it discounted MFG system} is given by:
\begin{equation*}
\begin{cases}
-\frac{d}{dt}u^r_x(t) = -ru^r_x(t) + H(x, \Delta_x u^r(t)) + F(x,\mu^r(t)),\\
\frac{d}{dt}\mu^r_x(t) = \sum_{y\in [d]} \mu^r_y(t)\gamma^*_x(y,\Delta_yu^r(t)),\\
\mu^r_x(0) = \mu_{0,x},
\end{cases}
\tag{MFG$_r$}\label{MFG^r}
\end{equation*}
where $u^r:\R_+ \to \R^d$ and $\mu^r:\R_+ \to \calP([d])$. After establishing some results for a stationary version of this system, we will return to \eqref{MFG^r}.

The next proposition states that by omitting the initial condition on $\mu^r$, then for sufficiently small $r>0$, there is a unique stationary solution to \eqref{MFG^r} and write this system as
\begin{align}\label{stat_MFG_r}
\begin{cases}
 -r\bar u^r_x+H(x,\Delta_x\bar u^r)+F(x,\bar \mu^r)=0,\\
\sum_{y\in[d]}\bar\mu^r_y\gamma^*_x(y,\Delta_y\bar u^r)=0,
\end{cases}
\end{align} with constants  $\bar\mu^r\in\calP([d])$ and $\bar u^r \in\R^d$. Moreover, it relates the discounted MFG system \eqref{stat_MFG_r} to the ergodic MFG system \eqref{erg_MFG}, proves that the ergodic MFG system \eqref{erg_MFG} has a unique solution without depending on the additional contractive assumptions used in \cite{Gomes2013}, and elucidates the relationship between the stationary MFG solution and the stationary ergodic MFE.

\begin{proposition}[Stationary Systems]\label{prop:statDisc_erg_mfg}
There exists $r_0>0$, such that for any $r\in(0,r_0)$ the following holds:
\begin{enumerate}[(i)]
    \item There exists a unique solution $(\bar u^r,\bar \mu^r)\in (0,\iy)^d \times \calP([d])$ to \eqref{stat_MFG_r}. Additionally, there is a positive constant $C$, such that for any $r\in(0,r_0)$, $|r\bar u^r|\le C$ and for any $x\in[d]$, $\bar\mu^r_x\ge C^{-1}$.
    
    \item  There exists a unique solution $(\bar \varrho,\bar  u, \bar\mu)\in\R \times \calM_0\times\calP([d])$ to \eqref{erg_MFG}
    . Moreover, there exists $C>0$ such that for any $r\in(0,r_0)$, 
\begin{align}\notag
\max_{x\in[d]}|r\bar u^r_x-\bar\varrho|+\Big(\sum_{x\in[d]}|\Delta_x(\bar u^r-\bar u)|^2\Big)^{1/2}+|\bar\mu^r-\bar\mu|\le Cr^{1/2}.
\end{align}
Together with (i), $\bar\mu_x>0$ for any $x\in[d]$. 

\item From the solution of \eqref{erg_MFG} one can naturally construct a stationary ergodic MFE $(\bar\al, \mu)$ in the sense of Definition \ref{def:stat_erg_MFE}, where $\bar\al_{xy} := \gamma^*_y(x,\Delta_x \bar u)$ and $\mu$ is the flow of distributions of the jump process of the representative player $(X^{\bar\al}_t)_{t\geq 0}$ on $[d]$. Conversely, any stationary ergodic MFE $(\tilde \al, \tilde\mu)$ corresponds to the unique solution to the MFG system \eqref{erg_MFG},  $(\bar\varrho,\bar u,\bar\mu)$, in the sense that $\tilde\al$ is of the form $\tilde\al_{xy} = \gamma^*_y(x,\Delta_x \bar u)$, the flow of distributions satisfy $\tilde\mu(t) \to \bar\mu$ as $t\to\iy$, and $\bar\vr$ is the value of playing the equilibrium $(\tilde\al,\tilde\mu)$.
\end{enumerate}

\end{proposition}

The next proposition establishes the existence and uniqueness of a solution to \eqref{MFG^r}. Then, the proposition proves a strong decay to the stationary solution for a solution with any initial data. By using an analysis of the linearized systems, the uniform bound passes to any difference of discounted MFG solutions while allowing for initial data in the bound. These bounds hold uniformly in $r$ in a small neighborhood of $0$. Finally, the proposition connects the discounted MFG system \eqref{MFG^r} to the discounted MFE.

\begin{proposition}[Discounted MFG System]
\label{prop:MFG}
There exists $r_0>0$ such that for any fixed $r\in (0,r_0)$, the following holds:
\begin{enumerate}[(i)]
\item The system \eqref{MFG^r} admits a unique solution. 
\item There exist $\gamma,r_0>0$ and $C>0$, independent of $\mu_0$, such that for any $r\in(0,r_0)$ and $t\in\R_+$, 
\begin{align}\notag
|\mu^r(t)-\bar\mu^r|+|u^r(t)-\bar u^r|\le Ce^{-\gamma t}. 
\end{align} 
\item Let $(u^r,\mu^r)$ and $(\tilde u^r, \tilde \mu^r)$ be solutions to \eqref{MFG^r} where the initial data $\mu_0:=\mu^r(0)$ and $  \tilde \mu_0:=\tilde\mu^r (0)$ is allowed to differ. Then, there exist $C,\gamma >0$ independent of $\mu_0$, $\tilde\mu_0$, such that for all $r\in (0,r_0)$:
\begin{equation}\label{sol_ineq}
    |\mu^r (t) - \tilde\mu^r(t)|+ |(u^r(t) - \tilde u^r(t)| \leq Ce^{-\gamma t}|\mu_{0} - \tilde \mu_{0}|.
\end{equation} 
\item Any solution $(u^r, \mu^r)$ to the discounted MFG system \eqref{MFG^r} gives rise to a discounted MFE $(\al^r,\mu^r)$, where $\al^r_{xy}(t):=\gamma^*_y(x,\Delta_x u^r(t))$. Conversely, any discounted MFE $(\tilde\al^r,\tilde\mu^r)$ naturally corresponds to the solution $(\tilde u^r,\tilde\mu^r)$ to \eqref{MFG^r} where $\tilde u^r_x(t) := e^{rt} J_r(t,x,\tilde\al^r,\tilde\mu^r)$.
\end{enumerate}
\end{proposition}

In fact, we can establish the existence and uniqueness of a solution to \eqref{MFG^r} for any $r>0$ by using a fixed point theorem. However, we do not include this argument here since it is lengthy and not of much use to the focus of this paper. We also note that the proof of part (iv)---the relationship between \eqref{MFG^r} and the discounted MFE---holds for any $r>0$.

\subsection{The Master Equations}\label{sec:MEs}
Now, we are ready to describe the discounted and ergodic master equations and the connections to their respective MFGs. The {\it discounted master equation} is given by
\begin{align}
\tag{ME$_r$}\label{ME^r}
rU_r(x,\eta)= H(x,\Delta_x U_r(\cdot,\eta))+\sum_{y,z\in[d]}\eta_y D^\eta_{yz} U_r(x,\eta)\gamma^*_z(y,\Delta_y U_r(\cdot,\eta))+F(x,\eta),
\end{align} where $U_r:[d]\times\calP([d])\to\R$. The following proposition summarizes significant uniform-in-$r$ regularity properties of the solution of the master equations and connect them to their respective MFGs.
\begin{proposition}[Discounted Master Equation]\label{prop:MEr}
There exists $r_0>0$, such that for any $r\in(0,r_0)$ the following holds:

\begin{enumerate}[(i)]

    \item The master equation \eqref{ME^r} admits a solution.
    
    \item The solution $U_r$ is unique and satisfies $U_r(x,\mu_0) = u^r_x(0)$ where $(u^r,\mu^r)$ solves \eqref{MFG^r} with $\mu^r(0) = \mu_0$. In particular, $U_r(x,\bar\mu^r) = \bar u^r_x$ where $(\bar u^r, \bar\mu^r)$ solves the stationary discounted MFG system \eqref{stat_MFG_r}.
    
    \item $U_r$ is differentiable with respect to $\eta$ and there is a constant $C>0$, such that for any $\eta_1,\eta_2\in\calP([d])$, $r\in (0,r_0)$, and $x\in [d]$,
    \begin{align}\notag
|rU_r(x,\eta_1)|+|D^\eta_1 U_r(x,\eta_1)|\le C\quad\text{and}\quad |D^\eta_1 U_r(x,\eta_1)-D^\eta_1 U_r(x,\eta_2)|\le C|\eta_1-\eta_2|.
\end{align}

    \item $U_r$ is monotone in the Lasry--Lions sense.
\end{enumerate}
\end{proposition}

Property (iii) above is essential to show that $\{U_r\}_{r\in (0,r_0)}$ are sequentially compact and so enables us to obtain the existence and smoothness of the solution to the ergodic master equation. Proving this point is one of the main challenges in this work.

\begin{theorem}[Ergodic Master Equation]\label{prop:ME}
\begin{enumerate}[(i)]
    \item The master equation \eqref{ME} admits a solution $(\varrho,U_0)$.
    
    \item $U_0$ is differentiable with respect to $\eta$ and there is a constant $C>0$ such that for any $\eta_1,\eta_2\in \calP([d])$ and $x\in [d]$,
\[
|D^\eta_1 U_0(x,\eta_1) - D^\eta_1 U_0(x,\eta_2)| \leq C|\eta_1-\eta_2|.
\]

    \item $U_0$ is monotone in the Lasry--Lions sense.
    
    \item Let $(\bar\varrho,\bar u,\bar\mu)\in\R\times\calM_0\times\calP([d])$ be the unique solution to the ergodic MFG system \eqref{erg_MFG} and recall that $\vr\in\R$ is the value in the ergodic master equation. Then, $\bar\varrho = \varrho$ and $U_0(\cdot,\bar\mu) = \bar u_\cdot + k\vec{1}$, for some constant $k\in\R$. Moreover, $\varrho \in\R$ is unique and $U_0:[d]\times \calP([d])\to\R$ is unique up to an additive constant in the following sense:  if $(\hat\rho,\hat U_0)$ solves \eqref{ME}, then $\hat\rho=\rho$ and there is $k\in\R$, such that for any $\eta\in\calP([d])$, $\hat U_0(\cdot,\eta)=U_0(\cdot,\eta)+k\vec{1}$.
\end{enumerate}
\end{theorem}

\subsection{Proofs Outline}

Section \ref{sec:useful_lemmas} establishes several useful lemmas which are used throughout the paper and can be skipped on a first reading. Section \ref{sec:stationary} is dedicated to proving Proposition \ref{prop:statDisc_erg_mfg}. There we show that the stationary solution to \eqref{MFG^r} is close to the ergodic MFG system \eqref{erg_MFG} up to an order of $r^{1/2}$. Then in Section \ref{sec:preparation}, we work with several discounted linearized systems that allow us to obtain both uniform estimates on the difference between solutions of the discounted MFG system with different initial conditions and to obtain regularity of the discounted and ergodic master equations. We use these linearized systems in Section \ref{sec:proof_prop_MFG} to prove Proposition \ref{prop:MFG}. Sections \ref{sec:proof_MEr} and \ref{sec:pf_thm} are dedicated to the proofs of Proposition \ref{prop:MEr} and Theorem \ref{prop:ME}, respectively.

Finally, we wish to emphasize that we will work with small discount factors. That is, all the results in this paper will hold for $r\in (0,r_0)$ where $r_0$ is small enough and $r_0$ may change throughout the paper.

\section{Preliminary Lemmas}\label{sec:useful_lemmas}

We introduce four lemmas that deal with ODEs of the general structure we will encounter throughout the paper. The first one, Lemma \ref{lem:b}, demonstrates an elementary equivalence of expressions. The second one, Lemma \ref{lem:transition_matrix}, notes that for reasonably nice stochastic processes there is a uniform exponential decay to the stationary distribution. The third, Lemma \ref{lem:M0}, will be useful in analyzing the difference of solutions to Kolmogorov's  equations. The final one, Lemma \ref{lem:v}, gives a useful bound on solutions to a particular ODE. In particular, the third and fourth lemmas are used numerous times throughout the paper. After we establish the four initial lemmas, we introduce two additional lemmas that deal with the MFG system in particular. Finally, in this section we will list which of the five main assumptions are relevant for each result. In every result after this section, all five main assumptions are used and so they will not be listed.

\begin{lemma}\label{lem:b}
For any $b\in\R^d$,
\begin{align}\notag
    |\Delta b| = \sqrt{2d} |b-\langle b \rangle|.
\end{align}
\end{lemma}

\begin{proof}
\begin{align*}
    2d|b-\langle b\rangle|^2 &= 2d \sum_{x\in [d]} [b_x^2 + \langle b\rangle^2_x -2b_x \langle b\rangle_x] \\
    &= 2d \Big[\Big(\sum_{x\in [d]} b_x^2\Big) + \frac{1}{d} \Big( \sum_{y\in [d]} b_y\Big)^2 - \frac{2}{d} \Big( \sum_{y\in [d]} b_y\Big)^2 \Big] \\
    &= 2\Big[d \sum_{x\in [d]} b_x^2 - \Big( \sum_{x\in [d]} b_x\Big)^2\Big] \\
    &= \sum_{y\in [d]} \sum_{x\in [d]} (b_x - b_y)^2 \\
    &= |\Delta b|^2.
\end{align*}
\end{proof}

\begin{lemma}\label{lem:transition_matrix}
Assume ($A_1$) and fix $\delta>0$. Then there exist parameters $c=c(\delta), C=C(\delta)>0$ such that for any transition matrix $P$ of a discrete-time Markov chain on $[d]$ with inputs greater or equal to $\delta$, any initial state $x\in[d]$, and any $k\in\N$, one has
\begin{align}\notag
|P^k(x,\cdot)-\pi(\cdot)|\le  Ce^{-ck},
\end{align}
where $\pi$ is the stationary distribution of the irreducible recurrent chain $P$. As a corollary, for any two initial distributions $p_0$ and $\hat p_0$, and any $k\in\N$,
\begin{align}\notag
|(p_0-\hat p_0)^TP^k|\le 2 Ce^{-ck},
\end{align}
\end{lemma}
The proof follows since the Doeblin condition holds (see for instance \cite[p.192]{Doob}). 
Let $\calQ := \calQ[\A]$ be the set of $d\times d$ transition matrices with rates from $\A$. 
\begin{lemma}\label{lem:M0}
Assume ($A_1$) and fix measurable functions $h:\R_+\to \calQ$ and $B:\R_+\to \calM_0$ satisfying $\int_0^\iy |B(t)|^2dt<\iy$. Let $m:\R_+\to\calM_0$ be the solution to the ODE:
\begin{align}\label{eqn:measure_ode}
\frac{d}{dt}m_x(t)=\sum_{y\in[d]}m_y(t)h_{yx}(t)+B_x(t),\qquad t\in\R_+,
\end{align}
with $m(0)=m_0\in\calM_0$. Then, there exist constants $c,C>0$, depending only on $\mathfrak{a}_l$ and $\mathfrak{a}_u$, such that for any $t\in\R_+$,
\[
|m(t)| \leq Ce^{-ct}|m(0)| + Ce^{-ct}\int_0^t e^{cs} |B(s)|_1 ds\leq Ce^{-ct}|m(0)| + C\int_0^t |B(s)|_1 ds.
\]
\end{lemma}
\begin{proof}
First, note that \eqref{eqn:measure_ode} admits a unique solution. Moreover, by taking $\sum_{x\in [d]}$ on both sides and recalling that $B(t)\in\calM_0$ and that $h\in\calQ$ implies $\sum_{x\in [d]} h_{yx}(t) = 0$, we get $m(t)\in\calM_0$. We now turn to prove the bound. Define \[
P(t):= \exp\Big(\int_0^t h(s) ds
\Big) \quad \text{ and }\quad \tilde P(t) := \exp\Big(\int_0^t -h(s) ds
\Big).
\] Note that $\tilde P(t) P(t) = I$. Multiplying the ODE by the matrix $\tilde P(t)$ on the right, we get
\[
\frac{dm}{dt}(t) \tilde P(t) + m(t)(-h(t))\tilde P (t) = B(t) \tilde P(t).
\] Using that $\frac{d}{dt}\tilde P(t) = -h(t)\tilde P(t)$ and the product rule:
\[
\frac{d}{dt}\Big(m(t)\tilde P(t)\Big)= \frac{dm}{dt}(t) \tilde P(t) + m(t)\frac{d\tilde P}{dt}(t).
\] Integrating the ODE, we have 
\[
m(t)\tilde P(t) - m(0) = m(t)\tilde P(t) - m(0) \tilde P(0) = \int_0^t B(s) \tilde P(s) ds.
\] Multiplying the equation by $P(t)$, we obtain:
\[
m(t) - m(0)P(t) = \int_0^t B(s) \tilde P(s)P(t) ds = \int_0^t B(s) \exp\Big(\int_s^t h(v) dv\Big) ds.
\] Consider first the case $B\equiv 0$; so $m(t) = m(0)P(t)$. Set the vectors $m^+_0=(m_{0,i}\vee 0:i\in[d])$ and $ m^-_0=((-m_{0,i})\vee 0:i\in[d])$. Then, $m(0)=m^+_0-  m^-_0$. Note that $p^+_0:=m^+_0/|m^+_0|_1$ and $  p^-_0:=  m^-_0/|  m^-_0|_1$ belong to $\calP([d])$. In order to use Lemma \ref{lem:transition_matrix}, we need to make sure that the exponent of $\frac{1}{k}\int_0^k h(s) ds\in\calQ$ is bounded away from zero term by term (uniformly in $k$). We start with the observation that for a matrix $A$ with all entries bounded below by $\epsilon>0$, we have that $e^A = I+A+A^2/2+\cdots$ and so $e^A$ must have entries bounded below by $\epsilon>0$. Denote the matrix $H:=\frac{1}{k}\int_0^k h(s) ds\in\calQ$. Since $H\in\calQ$ and the rates are bounded, there exists $\theta >0$ such that $H+\theta I$ has entries bounded below by $\epsilon>0$. Then, \[
e^H = e^{(H+\theta I)-\theta I} = e^{H+\theta I} e^{-\theta I} = e^{H+\theta I} e^{-\theta }.
\] Since $e^{-\theta }>0$ and applying the previous observation to $H+\theta I$, we find that $e^H$ has entries bounded below by some $\delta>0$. From here on, $c$ and $C$ are positive constants which depend only on $\mathfrak{a}_l,\mathfrak{a}_u,$ and $d$, but may change from line to line. 

Note that $|m(0)|=|m^+_0-m^-_0|=|m^+_0|_1||p^+_0-p^-_0|$, where the last equality follows since $m_0\in\calM_0$ implies  $|m^+_0|=|m^-_0|$. 
So, Lemma \ref{lem:transition_matrix} together with the observation that for any $k\in\N$, $\frac{1}{k}\int_0^k h(s) ds\in\calQ$, yields,
\[
|m(k)|\leq |m^+_0|_1\Big|(p^+_0-p^-_0)\exp\Big(\Big[\frac{1}{k}\int_0^k h(s) ds\Big]k\Big)\Big| \leq Ce^{-ck}|m^+_0|_1.
\] To obtain such a bound for general $t\in\R_+$, note that
\begin{align*}
|m(t)| &\leq |m^+_0|_1\Big|(p^+_0-p^-_0)\exp\Big(\Big[\frac{1}{\lfloor t\rfloor}\int_0^{\lfloor t\rfloor} h(s) ds\Big] \lfloor t\rfloor \Big)\Big| \exp\Big(\int_{\lfloor t\rfloor}^{t} |h(s)| ds \Big) \\
&\leq Ce^{-c\lfloor t\rfloor}|m^+_0|_1 \\
&\le Ce^{-ct}|m^+_0|_1.
\end{align*}
Now, $|m^+_0|_1\leq |m(0)|_1$ and use the equivalence of $L_1$ and $L_2$ norms in $\R^d$ to conclude the case $B\equiv 0$. For general $B$, define 
$$
\beta(s,t):= \sup_{\substack{|v|_1 = 1\\ v\cdot \vec{1}=0}}|v^T \tilde P(s)P(t)|_1 = \frac{1}{2}\max_{x,y\in [d]} |(e_x-e_y)^T \tilde P(s)P(t)|_1.
$$ For any $x,y\in [d]$, $e_x,e_y\in\calP([d])$; so, Lemma \ref{lem:transition_matrix} and a similar computation as with $m$ above implies $\beta(s,t) \leq Ce^{-c(t-s)}$. Since $B(s) \in\calM_0$, we find that $B(s)\cdot \vec{1} = 0$. Therefore,
\begin{align*}
\int_0^t |B(s)^T \tilde P(s)P(t)|_1 ds &=\int_0^t |B(s)|_1 \Big|\frac{B(s)^T }{|B(s) |_1}\tilde P(s)P(t)\Big|_1 ds \\
&\le C\int_0^t |B(s)|_1 \beta(s,t) ds \\
&\le C \int_0^t |B(s)|_1 e^{-c(t-s)} ds.
\end{align*} 
\end{proof}

\begin{lemma}\label{lem:v} 
Assume ($A_1$) and fix $T>0$ and measurable functions $h:[0,T]\to \calQ$ and $A :[0,T]\to \R^d$ satisfying $\int_0^T |A(t)|^2dt<\iy$. Let $v:[0,T]\to\R_+^d$ satisfy
\begin{align}\notag
-\frac{d}{dt}v_x(t) = h_{x}(t)\cdot\Delta_x v(t)+A_x(t),\qquad t\in[0,T],
\end{align}
with a given terminal condition $v(T)=v_T\in\R^d$; we use the notation $h_x:=(h_{xy}:y\in[d])$, $x\in[d]$. Then, there exist constants $c,C>0$, depending only on $\mathfrak{a}_l$ and $\mathfrak{a}_u$, such that for any $t\in[0,T]$,
\begin{align}\label{general_m}
|\Delta v(t)|\le Ce^{-c(T-t)}|\Delta v_T|+C\int_t^Te^{-c(s-t)}|A(s)|ds.
\end{align} As a consequence, if we have $w:[0,T]\to\R_+$ which satisfies
\begin{align}\notag
-\frac{d}{dt}w_x(t)= -rw_x(t) + h_{x}(t)\cdot\Delta_x w(t)+A_x(t),\qquad (t,x)\in[0,T]\times[d],
\end{align} then there exist $c,C>0$ depending only on $\mathfrak{a}_l$ and $\mathfrak{a}_u$,  such that for any $t\in [0,T]$, 
\[
e^{-rt}|\Delta w(t)|\le Ce^{-rt}e^{-c(T-t)}|\Delta w_T|+C\int_t^T e^{-c(s-t)}e^{-rs}|A(s)|ds.
\]
\end{lemma}
\begin{proof}
By Lemma \ref{lem:b}, it suffices to establish the lemma with $|v(t) - \langle v(t)\rangle|$ instead of $|\Delta v(t)|$ and we proceed in this way. Fix arbitrary $t_0\in[0,T]$ and $\eta\in\R^d$ and let $m:[t_0,T]\to\calM_0$ be the solution to the ODE
\begin{align}\label{eq:m_no_B}
\frac{d}{dt}m_x(t)=\sum_{y\in[d]}m_y(t)h_{yx}(t),\qquad t\in[t_0,T],
\end{align}
with $m(t_0)=\eta-\langle \eta\rangle $. Then, $(d/dt)(v(t)\cdot m(t))= -A(t)\cdot m(t)$, which implies
\begin{align}\notag
v(T)\cdot m(T)-v(t_0)\cdot m(t_0)= -\int_{t_0}^TA(s)\cdot m(s)ds.
\end{align}
For any $v\in \R^d$, denote $\bar v:=v-\langle v\rangle$. Note that for any $a,b\in\R^d$, $a\cdot\langle b\rangle=b\cdot \langle a\rangle$, so we get 
$$v(t_0)\cdot (\eta-\langle \eta\rangle)=\bar v(t_0)\cdot \eta.$$ Moreover, since $m(T)\in\calM_0$, $v(T)\cdot m(T)=\bar v(T)\cdot m(T)$. As a consequence,
\begin{align}\label{eq:bar_v}
\bar v(t_0)\cdot \eta= \bar v(T)\cdot m(T)+\int_{t_0}^TA(s)\cdot m(s)ds.
\end{align}
From Lemma \ref{lem:M0} applied on the interval $[t_0,T]$ to \eqref{eq:m_no_B}, we have 
$$|m(t)|\le C|m(t_0)| e^{-c(t-t_0)}= C|\eta-\langle\eta\rangle| e^{-c(t-t_0)}\le 2C|\eta| e^{-c(t-t_0)},$$ 
for some $c,C>0$, depending only on $\mathfrak{a}_l$ and $\mathfrak{a}_u$. Plugging in the last bound in \eqref{eq:bar_v}, we get
\begin{align}\notag
|\bar v(t_0)\cdot \eta|
&\le 
|m(T)||\bar v(T)|+\int_{t_0}^T|A(s)| |m(s)|ds\\\notag
&\le 
C|\eta|\Big(e^{-c(T-t_0)}|\bar v(T)|+\int_{t_0}^Te^{-c(s-t_0)}|A(s)| ds\Big).
\end{align}
Since $\eta\in\R^d$ is arbitrary, the result \eqref{general_m} follows. 

For $w$ which satisfies
\[
-\frac{d}{dt}w_x(t)=  -rw_x(t) + h_{x}(t)\cdot\Delta_x w(t)+A_x(t),\qquad (t,x)\in[0,T]\times[d],
\] we note that $v(t) = e^{-rt}w(t)$ satisfies 
\[
-\frac{d}{dt}v_x(t)= h_{x}(t)\cdot\Delta_x v(t)+e^{-rt}A_x(t),\qquad (t,x)\in[0,T]\times[d].
\] So we can apply the previous result \eqref{general_m} to obtain the corollary for $w$.
\end{proof}

The proofs of Propositions \ref{prop:statDisc_erg_mfg} and \ref{prop:MFG} rely on two more auxiliary lemmas that are specific to the MFG system. The first is a preliminary bound on the solutions to \eqref{MFG^r} and the second is a value-measure duality lemma, which provides a useful way to bound integrals of the solutions. Throughout the proof, $\mu_0\in\calP([d])$ is fixed. 

The finite-horizon discounted MFG system we will use in our approach of Proposition \ref{prop:MFG} is:
\begin{equation*}
\begin{cases}
-\frac{d}{dt}u^{T}_x(t) = -ru^{T}_x(t) + H(x, \Delta_x u^{T}(t)) + F(x,\mu^T(t)),\\
\frac{d}{dt}\mu^T_x(t) = \sum_{y\in [d]} \mu^T_y(t)\gamma^*_x(y,\Delta_yu^T(t)),\\
\mu^T_x(0) = \mu_{0,x}, \quad u^T(T) = 0,\quad t\in [0,T].
\end{cases}
\tag{MFG$_r^T$}\label{MFG^r_T}
\end{equation*} 
Under the regularity assumptions on $H,F,$ and $\gamma^*$, we have by the same techniques as in \cite{bay-coh2019} that there exists a unique solution to \eqref{MFG^r_T}.

Recall that the Hamiltonian $H$ is assumed to be locally Lipschitz. The following lemma shows that solutions to the above ODEs are bounded, uniformly in $T$, and so $H$ only ever takes inputs from a compact set. We use this bound in several places throughout the paper.

\begin{lemma}[Solution Bound]\label{lem:sol_bound}
Assume ($A_2$), ($A_3$) and fix an arbitrary $\theta:\R_+\to\calP([d])$ satisfying $\theta(0)=\mu_0$. Assume that $u$ and $u^T$ solve the following ODEs:
\begin{align}\label{ODE^r_1}
-\frac{d}{dt}u_x(t) &= -ru_x(t) + H(x, \Delta_x u(t)) + F(x,\theta(t)),\qquad t\in\R_+, \\
-\frac{d}{dt}u_x^T(t) &= -ru_x^T(t) + H(x, \Delta_x u^T(t)) + F(x,\theta(t)), \quad t\in [0,T].\label{ODE^T_1}
\end{align}
Then, 
\begin{align}\label{bound_u}
&\max_{(t,y)\in [0,T]\times [d]} |u^T_y(t)|\le \frac{C_{f+F}}{r} \quad\text{ and }\quad \sup_{(t,y)\in\R_+\times[d]}|u_y(t)| \le \frac{C_{f+F}}{r}.
\end{align} 
\end{lemma}

\begin{proof}
We begin with the argument for $u(t)$. Fix $t\in\R_+$ and let $(X_s)_{s\ge t}$ be a jump process on $[d]$, starting at $X_t=y$, and with jump rates $\al_{xz}(s):=\gamma^*_z(x,\Delta_{x}u(s))$. 
Applying It\^o's lemma to $e^{-rt}u_{X_t}(t)$ and taking expectations on both sides, one gets,
\begin{align}\notag
e^{-rt}u_y(t) = \E\Big[\int_t^\infty e^{-rs}\sum_{x\in[d]}\one_{\{X_{s}=x\}}\Big( - \frac{du_x}{dt}(s) +ru_x(s)-  \alpha_x(s)\cdot\Delta_xu(s) \Big)ds\Big],
\end{align} 
where $\alpha_x:=(\alpha_{xy}:y\in[d])$. Since $u$ is a solution to \eqref{ODE^r_1} and $\gamma^*$ minimizes $H$, it follows that:
\[
-\frac{du_x}{dt}(s)+ru_x(s)-\alpha_x(s) \cdot \Delta_x u(s)= f(x,\alpha_x(s)) + F(x,\theta(s)),\qquad (s,x)\in[t,\infty)\times[d].
\] 
As a result,
\begin{align*}
e^{-rt}u_y(t) = \E\Big[\int_t^\infty e^{-rs}\sum_{x\in[d]}\one_{\{X_{s}=x\}}\left( f(x,\alpha_x(s))+F(x,\theta(s)) \right)ds\Big]. \numberthis{\label{eqn:u_true_form}}
\end{align*} Using the bound $C_{f+F}$ of $f+F$, we obtain \eqref{bound_u} for $u$. 

To establish $\eqref{bound_u}$ for $u^T$, we follow the same steps as with $u$ except on the interval $[0,T]$. Since $u^T(T)=0$, the only difference is that $\infty$ in the integral is replaced with $T$. 
\end{proof}

Eventually we prove that the solutions of the finite horizon system  converge to the unique solution to \eqref{MFG^r} as the time horizon $T\to\iy$; to show this, we require the following. 

\begin{lemma}[Value and measure duality]\label{lem:val-meas-duality}
    Assume ($A_2$), ($A_4$), and ($A_5$). Let $(u^r,\mu^r)$ and $(\tilde u^r, \tilde \mu^r)$ be solutions to \eqref{MFG^r} or $\eqref{MFG^r_T}$ where the initial data $\mu_0^r$ and $\tilde \mu_0^r$ is allowed to differ. Fix $0\leq t_1<t_2\leq T$. In the case of solutions to \eqref{MFG^r}, we allow $T=\iy$. Then:
    \begin{align*}
 C_{2,H} \int_{t_1}^{t_2} e^{-rt}\sum_{x\in [d]}  |\Delta_x (u^r -\tilde u^r)(t)|^2 (\mu_x^r(t)+\tilde\mu_x^r(t))dt &\leq  -e^{-rt}(\mu^r(t)-\tilde\mu^r(t))\cdot (u^r(t)-\tilde u^r(t))\Big|_{t_1}^{t_2}.
\end{align*} 
\end{lemma}

\begin{proof}
For $v:=u-\tilde u$ and $m:=\mu-\tilde\mu$, an application of the chain rule implies that
\begin{align*}
    &e^{-rt_2}\sum_{x\in [d]} m_x(t_2) v_x(t_2)-e^{-rt_1}\sum_{x\in [d]} m_x(t_1) v_x(t_1) \\
    &\quad= \int_{t_1}^{t_2} \frac{d}{dt}\Big[ e^{-rt} \sum_{x\in [d]} m_x(t) v_x(t)\Big] dt\\
    &\quad= \int_{t_1}^{t_2} e^{-rt}\sum_{x\in [d]}\Big[ -r m_x(t) v_x(t) + \frac{dm_x}{dt}(t)v_x(t) + \frac{dv_x}{dt}(t)m_x(t)\Big]dt\\
    &\quad= \int_{t_1}^{t_2} e^{-rt}\sum_{x\in [d]}\Big[(H(x,\Delta_x\tilde u(t))- H(x,\Delta_x u(t)) + F(x,\tilde \mu(t))-F(x,\mu(t)))m_x(t) \\
    &\qquad\qquad\qquad\qquad+ v_x(t)\sum_{y\in [d]} \phi_{yx}(t) \Big]dt,
\end{align*} where $\phi_{yx}(t) =  \mu_y(t)\gamma^*_x(y,\Delta_yu(t)) -\tilde \mu_y(t)\gamma^*_x(y,\Delta_y\tilde u(t)).$ 
 Rearranging this expression, using the Lasry--Lions monotonicity condition for $F$, and the fact that $$\sum_x\gamma^*_x(y,\Delta_y u(t))=\sum_x\gamma^*_x(y,\Delta_y \tilde u(t))=0,$$ we get,
\begin{align}\label{ineq:m0}
\begin{split}
    0 &\leq \sum_{x\in[d]}\int_{t_1}^{t_2} e^{-rt} (F(x,\mu(t))-F(x,\tilde\mu (t)))m_x(t)dt \\
    &= \sum_{x\in[d]}\int_{t_1}^{t_2} e^{-rt} \Big[ v_x(t) \sum_{y\in [d]} \phi_{yx}(t) + (H(x,\Delta_x\tilde u(t))- H(x,\Delta_x u(t)))m_x(t)\Big]dt\\
    &\qquad-e^{-rt}(\mu^r(t)-\tilde\mu^r(t))\cdot (u^r(t)-\tilde u^r(t))\Big|_{t_1}^{t_2}\\
    &= \int_{t_1}^{t_2} e^{-rt} 
    \sum_{x\in[d]}
    [ H(x,\Delta_x \tilde u(t))- H(x,\Delta_x u(t))+\Delta_x v\cdot\gamma^*(x,\Delta_x u(t))]\mu_x(t)dt\\
    &\quad
    +\int_{t_1}^{t_2} e^{-rt} 
    \sum_{x\in[d]}
    [ H(x,\Delta_x u(t))- H(x,\Delta_x\tilde u(t))-\Delta_x v\cdot\gamma^*(x,\Delta_x \tilde u(t))]\tilde\mu_x(t)dt\\
    &\qquad-e^{-rt}(\mu^r(t)-\tilde\mu^r(t))\cdot (u^r(t)-\tilde u^r(t))\Big|_{t_1}^{t_2}.
\end{split}
\end{align}
Since $H$ is differentiable and $\gamma^*$ is Lipschitz, \eqref{gamma_H} says $\gamma^*(x,p) = D_p H(x,p)$. 
Moreover, since we assumed $D^2_{pp}H(x,p)$ is bounded away from $0$, we obtain the inequalities:
\begin{align}\notag
&H(x,\Delta_x \tilde u(t))- H(x,\Delta_x u(t))+\Delta_x v\cdot\gamma^*(x,\Delta_x u(t))
 \leq -C_{2,H} |\Delta_x v|^2,
\\\notag
&H(x,\Delta_x  u(t))- H(x,\Delta_x \tilde u(t))-\Delta_x v\cdot\gamma^*(x,\Delta_x\tilde  u(t))
 \leq -C_{2,H} |\Delta_x v|^2.
\end{align} 
Combining these two observations and then moving the two integrals in \eqref{ineq:m0} to the left-hand side yields the desired inequality.   
\end{proof}

We end this section with a comment about how the assumptions ($A_1$)--($A_5$) and the above lemmas allow us to obtain the  exponential bounds in Proposition \ref{prop:MFG}. We also draw parallels between the finite state MFG given here and the MFG with diffusion dynamics on the torus, see e.g., \cite{CardaliaguetDelarueLasryLions,car-por}.

\begin{remark}
\begin{itemize}
    \item Parallels between our setup and the diffusion case on the torus: We work with a Markov chain on a finite state space. This setup shares some features to the MFG in the diffusion case on the torus, which is compact. Had we allowed for a discrete but infinite state space, we would need to assume further conditions that lead to stability, positive recurrence, and proper growth conditions on the cost terms. As for the control component, assumption ($A_1$) asserts that the rates are in a compact set that is bounded away from zero. In a way, this condition is equivalent to assuming that the diffusion coefficient gets values in a compact set that is bounded away from zero. 

\item Exponential bounds: Assumption ($A_1$) implies the geometric ergodicity in Lemma \ref{lem:transition_matrix}. Lemma \ref{lem:M0} builds on this geometric ergodicity and improves it. To see this, consider two flows of measures $\mu$ and $\tilde\mu$, associated with the time-dependent rate matrices  $Q(t)$ and $\tilde Q(t)$ and the initial conditions $\mu_0$ and $\tilde\mu_0$, respectively. Then, $m:=\mu-\tilde\mu$ satisfies \eqref{eqn:measure_ode} with $h=Q$ and $B(t)=\tilde\mu^T(t)(Q-\tilde Q)(t)$. One case that shows up in several setting in the sequel is when the rate matrices $Q$ and $\tilde Q$ are functions of some value functions $u$ and $\tilde u$, respectively, that are coupled with $\mu$ and $\tilde\mu$ through MFG systems, see e.g., Sections \ref{sec:41} and \ref{sec:proof_prop_MFG}. Then, we get an exponential bound for the difference between $\mu$ and $\tilde\mu$ plus an error that emerges from the difference between the value functions. Lemma \ref{lem:v} provides a complimentary bound for the difference between the values by considering $v=u-\tilde u$ and by pulling off similar tricks to the ones performed in Lemma \ref{lem:val-meas-duality}, which in turn relies on the monotonicity of $F$ ($A_4$) and the concavity of $H$ ($A_5$), together with the fact that the gradient of the Hamiltonian is the optimal selector rate matrix \eqref{gamma_H}. These ideas go beyond analyzing  differences between solutions to MFG systems and are useful in the analysis of linearized systems of the MFG systems, see e.g., Sections  
\ref{sec:prop223} and \ref{sec:proof_MEr}. For these systems we need also the regularity of $F$ ($A_3$).
\end{itemize}
\end{remark}

\section{Proof of Proposition \ref{prop:statDisc_erg_mfg}}
\label{sec:stationary}

In this section, we provide a relationship between two stationary MFG systems: the discounted MFG system \eqref{stat_MFG_r} in which the initial distribution of the players is the stationary distribution, and the ergodic MFG system \eqref{erg_MFG}. We then provide existence and uniqueness results for \eqref{erg_MFG}. Finally, we prove the relationship between the stationary ergodic MFG solution and stationary ergodic MFE.

\subsection{Proof of (i)}\label{sec:41}
We start the proof with the existence of a solution to \eqref{stat_MFG_r}. Then, we show that for any such solution $\bar\mu^r$ is bounded away from 0, uniformly in $r$. The latter property is utilized to establish uniqueness.

{\bf Existence of a solution to \eqref{stat_MFG_r}:} For any fixed $\mu\in \calP([d])$ we define a map $\calS^\mu : [0,C_{f+F}/r]^d \to [0,C_{f+F}/r]^d$ to be:
\[
\calS^\mu (u)(x) := \inf_{a \in \A^d_x} \Big[\Big(\sum_{y,y\neq x} a_y + r\Big)^{-1}\Big( \sum_{y,y\neq x} a_yu_y + f(x,a) + F(x,\mu)\Big)\Big],
\] 
where recall that $\mathbb{A}^d_x:= \{ a\in\R^d : \forall y\neq x,\quad a_y\in [\fra_l, \fra_u], \quad a_x = -\sum_{y,y\neq x} a_y\}$. 
Then, 
\begin{align*}
    |S^\mu (u-\tilde u)(x)| &\leq \sup_{a\in\A^d_x} \left[ \Big(\sum_{y\neq x} a_y + r\Big)^{-1} \sum_{y\neq x} a_y |u_y - \tilde u_y|\right] \\
    &\leq \sup_{a\in\A^d_x} \left[ \frac{\sum_{y\neq x} a_y}{\sum_{y\neq x} a_y + r} \|u - \tilde u\|_{\sup}\right]\\
    &\leq \frac{(d-1)\mathfrak{a}_u}{(d-1)\mathfrak{a}_u +r}\|u-\tilde u\|_{\sup},
\end{align*} where $\|\cdot\|_{\sup}$ is the sup-norm on $\R^d$. Taking the supremum of the left hand side over $x\in [d]$, $\calS^\mu$ is a contraction with respect to the sup-norm. By the Banach fixed point theorem, there exists a unique fixed point $u^\mu$ of $S^\mu$ and we note that this fixed point satisfies:
\[
ru^\mu_x =  H(x,\Delta_x u^\mu) + F(x,\mu).
\] Therefore, the map $\calT_1 : \calP([d])\to [0,C_{f+F}/r]^d$ which sends $\mu\mapsto u^\mu$ is well-defined. Suppose $u$ and $\tilde u$ are the fixed points of $\calS^\mu$ and $\calS^{\tilde \mu}$, respectively. We show that $\calT_1$ is Lipschitz. First, 
\begin{align*}
    |\calT_1(\mu) - \calT_1(\tilde\mu)| = |(u - \tilde u)(x)| &\leq |S^\mu (u-\tilde u)(x)| + |(S^\mu - S^{\tilde \mu})(\tilde u)(x)| \\
    &\leq \frac{(d-1)\mathfrak{a}_u}{(d-1)\mathfrak{a}_u +r} \|u - \tilde u\|_{\sup} + |(S^\mu - S^{\tilde \mu})(\tilde u)(x)|.
\end{align*} By Lipschitz continuity of $F$ and boundedness of the controls, the last term can be bounded as follows:
\[
|(S^\mu - S^{\tilde \mu})(\tilde u)(x)| \leq \sup_{a\in \A^d_x} \left| \frac{F(x,\mu) - F(x,\tilde \mu)}{\sum_{y\neq x}a_y +r }\right| \leq \frac{C_{L,F}}{(d-1)\mathfrak{a}_l + r}\|\mu - \tilde \mu\|_{\sup}.
\] Combining this with the previous computation proves that $\calT_1$ is Lipschitz. Let the function $\calT_2 : [0,C_{f+F}/r]^d \to \calP([d])$ be defined by a composition of two functions. First, one which maps $[0,C_{f+F}/r]^d \ni (u_x)_{x\in [d]}\mapsto (\gamma^*_y(x,\Delta_x u))_{x,y\in [d]}\in\calQ$ and second which maps a fixed transition matrix $\beta\in\calQ$ to its corresponding stationary distribution $\bar{\mu}_\beta\in\calP([d])$. Note that the second function is indeed well-defined since the rates are bounded away from zero on a finite state space and thus the corresponding jump process has a unique stationary distribution. Since $\gamma^*$ is Lipschitz and the stationary distribution depends continuously on the transition matrix of the Markov chain, $\calT_2$ is continuous. Thus, the map $\calT_1 \circ \calT_2 : [0,C_{f+F}/r]^d\to [0,C_{f+F}/r]^d$ has a fixed point by Brouwer. That is,
\[
 -r\bar u^r_x+H(x,\Delta_x\bar u^r)+F(x,\bar \mu^r)=0,
\] where $\bar \mu^r$ is the stationary distribution associated to $\gamma^*(x,\Delta_x \bar u^r)$ and so simultaneously:
\[
\sum_{y\in[d]}\bar\mu^r_y\gamma^*_x(y,\Delta_y\bar u^r)=0.
\]

{\bf Showing bound for $\bar \mu^r$, independent of $r$:}  Let $x\in [d]$ be chosen so that $\bar\mu_x^r$ is minimal. Using that $\bar \mu^r$ is a stationary solution, and the definition of $\gamma^*$:
\begin{align*}
    \bar\mu_x^r &= \frac{-\sum_{y,y\neq x} \bar\mu_y^r \gamma_x^*(y,\Delta_y \bar u^r)}{\gamma_x^*(x,\Delta_x \bar u^r)}\\
    &= \sum_{y,y\neq x} \bar\mu_y^r \frac{\gamma_x^*(y,\Delta_y \bar u^r)}{\sum_{z,z\neq x} \gamma_x^*(z,\Delta_z \bar u^r)}\\
    &\geq \sum_{y,y\neq x} \bar\mu_y^r \frac{\mathfrak{a}_l}{(d-1) \mathfrak{a}_u},
\end{align*} where we now denote $A:=\mathfrak{a}_l/((d-1) \mathfrak{a}_u)$. Recall that $\bar\mu_x + \sum_{y,y\neq x} \bar\mu_y = 1$. Combining this observation with the previous inequality and the fact that $\bar\mu_x^r\geq 0$ yields:
\begin{align*}
    \bar\mu_x^r \geq \frac{A}{1+A}>0.
\end{align*} We note that $x$ was chosen minimally and that the bound is independent of $r$.

{\bf Uniqueness of a solution to \eqref{stat_MFG_r}:} Let $(\bar u^r, \bar \mu^r)$ and $(\bar v^r, \bar\theta^r)$ solve the stationary discounted MFG \eqref{stat_MFG_r}. We claim the following duality,
\begin{align*}
    C_{2,H} \sum_{x\in [d]} |\Delta_x (\bar u^r - \bar v^r)|^2 (\bar \mu_x^r + \bar \theta_x^r)\leq r (\bar u^r - \bar v^r)\cdot(\bar \mu^r - \bar \theta^r). \numberthis{\label{eqn:stat_mfgr_duality}}
\end{align*} To see this, we just follow the same lines as in the proof of Lemma \ref{lem:val-meas-duality}; that is, we replace $(\tilde u^r, \tilde\mu^r)$ with $(\bar v^r, \bar \theta^r)$ and $(u^r,\mu^r)$ with $(\bar u^r,\bar \mu^r)$. Then, we repeat the same computation and in this case the derivative quantities are replaced by just adding zero.

Next, we claim that there exists $C>0$ independent of $r$ such that:
    \begin{align}\label{cor:T2Lipschitz}
|\bar\mu^r-\bar\theta^r|
\le 
C\Big(\sum_{x\in[d]}|\Delta_x(\bar u^r-\bar v^r)|^2\Big)^{1/2}.
\end{align} To see this, use Lemma \ref{lem:M0} with $m(t)\equiv \bar\mu^r - \bar\theta^r$ and \eqref{stat_MFG_r} to obtain $c,C>0$ independent of $r$ that may change from one line to the next such that:
\begin{align*}
    |\bar \mu^r - \bar \theta^r| &\leq Ce^{-ct} |\bar \mu^r - \bar \theta^r| + C \Big|\Big(\sum_{y\in [d]} \bar{\theta}^r_y\left[ \gamma^*_x(y,\Delta_y \bar u ^r) - \gamma^*_x(y,\Delta_y \bar v ^r)\right]\Big)_{x\in [d]} \Big| \\
    &\leq Ce^{-ct} |\bar \mu^r - \bar \theta^r| + C \sqrt{\sum_{x\in [d]}\Big(\sum_{y\in [d]} \bar\theta^r_y\left| \gamma^*_x(y,\Delta_y \bar u ^r) - \gamma^*_x(y,\Delta_y \bar v ^r)\right|\Big)^2}\\
    &\leq Ce^{-ct} |\bar \mu^r - \bar \theta^r| + C \sqrt{\Big(\sum_{y\in [d]} \bar\theta^r_y |\Delta_y(\bar u^r - \bar v^r)| \Big)^2} \\
    &\leq Ce^{-ct} |\bar \mu^r - \bar \theta^r| + C\sum_{y\in [d]} \bar\theta^r_y |\Delta_y(\bar u^r - \bar v^r)| \\
    &\leq Ce^{-ct} |\bar \mu^r - \bar \theta^r| + C|\bar\theta^r| \Big(\sum_{y\in [d]}  |\Delta_y(\bar u^r - \bar v^r)|^2\Big)^{1/2},
\end{align*} where we used the Lipschitz continuity of $\gamma^*$ and Cauchy--Schwarz. We have that $|\bar\theta^r|$ is bounded by $1$ and so choosing $t>0$ large enough proves \eqref{cor:T2Lipschitz}.

We now establish uniqueness by a duality argument. With the reasoning below, there are $C_1,C_2,C_3>0$, independent of $r$, such that,
\begin{align*}
C_1^{-1}\sum_{x\in[d]}|\Delta_x(\bar u^r-\bar v^r)|^2
&\le 
C_2^{-1}\sum_{x\in[d]}|\Delta_x(\bar u^r-\bar v^r)|^2(\bar\mu^r_x + \bar \theta^r_x)
\\
&\le 
 r \sum_{x\in [d]}(\bar u^r_x-\bar v^r_x)(\bar\mu^r_x-\bar\theta^r_x)
\\
&= r \Big[\sum_{x\in [d]} (\bar u^r_x-\bar v^r_x)(\bar\mu^r_x-\bar\theta^r_x) - \sum_{x\in [d]} (\langle \bar u^r \rangle - \langle\bar v^r\rangle)(\bar\mu^r_x-\bar\theta^r_x)\Big]
\\
&\le r|\bar \mu^r - \bar \theta^r||(\bar u^r - \bar v^r)-(\langle \bar u^r\rangle -\langle \bar v^r\rangle)|
\\
&\le 
rC_3\Big(\sum_{x\in[d]}|\Delta_x(\bar u^r-\bar v^r)|^2\Big)^{1/2}|\bar\mu^r-\bar\theta^r|.
\end{align*} The first inequality follows since the stationary distributions are bounded away from zero; the second by \eqref{eqn:stat_mfgr_duality}; the third line follows since $\bar \mu^r - \bar\theta^r \in \calM_0$; the fourth is Cauchy--Schwarz after combining the sums; and the final line is Lemma \ref{lem:b}. Combining it with \eqref{cor:T2Lipschitz} implies that for sufficiently small $r$, $(\bar u^r,\bar\mu^r) = (\bar v^r, \bar\theta^r)$. 

\subsection{Proof of (ii)} 
We start with establishing the existence of a solution to \eqref{erg_MFG}. Then, we prove the convergence of the solution of \eqref{stat_MFG_r} to the solution of 
\eqref{erg_MFG} as $r\to 0$, which ultimately, implies uniqueness to \eqref{erg_MFG}.

{\bf Existence of a solution to \eqref{erg_MFG}:} Fix $y_1,y_2\in [d]$. Existence follows from a vanishing discount argument applied to $(r\bar u^r_{y_1}, \bar u^r_{y_1} - \bar u^r_{y_2}, \bar\mu^r)$, which we establish in several steps. First, we note that $\{\bar\mu^r\}_r$ belong to the compact set $\calP([d])$. By \eqref{bound_u}, we know that $|r \bar u_{y}^r| \leq C_{f+F}$ for any $y\in [d]$. So, we turn to showing that $|\bar u^r_{y_1} - \bar u^r_{y_2}|$ is uniformly bounded. Let $(X_t)_{t\geq 0}$ be a jump process on $[d]$, starting at $y_1 \in [d]$, with rates $\bar\al_{xy}^r:=\gamma^*_y(x,\Delta_x \bar u^r)$, $\bar\al_{x}^r := (\bar\al^r_{xy})_{y\in [d]}$. As with the derivation of \eqref{bound_u}, we can apply It\^o's lemma to $e^{-rt} \bar u^r_{X_t}$, use \eqref{stat_MFG_r}, and take expectations on both sides to obtain:
\begin{align}\label{eqn:equiv_u_def}
    \bar u_{y_1}^r &= \EE_{y_1}\Big[ \int_0^\iy e^{-rt} \sum_{x\in [d]} \1_{\{X_{t} = x\}} \big(f(x,\bar\al^r) + F(x,\bar\mu^r)\big) dt \Big] = J(0,y_1,\bar\al^r,\bar\mu^r).
\end{align} So in order to find a uniform bound for $\bar u^r_{y_1} - \bar u^r_{y_2}$, we can deal with the equivalent discounted cost functional. Let $\tau_{y_2}$ be the stopping time:
\[
\tau_{y_2}:=\inf\{t\geq 0 \mid X_t = y_2\}.
\] Notice that:
\begin{align*}
    J_r(0,y_1,\bar\al^r,\bar\mu^r) &= \EE_{y_1}\Big[ \int_0^{\tau_{y_2}} e^{-rt} \big(f(X_{t},\bar\al^r) + F(X_{t},\bar\mu^r)\big) dt  + J_r(0,X_{\tau_{y_2}},\bar\al^r,\bar\mu^r) \\
    &\qquad - (1-e^{-r\tau_{y_2}})J_r(0,X_{\tau_{y_2}},\bar\al^r,\bar\mu^r) \Big]. \numberthis{\label{eqn:equiv_u_def_2}}
\end{align*} By boundedness of $f$ and $F$ and since $e^{-rt}\leq 1$, the integral on the right hand side is bounded above by $\tau_{y_2}C_{f+F}$. Moreover, the numerical inequality $( 1-e^{-rm})/r \leq m$ implies that $r\tau_{y_2} |\bar u_{y_2}^r|$ is a bound for the final term on the right hand side. All together, \eqref{eqn:equiv_u_def_2} and these observations imply:
\begin{align*}
    |\bar u^r_{y_1} - \bar u^r_{y_2}| = |J_r(0,y_1,\bar\al^r,\bar\mu^r)  - J_r(0,y_2,\bar\al^r,\bar\mu^r)| &\leq \EE_{y_1}[\tau_{y_2}](C_{f+F} + r |\bar u^r_{y_2}|).
\end{align*} Since the jump rates of $X_t$ are in $\A$, a compact set bounded away from zero, $\EE_{y_1}[\tau_{y_2}]\leq C$ for some positive constant depending only on $\mathfrak{a}_l$ and $\mathfrak{a}_u$. Using the uniform bound for $r\bar u^r$, we see that $\bar u^r_{y_1} - \bar u^r_{y_2}$ is bounded uniformly in $r$ too. 

Therefore, for any sequence of vanishing discounts $(r_n)_{n\in\N}$, $r_n\to 0$, we obtain a subsequence, which by abuse of notation we also denote by $(r_n)_n$, for which $(r\bar u^r_{y_1}, \bar u^r - \langle\bar u^r\rangle,\bar \mu^r)$ converges. Call this limit $(\bar\vr_{y_1}, \bar u, \bar\mu)\in\R\times\calM_0\times\calP([d])$. Moreover, for any $x\in [d]$:
\begin{align*}
    \lim_{n\to\iy} [\Delta_{x} \bar u^{r_n}]_{y_1} &= \lim_{n\to\iy}[\bar u^{r_n}_{y_1} - \bar u^{r_n}_x] \\
    &= \lim_{n\to\iy}[\bar u^{r_n}_{y_1} -\langle\bar u^r\rangle + \langle\bar u^r\rangle - \bar u^{r_n}_x] \\
    &= \bar u_{y_1} - \bar u_{x} \\
    &= [\Delta_x \bar u]_{y_1}.
\end{align*} With this observation, we note that the limit solves \eqref{erg_MFG} if we can remove the dependence on $y_1\in [d]$ for $\bar\vr_{y_1}$. To this end, set $\bar\al_{xy} :=  \gamma^*_y(x,\Delta_x \bar u)$ and note that by the continuity of $f$, $\gamma^*$, and $F$, combined with the Tauberian theorem (see, e.g. \cite[Proposition~A.5]{MR2554588}), we have:
\[
r_n J_{r_n} (0,y_1,\bar\al^{r_n},\bar\mu^{r_n} ) \to \limsup_{T\to\iy} \frac{1}{T} \EE_{y_1} \int_0^T (f(X_t,\bar\al) + F(X_t,\bar\mu)) dt = J_0(\bar\al,\bar\mu),
\] 
So $\bar\vr_{y_1} = J_0(\bar\al,\bar\mu)$, where the latter is independent of $y_1\in [d]$ as desired.

{\bf Convergence of stationary discounted to ergodic:} To establish the convergence to \eqref{erg_MFG}, we argue similarly as in the proof of uniqueness for \eqref{stat_MFG_r}. We claim that
\begin{align*}
    C_{2,H} \sum_{x\in [d]} |\Delta_x (\bar u^r-\bar u)|^2 (\bar\mu^r_x + \bar \mu_x) \leq (r\bar u^r - \bar\varrho) \cdot (\bar\mu^r - \bar\mu). \numberthis{\label{eqn:stat_mfgr_erg_duality}}
\end{align*} To see this, the computation is mostly similar to that of \eqref{eqn:stat_mfgr_duality} (that is, Lemma \ref{lem:val-meas-duality}), but we make note of the subtle differences:
\begin{align*}
    \sum_{x\in [d]} (r\bar u_x^r - \bar\varrho) (\bar\mu_x^r - \bar\mu_x) &= \sum_{x\in [d]} \Big(H(x,\Delta_x \bar u^r) - H(x,\Delta_x \bar u) +F(x,\bar\mu^r) - F(x, \bar\mu) \Big)(\bar\mu_x^r - \bar\mu_x) \\
    &\geq \sum_{x\in [d]} \Big(H(x,\Delta_x \bar u^r) - H(x,\Delta_x \bar u) \Big)(\bar\mu_x^r - \bar\mu_x) \\
    &=  \sum_{x\in [d]} (H(x,\Delta_x \bar u^r) - H(x,\Delta_x \bar u) +\Delta_x(\bar u^r - \bar u)\cdot \gamma^*(x,\Delta_x\bar u^r) )\bar\mu_x^r \\
    &\quad\qquad + (H(x,\Delta_x \bar u) - H(x,\Delta_x \bar u^r) +\Delta_x(\bar u - \bar u^r)\cdot \gamma^*(x,\Delta_x\bar u) )\bar\mu_x \\
    &\geq \sum_{x\in [d]} C_{2,H}|\Delta_x (\bar u^r - \bar u)|^2(\bar\mu^r_x + \bar\mu_x),
\end{align*} where in each line (resp.) we used the value equations from \eqref{stat_MFG_r} and \eqref{erg_MFG}, Lasry--Lions monotonicity, the measure equations from \eqref{stat_MFG_r} and \eqref{erg_MFG} and  the fact that $\sum_x\gamma^*_x(y,\Delta_y \bar u^r)=\sum_x\gamma^*_x(y,\Delta_y \bar u)=0$, and finally the bound on $D^2_{pp}H(x,p)$. Moreover, we claim that if $(\bar \varrho,\bar u,\bar \mu)\in\R\times \calM_0 \times \calP([d])$ solves the ergodic MFG system \eqref{erg_MFG}, then
\begin{align*}
    |\bar\mu^r - \bar\mu| \leq C\Big(\sum_{x\in[d]}|\Delta_x(\bar u^r-\bar u)|^2\Big)^{1/2}.
\end{align*} The proof follows the same lines as that of \eqref{cor:T2Lipschitz}. Finally, there exist $C_1,C_2,\hat C >0$ such that:
\begin{align*}
C^{-1}_1\sum_{x\in[d]}|\Delta_x(\bar u^r-\bar u)|^2
&\le 
C^{-1}_2\sum_{x\in[d]}|\Delta_x(\bar u^r-\bar u)|^2(\bar\mu^r_x+\bar \mu_x)
\\
&\leq \sum_{x\in [d]}(r\bar u_x^r -\bar \varrho)(\bar\mu_x^r- \bar\mu_x) \\
&= \sum_{x\in [d]} (r\bar u^r_x -\langle r \bar  u^r \rangle )(\bar\mu^r_x - \bar\mu_x) \\
&\leq r\hat C \Big(\sum_{x\in [d]} |\Delta_x \bar u^r|^2 \Big)^{1/2},
\end{align*} where we used the uniform bound of $\bar\mu^r$ away from $0$, \eqref{eqn:stat_mfgr_erg_duality}, the fact that $\sum_{x\in [d]} k (\bar\mu_x^r - \bar\mu_x)  = 0$ for any constant $k$ since $\bar\mu^r - \bar\mu \in \calM_0$, and then Lemma \ref{lem:b}.

{\bf Uniqueness of a solution to \eqref{erg_MFG}:} Uniqueness follows immediately by the previous section. That is, let $(\bar \vr,\bar u,\bar\mu)$, $(\tilde\vr,\tilde u, \tilde \mu) \in \R\times\calM_0\times \calP([d])$ be two solutions to \eqref{erg_MFG}. Fix $x\in [d]$ arbitrarily and by the bound in Proposition \ref{prop:statDisc_erg_mfg} (ii), we have:
\begin{align*}
    |\bar\vr - \tilde\vr| + |\Delta_x (\bar u - \tilde u)| + |\bar\mu - \tilde\mu| &\leq |r\bar u^r_x - \bar\vr| + |r\bar u^r_x - \tilde\vr| + |\Delta_x (\bar u^r - \bar u)| + |\Delta_x (\bar u^r - \tilde u)| \\
    &\qquad + |\bar\mu^r - \bar\mu| + |\bar\mu^r - \tilde\mu| \\
    &\leq Cr^{1/2}.
\end{align*} Since the left hand side is independent of $r$, we can pass the limit $r\to 0$ to get uniqueness.

\subsection{Proof of (iii)} We now show that the solution to the ergodic MFG system \eqref{erg_MFG} characterizes a stationary ergodic MFE in the sense of Definition \ref{def:stat_erg_MFE}. First, let $(\bar\vr,\bar u,\bar\mu)\in \R\times \calM_0\times \calP([d])$ be the solution to the ergodic MFG system \eqref{erg_MFG}. We claim that  $(\bar\al,\mu)$ is a stationary ergodic MFE, where $\bar\al_{xy} := \gamma^*_y(x,\Delta_x \bar u)$ and $\mu(t) := \PP\circ (X^{\bar \al}_t)^{-1}$.

We start with showing that $J_0(\bar \al,\mu) = J_0(\bar\al,\bar\mu)$, where in the latter expression we abuse notation and consider the flow of constant distributions $\bar\mu$. Let $\eps>0$ be arbitrary. Since $\bar\mu$ satisfies: 
\[\sum_{x\in [d]} \bar\mu_x \gamma^*_y(x,\Delta_x \bar u)=0,\] we note that $\bar\mu$ is the stationary distribution of the exponentially ergodic Markov chain $X^{\bar\al}_t$. And since $\mu(t)$ is the law of $X_t^{\bar \al}$, $\mu(t)\to\bar\mu$ as $t\to\iy$, and we can choose $T_0>0$ so that for all $t>T_0$, we have $|\mu(t) - \bar\mu| < \eps$. Using the definition of the cost functional, taking the absolute value inside, and using Lipschitz continuity of $F$,
\begin{align*}
    |J_0(\bar\al,\mu) - J_0(\bar\al, \bar\mu)| &\le \Big|\limsup_{T\to\iy} \frac{1}{T} \EE \int_0^T \big( F(X^{\bar\al}_t, \mu(t)) - F(X^{\bar\al}_t,\bar\mu) \big) dt \Big| \\
    &= \Big|\limsup_{T\to\iy} \frac{1}{T} \EE \int_{T_0}^T \big( F(X^{\bar\al}_t, \mu(t)) - F(X^{\bar\al}_t,\bar\mu) \big) dt \Big| \\
    &\leq \eps C_{L,F} .  
\end{align*} Since $\eps>0$ was arbitrary, we have the desired equality.

We now show that $\bar\vr = J_0(\bar\al,\bar\mu)=J_0(\bar\al,\mu)$ and that with any other strategy there is inequality; the latter will follow since the Hamiltonian in \eqref{erg_MFG} is a minimizer and hence, $\bar\al = \argmin_{\al\in\calA} J_0(\al,\mu)$. We proceed by applying It\^o's lemma to $\bar u_{X_t^{\bar\al}}$, and taking expectations 
\begin{align*}
    \EE[\bar u_{X_T^{\bar\al}}] - \bar u_{X_0^{\bar\al}} = \EE\int_0^T \sum_{x\in [d]} \1_{\{X^{\bar\al}_{t} = x\}} \big\{\bar\al_x \cdot \Delta_x \bar u\big\} dt.
\end{align*} 
Divide by $T$, pass the limit as $T\to\iy$, and then apply \eqref{erg_MFG}, to get:
\begin{align*}
    0 &= \limsup_{T\to\iy}\frac{1}{T}\EE \int_0^T \sum_{x\in [d]} \1_{\{X^{\bar\al}_{t} = x\}} \big\{\bar\al_x \cdot \Delta_x \bar u\big\} dt \\
    &= \limsup_{T\to\iy}\frac{1}{T}\EE \int_0^T \sum_{x\in [d]} \1_{\{X^{\bar\al}_{t} = x\}} \big\{-f(x,\bar\al) - F(x,\bar\mu) + \bar\vr \big\} dt. \\
\end{align*} Note that $\bar\vr$ does not depend on $x\in [d]$ or $t\in\R_+$ and so we can bring it to the left hand side. Negating both sides and recalling the definition of the ergodic cost,
\begin{align*}
    \bar\vr = J_0 (\bar\al, \bar\mu).
\end{align*} We note that we have equality because $\bar\al$ is the unique minimizer prescribed by $\gamma^*$ and moreover, for all other strategies $\al$, we have inequality. Therefore,
\begin{align*}
    J_0(\bar\al, \bar\mu) = \bar\vr \leq J_0(\al,\bar\mu),
\end{align*} for all $\al\in\calA_s$ and so
\begin{align}\label{MFE:forward_1}
    \bar\al = \argmin_{\al \in \calA_s} J_0(\al,\bar\mu).
\end{align} Together with the definition of $\mu(t)$, \eqref{MFE:forward_1} implies that $(\bar\al,\mu)$ is a stationary ergodic MFE.

We now prove the converse---that is, a stationary ergodic MFE naturally characterizes the stationary ergodic MFG solution. Let $(\tilde \al, \tilde\mu) \in \calA_s \times \calC^1(\R_+,\calP([d]))$ be a stationary ergodic MFE. For simplicity, think about $\tilde\alpha$ as a rate matrix in $\calQ[\A]$. 
Denote the limiting distribution of $X^{\tilde\al}_t$ by $\tilde\mu^\iy$ and set the vector $\ell := (f(x,\tilde\al) + F(x,\tilde \mu^\iy))_{x\in [d]}$. Correspondingly, we denote the limiting (as $t\to\iy$) transition matrix of $e^{\tilde\al t}$ by $[\tilde \mu^\iy]$. Define the vector of potentials as:
\[
\tilde u := \check u-\langle \check u\rangle\in\calM_0,\qquad\text{where}\qquad \check u:=\int_0^\iy (e^{\tilde\al t} - [\tilde\mu^{\iy}]) \ell dt.
\] Note that since $X^{\tilde\al}_t$ is an exponentially ergodic Markov process and since $f$ and $F$ are bounded, the integral is finite and hence $\tilde u$ is well-defined. Since $\tilde\alpha\in\calQ[\A]$, we have $\tilde\alpha\langle \check u\rangle=0$. Moreover, every column of $[\tilde \mu^\iy]$ is the stationary distribution of the rate matrix $\tilde\al$, so we can multiply the potential vector on the left by $\tilde\al$ and compute:
\begin{align*} 
    \tilde\al \tilde u &= \int_0^{\iy} (\tilde\al e^{\tilde \al t} - \tilde \al [\tilde \mu^\iy]) \ell dt \\ \notag
    &= \int_0^{\iy} \tilde\al e^{\tilde \al t} \ell dt \\ \notag
    &= \lim_{T\to\iy} \int_0^T \Big[\frac{d}{dt} e^{\tilde \al t}\Big] \ell dt \\ 
    &= \lim_{T\to\iy} (e^{\tilde \al T} - I) \ell \\ 
    &= ([\tilde\mu^{\iy}] - I) \ell.  \numberthis\label{eqn:bias_stationary}
\end{align*} Define the MFG value $\tilde \vr$ as $\tilde \vr := J_0(\tilde \al, \tilde\mu)$. Since $\tilde\mu(t) \to\tilde\mu^\iy$, we have $\tilde\vr = J_0(\tilde \al, \tilde\mu^\iy)$. Recall that $\tilde\mu^\infty$ is the stationary limiting distribution of the process $X_{t}^{\tilde\al}$. Summarizing the above and using ergodicity,
\begin{equation}\label{eqn:varrho}
    \tilde \vr := J_0(\tilde \al, \tilde\mu) = J_0(\tilde\al,\tilde\mu^\iy) =  \sum_{y\in [d]} \tilde\mu_y^\iy (f(y,\tilde\al) + F(y,\tilde\mu^\iy)).
\end{equation}
This implies that $\tilde \vr\vec{1}=[\tilde\mu^\iy] \ell$. Making use of this observation, the fact that the columns of a rate matrix sum to the zero vector, and looking at the $x$-component of \eqref{eqn:bias_stationary}, we get,
\begin{align*}
    \tilde \al_x \cdot \Delta_x \tilde u =\sum_{y\in [d]} \tilde \al_{xy} \tilde u_y = \tilde\vr - f(x,\tilde\al) - F(x,\tilde\mu^\iy).
\end{align*} 

We now turn to showing that $\tilde\al_{z}= \gamma^*(z,\Delta_x\tilde u)$. Arguing by contradiction, assume that there exists $z\in[d]$ such that $\tilde\al_{z}\ne \gamma^*(z,\Delta_x\tilde u)$. Then, from the above and by the uniqueness of the minimizer $\gamma^*$ for the Hamiltonian $H$, there exists $\eps>0$, such that, 
\[
\tilde \vr -\eps\geq f(z,\gamma^*(z,\Delta_{z}\tilde u)) + \gamma^*(z,\Delta_{z}\tilde u)\cdot \Delta_{z} \tilde u + F(z,\tilde\mu^\iy),
\]
and for any $x\in [d]$,
\[
\tilde \vr \geq f(x,\gamma^*(x,\Delta_{x}\tilde u)) + \gamma^*(x,\Delta_{x}\tilde u)\cdot \Delta_{x} \tilde u + F(x,\tilde\mu^\iy).
\]
Let $X^\gamma_t$ be the jump process with rates given by $\gamma^*_y(x,\Delta_x \tilde u)$ and with starting state $X_0=x_0\in[d]$. 
Using It\^o's lemma on $\tilde u_{X^\gamma_t}$, taking expectations, and dividing by $T$,
\begin{align*}
    \frac{\EE[\tilde u_{X^\gamma_T}] - \tilde u_{X^\gamma_0}}{T} 
    &= \frac{1}{T} \EE\int_0^T \big(\gamma^*(X^{\gamma}_t,\Delta_{X^{\gamma}_t}\tilde u))\cdot \Delta_{X^{\gamma}_t} \tilde u \big) dt\\ 
    &= \frac{1}{T} \EE\int_0^T\sum_{x,x\ne z}\one_{\{X^{\gamma}_{t}=x\}} \big(\gamma^*(x,\Delta_{x}\tilde u))\cdot \Delta_{x} \tilde u \big) dt\\\numberthis\label{eqn:MFE_Ito1}
    &\quad+ \frac{1}{T} \EE\int_0^T\one_{\{X^{\gamma}_{t}=z\}} \big(\gamma^*(z,\Delta_{z}\tilde u)\cdot \Delta_{z} \tilde u \big) dt\\
    &\le \tilde \varrho-\eps\frac{1}{T} \int_0^T\EE[\one_{\{X^\gamma_{t}=z\}}]dt\\
    &\quad -\frac{1}{T} \EE\int_0^T\left[f(X^\gamma_t,\gamma^*(X^{\gamma}_{t},\Delta_{X^{\gamma}_{t}}\tilde u))  + F(X^{\gamma}_{t},\tilde\mu^\iy)\right]dt. 
\end{align*}
Since the rates of the chain $(X^\gamma_t)_{t\geq 0}$ are in the compact set $\A$ which is bounded away from zero, it follows that the stationary distribution is strictly positive. Hence, 
$$
\lim_{T\to\iy}\frac{1}{T} \int_0^T\EE[\one_{\{X^\gamma_{t}=z\}}]dt>0.
$$
Taking $\limsup_{T\to\iy}$ on both sides of \eqref{eqn:MFE_Ito1}, we obtain that 
$$
J_0(\gamma^*(\cdot,\Delta_\cdot\tilde u),\tilde\mu^\iy)<\tilde \varrho.
$$
Combining it with \eqref{eqn:varrho}, and the argument proceeding it, we get,
$$
J_0(\gamma^*(\cdot,\Delta_\cdot\tilde u),\tilde\mu)<J_0(\tilde\al,\tilde\mu),
$$
which contradicts the optimality of $\tilde\al$. As a result, for any $x,y\in[d]$, $\tilde\al_{xy}= \gamma^*_y(x,\Delta_x\tilde u)$.
That is, $(\tilde\vr,\tilde u,\tilde\mu^\iy)$ satisfies the value equation of \eqref{erg_MFG}, so it remains to show that the corresponding Kolmogorov equation simultaneously holds. 

Since $\tilde\mu^\iy$ is the stationary distribution associated to $\tilde\al$, this means
\[
\sum_{y\in [d]} \gamma^*_x(y,\Delta_y \tilde u) \tilde\mu_y^\iy =0. 
\] Therefore, $(\tilde \vr, \tilde u, \tilde \mu^\iy)$ satisfies the ergodic MFG system \eqref{erg_MFG} and by uniqueness of the solution of \eqref{erg_MFG}, $\tilde \vr = \bar\vr$, $\tilde\mu^\iy = \bar\mu$, and $\tilde u$ is equal to $\bar u$ up to a constant. So, we are able to recover the stationary ergodic MFG solution from the stationary ergodic MFE.
\qed

\section{Preliminary Results: Linearized System around the Stationary MFG System}\label{sec:preparation}

In what follows, we present several intermediary ODE systems and results in order to estimate the difference between solutions to the discounted MFG system with different initial conditions in Section \ref{sec:proof_prop_MFG}. We also use these linearized systems in order to approach the discounted master equation \eqref{ME^r} and establish regularity properties of its solution $U_r$ uniformly in $r$ around zero in Section \ref{sec:proof_MEr}. By a vanishing discount argument, we will prove results for the ergodic master equation \eqref{ME} and its solution $U_0$. In order to accomplish this, we require $U_r$ is sufficiently regular, uniformly in $r\in (0,r_0)$ for small enough $r_0>0$. To this end, we introduce and analyze certain linearized systems. Since we define $U_r$ through $u^r$, $U_r$ inherits some regularity through the linearized systems that we define for the discounted problems.

Some of the results in this subsection are adaptations of Cardaliaguet--Porretta \cite[Section 3]{car-por} to the finite-state case. To make the paper self-contained, and since we consider a finite-state Markov chain model rather than diffusions on a torus, we provide all the details. This section culminates in Lemma \ref{lem:lin_MFG_r_exp_A_B}, a key lemma that allows introduction of a strong decay into bounds on the linearized system's solutions.

Given the solution $(\bar u^r, \bar\mu^r)$ to \eqref{stat_MFG_r}, $m_0\in\calM_0$, and measurable functions $A:\R_+\to\R^d$ and $B:\R_+\to\calM_0$, consider the linearized system around $(\bar u^r,\bar \mu^r)$, defined for $t\in\R_+$ by:
\begin{align}
\label{lin_MFG_r_around_stat}
\begin{cases}
-\frac{d}{dt}v_x(t)=-rv_x(t) + \gamma^*(x,\Delta_x \bar u^r)\cdot \Delta_x v(t) + D^\eta_1 F(x,\bar\mu^r)\cdot m(t)+A_x(t),\\
\frac{d}{dt}m_x(t)=\sum_{y\in[d]} m_y(t)\gamma^*_x(y,\Delta_y \bar u^r) + \sum_{y\in[d]}\bar \mu^r_y \nabla_p\gamma^*_x(y,\Delta_y \bar u^r)\cdot \Delta_y v(t)+B_x(t),\\
m(0)=m_{0},\qquad\text{$v$ is bounded}.
\end{cases}
\end{align}

\noindent

The existence and uniqueness of this system on a finite time horizon with terminal condition for $v$ follows by the same arguments given in \cite{bay-coh2019}. The proof in the infinite horizon case relies on estimates that are driven by the following linearized duality lemma. This lemma is also useful in the sequel.
\begin{lemma}[Linearized duality]\label{lem:duality}
Assume that $(v,m)$ satisfies \eqref{lin_MFG_r_around_stat} on the interval $[0,T]$ with $A\equiv 0$. Then, there are constants $C_1,C_2>0$ independent of $t,T,m_0$ and $r$, such that for any $0\leq t_1 < t_2\leq T$,
\begin{align}\notag
\int_{t_1}^{t_2}e^{-rs}|\Deltadot v(s)|^2ds
&\le 
- C_1\Big[e^{-rs}m(s)\cdot v(s)\Big]^{t_2}_{t_1}
+C_1\int_{t_1}^{t_2}e^{-rs}|B(s)|^2ds
\\\notag
&\le 
C_2\left(e^{-rt_1}|\Deltadot v(t_1)|| m(t_1)|+e^{-rt_2}|\Deltadot v(t_2)|| m(t_2)|\right) + C_2\int_{t_1}^{t_2}e^{-rs}|B(s)|^2ds
.
\end{align}
\end{lemma}

\begin{proof}
The last inequality follows from Lemma \ref{lem:b}. 
Fix $0\le t_1 <t_2 \le T$. Then, the chain rule implies (and the reasoning is given below):
\begin{align*}
&\Big[e^{-rs}m(s)\cdot v(s)\Big]^{t_2}_{t_1}\\
&\quad
=\int_{t_1}^{t_2}e^{-rs}\sum_{x\in[d]}v_x(s)\Big[\sum_{y\in[d]}m_y(s)\gamma^*_x(y,\Delta_y\bar u^r)+\sum_{y\in[d]}\bar\mu^r_y\nabla_p\gamma^*_x(y,\Delta_y\bar u^r)\cdot\Delta_y v(s)+B_x(s)\Big]ds\\
&\qquad
-
\int_{t_1}^{t_2}e^{-rs}\sum_{x\in[d]}m_x(s)\Big[\gamma^*(x,\Delta_x\bar u^r)\cdot\Delta_x v(s)+D^\eta_1 F(x,\bar\mu^r)\cdot m(s)\Big]ds\\
&\quad
=
\int_{t_1}^{t_2}e^{-rs}\Big[\sum_{x\in[d]}v_x(s)\sum_{y\in[d]}\bar\mu^r_y\nabla_p\gamma^*_x(y,\Delta_y\bar u^r)\cdot\Delta_y v(s)-
\sum_{x\in[d]}m_x(s)D^\eta_1 F(x,\bar\mu^r)\cdot m(s)\\
&\qquad\qquad\qquad\quad+v(s)\cdot B(s)\Big]ds\\
&\quad
=
-\int_{t_1}^{t_2}e^{-rs}\Big[\sum_{y\in[d]}\bar\mu^r_y(\Delta_y v(s))^\top\Gamma(y)(\Delta_y v(s))
+
\sum_{x\in[d]}m_x(s)D^\eta_1 F(x,\bar\mu^r)\cdot m(s)-v(s)\cdot B(s)\Big]ds\\
&\quad\le 
-C_{2,H}\int_{t_1}^{t_2}e^{-rs}\sum_{y\in[d]}|\Delta_y v(s)|^2ds
+\int_{t_1}^{t_2}e^{-rs}v(s)\cdot B(s)ds,
\end{align*} where $\Gamma_{xz}(y)=-\partial_{p_z}\gamma^*_x(y,\Delta_y \bar u^r)
=-\partial^2_{p_x,p_z}H(y,\Delta_y\bar u^r).$
The second and third equalities use the identity 
$\sum_{x\in[d]}\gamma^*_x(y,\Delta_y \bar u^r)=0$. By \eqref{monotone_2}, we know that $\sum_{x\in [d]} p_x D^\eta_1 F(x,\bar\mu^r)\cdot p\geq 0$ for any $p\in\calM_0$. The inequality above follows from this fact and since $\{\bar\mu_y^r\}_{y\in[d]}$ are bounded away from $0$. Finally, the bound in the statement of the lemma follows from the fact that $B(s)\in\calM_0$, by Young's inequality applied to $(v(s)-\langle v(s)\rangle)\cdot B(s)$, and by Lemma \ref{lem:b}.
\end{proof}

\begin{proposition}\label{prop:lin_MFG_r_exi}
Assume that there exist $C,\theta >0$ such that $|A(t)| +|B(t)| \leq C  e^{-\theta t}$. Then there exists $r_0>0$ such that for any $r\in (0,r_0)$, the system \eqref{lin_MFG_r_around_stat} admits a unique solution and for any $t\in\R_+$, $m(t)\in\calM_0$. 
\end{proposition}

\begin{proof} The existence follows in two steps. The first step is to show that the forward-backward system \eqref{lin_MFG_r_around_stat} but with finite horizon $[0,T]$ with the terminal condition $v(T)=0$ has a unique solution. For the second step, we send $T\to\infty$ and establish existence and uniqueness for the infinite horizon system. 

The first part follows by the same arguments given in \cite[Proposition 3.1]{bay-coh2019}. Notice that due to the discount factor, we have an extra term in the ODE of $v$ as compared with the one in \cite{bay-coh2019}. Since this additional term  appears linearly, the proof of the existence and uniqueness on any finite horizon interval still follows the same lines. 

Denote by $(\vT,\mT)$ the solution on the finite horizon $[0,T]$ with the terminal condition $v^{(T)}(T)=0$. We show that the sequence $\{(\vT,\mT)\}_{T\in\N}$ is a Cauchy sequence under the sup-norm, on any fixed interval $[0,T_0]$, $T_0>0$. Throughout the proof, $C$ and $c$ are generic positive constants, which may change from one line to the next and are independent of $r$, $m_0$, $T$ and $\tilde T$.

Since $A$ and $B$ are assumed to have exponential decay, we will denote the constants $C_A$, $C_B>0$ as those that result from the bound:
\begin{align}\label{cab}
    \int_0^\iy |A(t)| dt \leq  \int_0^\iy Ce^{-\theta t} \leq C_A,
\end{align} and similarly for $C_B$. At points throughout the proof we may temporarily assume that $A$ or $B$ is zero; if so, $C_A$ or $C_B$ will be set to zero in the citation of any bound derived in the more general, nonzero case. Note that $C_A$ and $C_B$ are independent of $r$, $T$, $\tilde T$, and $m_0$.

Fix arbitrary $\TT>T>0$. 
By Lemmas \ref{lem:M0} and \ref{lem:duality}, for any $t\in[0,T]$,
\begin{align}\notag
|(\mT-\mTT)(t)|
&\le 
Ce^{-ct}|(\mT-\mTT)(0)|+C\Big(\int_0^t|\Deltadot (\vT-\vTT)(s)|^2ds\Big)^{1/2}\\\notag
&\le 
Ce^{rt/2}\Big(\int_0^te^{-rs}|\Deltadot (\vT-\vTT)(s)|^2ds\Big)^{1/2}\\\notag
&\le 
Ce^{rt/2}\Big(\int_0^T e^{-rs}|\Deltadot (\vT-\vTT)(s)|^2ds\Big)^{1/2}\\\notag
&\le 
Ce^{-r(T-t)/2}|\Delta (\vT-\vTT)(T)|^{1/2}|(\mT-\mTT)(T)|^{1/2},
\end{align}
where we used multiple times the identity $\mT(0)=\mTT(0)$. Together with $\vT(T)=0$, and plugging in the same computation for when $t=T$, we obtain:
\begin{align}\label{bound_dif_m}
\sup_{t\in [0,T]} e^{-rt/2}|(\mT-\mTT)(t)|
&\le 
Ce^{-rT/2}|\Delta \vTT(T)|.
\end{align}
From Lemma \ref{lem:v}, \eqref{bound_dif_m}, and since $\vT(T)=0$, for any $t\in[0,T]$,
\begin{align*}\notag
e^{-rt}|\Delta( \vTT-\vT)(t)| &\le Ce^{-rt}e^{-c(T-t)}|\Delta \vTT (T)| + C \int_t^T e^{-c(s-t)} e^{-rs} |(\mTT - \mT)(s)| ds \\
&\leq Ce^{-rt}e^{-c(T-t)}|\Delta \vTT (T)| + C \int_t^T e^{-c(s-t)} e^{-rs/2} e^{-rT/2}|\Delta \vTT (T)| ds \\\notag
&\leq Ce^{-rt}e^{-c(T-t)}|\Delta \vTT (T)| + Ce^{-r(T+t)/2}|\Delta \vTT (T)|.
\end{align*}
Multiplying both sides by $e^{rt}$ and then for all $0<r<c$:
\begin{align}\label{bound_dif_v}
|\Delta (\vTT-\vT)(t)|
&\le Ce^{-r(T-t)/2}|\Delta \vTT(T)|.
\end{align} 

To develop the bound from \eqref{bound_dif_v}, we start with estimating $\mTT(t)$ for $t\in[0,\TT]$. Using Lemmas \ref{lem:M0} and \ref{lem:duality}, \eqref{cab}, and $\vTT(\tilde T)=0$:
\begin{align}\label{bound_m_TT}
\begin{split}
|\mTT(t)|
&\le 
Ce^{-ct}|m_0|+C \Big(\int_0^t|\Deltadot \vTT(s)|^2ds\Big)^{1/2}+C_B\\
&\le 
Ce^{-ct}|m_0|+Ce^{rt/2}\Big(\int_0^te^{-rs}|\Deltadot \vTT(s)|^2ds\Big)^{1/2}+C_B\\
&\le 
Ce^{-ct}|m_0|+Ce^{rt/2}\Big(\int_0^{\tilde T} e^{-rs}|\Deltadot \vTT(s)|^2ds\Big)^{1/2}+C_B\\
&\le 
Ce^{-ct}|m_0|+Ce^{rt/2}|m_0|^{1/2}|\Delta\vTT(0)|^{1/2}+C_B.
\end{split}
\end{align}
Using Lemma \ref{lem:v}, $\vTT(\tilde T)=0$, the last bound \eqref{bound_m_TT}, and the fact that $|A(t)|\leq Ce^{-\theta t}$, we get:
\begin{align}\label{bound_v_TT}
\begin{split}
e^{-rt}|\Delta \vTT(t)|
&\le 
C\int_t^{\TT}e^{-c(s-t)}e^{-rs}\left(|\mTT(s)|+ A(s)\right)ds\\
&\le 
C\int_t^{\TT}e^{-c(s-t)}e^{-rs}
\Big(e^{-cs}|m_0|+e^{rs/2}|m_0|^{1/2}|\Delta \vTT(0)|^{1/2}+A(s)+ C_B
\Big)ds\\
&\le 
C \Big(e^{-rt}(|m_0|+C_A + C_B)+e^{-rt/2}|m_0|^{1/2}|\Delta \vTT(0)|^{1/2}\Big).
\end{split}
\end{align}
Specifically, by taking $t=0$,
\begin{align}\notag
|\Delta \vTT(0)|
&\le 
C\Big(|m_0|^{1/2}|\Delta \vTT(0)|^{1/2}+(|m_0|+C_A+C_B)\Big).
\end{align} Applying Young's inequality and rearranging,  we get that
\begin{align}\notag
|\Delta \vTT(0)|
&\le 
C(|m_0|+C_A+C_B).
\end{align}
Plugging this back in \eqref{bound_m_TT} and \eqref{bound_v_TT}, we get for any $t\in[0,\TT]$, 
\begin{align}\label{bound:lin_A_B}
\begin{split}
|\mTT(t)|
&\le 
Ce^{rt/2}(|m_0|+C_A+C_B)
\\
|\Delta \vTT(t)|&\le C e^{rt/2}(|m_0|+C_A+C_B).
\end{split}
\end{align}

Assume momentarily that $A\equiv B\equiv 0$. We will temporarily use this fact to improve \eqref{bound:lin_A_B} to exponential decay. 

{\bf Improving \eqref{bound:lin_A_B} to exponential decay - the case $A\equiv B\equiv 0$:} We will relax this assumption in the sequel. For every $r$, set the function, 
\begin{align}\notag
\rho^r(t):=\sup_{T, m_0} e^{-rt}m^{(T)} (t)\cdot v^{(T)} (t),\qquad t\in\R_+ ,  
\end{align} 
where the supremum is taken over all $T\geq t$ and $|m_0| \leq 1$ with $m_0 \in \calM_0$, and where $(v^{(T)},m^{(T)})$ solves the finite horizon linearized system with the initial condition $m(0)=m_0$ and the terminal condition $v^{(T)}(T)=0$. In what follows, we use multiple times the fact that for any $m_0\in \calM_0$ and any solution $(v^{(T)},m^{(T)})$ to \eqref{lin_MFG_r_around_stat}, one has for $t\in [0,T]$:
\begin{equation}\label{star_T}
    0\le e^{-rt} m^{(T)}(t) \cdot v^{(T)}(t) \le |m_0|^2 \rho^r(t).
\end{equation} The first inequality follows from Lemma \ref{lem:duality} and since $v^{(T)}(T)=0$.
The second inequality follows by the linearity of $(v^{(T)},m^{(T)})$ with respect to $m_0$ and the fact that $m^{(T)}(t)\cdot v^{(T)}(t)$ is homogeneous of degree $2$.

For any given $r$, the function $\rho^r(\cdot)$ is non-increasing by Lemma \ref{lem:duality} and since the interval on which we take the supremum in the definition of $\rho^r$, shrinks as $t$ increases. From \eqref{bound:lin_A_B} and \eqref{star_T}, for any $r_0$, $\{\rho^r(t)\}_{r\in(0,r_0), t \in \R_+}$ is bounded and nonnegative.  Set,
\begin{align}\notag
\rho(t):=\limsup_{r\to 0^+}\rho^r(t),\qquad t \in \R_+.
\end{align} This function inherits the non-negativity and non-increasing properties. We upgrade \eqref{bound:lin_A_B} to (a uniform) exponential decay from exponential growth (in $r$) by the following steps: 
\begin{enumerate}[(a)]
\item showing that $\rho_{\iy}:=\lim_{t\to\iy}\rho(t)=0$;
\item showing that there exist $\gamma, C_1, \bar r_0>0$ independent of $m_0$ such that for any $r\in(0,\bar r_0)$, for any $t \in \R_+$, 
\begin{align}\label{rho_exp_bound_T}
\rho^r(t)\le C_1 e^{-\gamma t};
\end{align} 
\item establishing the exponential bound.
\end{enumerate}

We start with proving (a). By the definition of $\rho_\iy$ as a double limit, let $(T_n, t_n, r_n)_{n\in\N}$ be an array of sequences such that $t_n\to\iy$, $T_n \geq t_n$, and $r_n \searrow 0$, such that for every $m^n_0\in\calM_0$ satisfying $|m^n_0|\le 1$, one has,
\begin{align}\notag
e^{-r_nt_n} m^n(t_n)\cdot v^n(t_n)\ge \rho_\iy -\frac{1}{n},
\end{align}
where $(v^n,m^n)$ are associated with $m^n_0$ and the horizon $T_n$.  Define the functions 
\begin{align}\notag
\tilde v^n(s): = e^{-r_nt_n/2}(v^n(t_n+s)-\langle v^n(t_n)\rangle),\qquad 
\tilde m^n(s):=e^{-r_nt_n/2}m^n(t_n+s),\qquad s\in[-t_n,0].
\end{align}
From \eqref{bound:lin_A_B}, the functions $\tilde m^n$ and $ \tilde v^n$ are locally bounded. Hence, along a subsequence, which we relabel by $\{n\}$, we have the locally-uniform convergence $(\tilde v^n,\tilde m^n)\to (\tilde v, \tilde m)$, where  $(\tilde v, \tilde m)$ satisfies for $t\in (-\iy,0]$ and $x\in[d]$,
\begin{align}\notag
\begin{cases}
-\frac{d}{dt}\tilde v_x(t)= \gamma^*(x,\Delta_x \bar u^r)\cdot \Delta_x\tilde v(t) + D^\eta_1 F(x,\bar\mu^r)\cdot \tilde m(t),\\
\frac{d}{dt}\tilde m_x(t)=\sum_{y\in[d]} \tilde m_y(t)\gamma^*_x(y,\Delta_y \bar u^r) + \sum_{y\in[d]}\bar \mu^r_y \nabla_p\gamma^*_x(y,\Delta_y \bar u^r)\cdot \Delta_y \tilde v(t).
\end{cases}
\end{align} Keep in mind that $m^n(s)\in\calM_0$, hence $\langle v^n(s)\rangle\cdot m^n(s)=0$. Recalling the definition of $\rho^r$, it follows that for any $s\le 0$ and $t_n \ge 0$, and large enough $n$,
\begin{align}\notag
e^{-r_ns}\tilde m^n(s)\cdot \tilde v^n(s)=e^{-r_n(t_n+s)} m^n(t_n+s)\cdot  v^n(t_n+s)\le \rho^{r_n}(t_n+s)\le \rho^{r_n}(t_n).
\end{align} Taking $n\to\iy$, we get for any $s\le 0$,
$$
\tilde m(s)\cdot\tilde v(s)\le \rho_\iy=\tilde m(0)\cdot\tilde v(0),
$$
where the equality follows by substituting $s=0$, taking the limit 
as $n\to\iy$, and recalling the definition of $(v^n, m^n)$. By Lemma \ref{lem:duality} applied to $(\tilde m,\tilde v)$, $\tilde m(s)\cdot\tilde v(s)$ is non-increasing with $s$. Together with the last bound, it follows that $
\tilde m(s)\cdot\tilde v(s)$ is constant, and so equals $\rho_\iy$. By duality again, it follows that $\Deltadot \tilde v(s)=0$ for any $s\le 0$, which together with $m(s)\in\calM_0$ implies that $\rho_\iy=\tilde m(s) \cdot \tilde v(s) = \tilde m(s) \cdot \Delta \tilde v(s) =0$.

We now turn to establishing (b). From part (a) we have that for any $\eps>0$, there are $\tau_0>0$ and $\bar r_0>0$ such that for any $t>\tau_0$ and $r\in(0,\bar r_0)$, $\rho^r(t) \leq \eps$. Combining it with the duality lemma and \eqref{star_T}, we get that there exists $C_2>0$, independent of $m_0$, such that for any $T >t_2 > t_1 >\tau_0$ and $r\in (0,\bar r_0)$ and finite horizon solutions $(v^{(T)},m^{(T)})$ to \eqref{lin_MFG_r_around_stat} (associated with $r$):
\begin{align}\notag
\int_{t_1}^{t_2}e^{-rs}|\Deltadot v^{(T)} (s)|^2ds\le C_2 \eps|m_0|^2. 
\end{align} Re-estimating $m$ on $[\tau_0,\tau_0 + t]$ with $\tau_0 + t \leq T$, using Lemma \ref{lem:M0}, \eqref{bound:lin_A_B}, and the above, we get 
\begin{align*}
|m^{(T)}(\tau_0+t)|
&\le Ce^{-c t}|m^{(T)}(\tau_0)|+C\Big(\int_{\tau_0}^{\tau_0+t}|\Deltadot v^{(T)}(s)|^2ds\Big)^{1/2}\\\notag
&\le Ce^{-ct+r\tau_0/2}|m_0|+Ce^{r(\tau_0+t)/2}\Big(\int_{\tau_0}^{\tau_0+t}e^{-rs}|\Deltadot v^{(T)}(s)|^2ds\Big)^{1/2}\\\notag
&\le CC_2 e^{r(\tau_0+t)/2}|m_0|\Big(e^{-(c+r/2)t}+\eps^{1/2}\Big).
\end{align*}
Take $t$ such that $CC_2 e^{-ct}\le 1/4$, recalling that $C, C_2$ are independent of $\eps$ and $r$. If necessary, we can take $T>0$ larger enough to accommodate this choice. Also, take $\eps$ such that $CC_2 \eps^{1/2}\le 1/4$. Denote $\tilde \tau:= \tau_0+t$, so we have
\begin{align}\label{bound_m_tau_T}
|m^{(T)}(\tilde \tau)|\le \frac{1}{2}|m_0|e^{r\tilde \tau /2}. 
\end{align}
Define, $(v^{\tilde \tau}(s), m^{\tilde \tau}(s)) := (v^{(T)}(\tilde \tau+s),m^{(T)}(\tilde \tau+s))$, $s\in [0, T-\tilde \tau]$. This is a solution to \eqref{lin_MFG_r_around_stat} with the initial condition $m^{\tilde \tau}(0)=m^{(T)}(\tilde \tau)$ on the time interval $[0,T-\tilde\tau]$, with the terminal condition $m^{\tilde\tau}(T-\tilde\tau)=0$. By \eqref{star_T}:
\begin{align*}
e^{-rs} m^{\tilde \tau} (s)\cdot  v^{\tilde \tau} (s)\le |m^{(T)}(\tilde \tau)|^2\rho^r(s),\qquad s\in[0,T-\tilde\tau].
\end{align*} Therefore, by \eqref{bound_m_tau_T},
\begin{align}\notag
e^{-r(s+\tilde \tau)}m^{(T)}(s+\tilde \tau)\cdot v^{(T)}(s+\tilde \tau)\le e^{-r\tilde \tau}|m^{(T)}(\tilde \tau)|^2\rho^r(s) \le \frac{1}{4} |m_0|^2 \rho^r(s).
\end{align}
Taking supremum over $\{m_0\in\calM_0 : |m_0|\le 1\}$ and $T\geq s+\tilde \tau$ on both sides, we get, 
\begin{align}\notag
\rho^r(s+\tilde \tau)\le \frac{1}{4}\rho^r(s).
\end{align} Since this can be established for any $\tilde \tau$ large enough (independently of $\eps$ and $r$), we obtain that (b) holds.

Finally, we turn to proving (c). We use Lemma \ref{lem:M0} on $[t/2,t]$, Lemma \ref{lem:duality}, and \eqref{rho_exp_bound_T}, to obtain that for any $t\in [0,T]$,
\begin{align*}
|m^{(T)}(t)|
&\le 
Ce^{-ct/2}|m^{(T)}(t/2)|+C\big(\int_{t/2}^t|\Deltadot v^{(T)} (s)|^2ds\Big)^{1/2}\\
&\le
Ce^{-ct/2+rt/2}|m_0|+Ce^{rt/4}\big(\int_{t/2}^t e^{-rs}|\Deltadot v^{(T)}(s)|^2ds\Big)^{1/2}\\
&\le 
CC_1 |m_0|\Big(e^{-ct/2+rt/2}+e^{rt/4-\gamma t/2}\Big).
\end{align*}
Take $r_0$ sufficiently small, such that for any $r\in(0,r_0)$, we have
\begin{align}\notag
|m^{(T)} (t)|\le C|m_0|e^{-\lambda t},\qquad t\in[0,T],
\end{align}
for some $\lambda\in(0,c)$. Modifying $\lambda$ if necessary, by an application  of \eqref{bound:lin_A_B} on the interval $[t/2,t]$, one gets,
\begin{align}\notag
|\Deltadot v^{(T)}(t)|\le C|m^{(T)}(t/2)|e^{rt/4}\le C|m_0|e^{-\lambda t}.
\end{align} Altogether, we have:
\begin{align}\label{eqn:fin_hor_vm_dec}
    |\Delta v^{(T)} (t)| + |m^{(T)}(t)| \leq C|m_0| e^{-\lambda t}
\end{align} for $t\in [0,T]$.

{\bf Improving \eqref{bound:lin_A_B} to exponential decay - general $A$ and $B$:}  Note that the bound \eqref{eqn:fin_hor_vm_dec} was established for $A\equiv B\equiv 0$. Yet, this kind of decay would be useful in proving this proposition; namely, by improving \eqref{bound_dif_m} and \eqref{bound_dif_v} with decay in $T$ and then sending $T,\tilde T\to\iy$. Therefore, we turn toward establishing \eqref{eqn:fin_hor_vm_dec} for possibly nonzero $A$ and $B$ with $|A(t)| + |B(t)| \leq Ce^{-\theta t}$ as in the statement of the proposition. We start by showing that one can reduce the problem to one with $A\equiv 0$. To this end, let $(v^A,m^A)$ be the unique solution to \eqref{lin_MFG_r_around_stat} with horizon $T>0$ and with some $A$ and $B$ as given in the proposition. Also, let $\check v$ be the unique (bounded) solution to:
\begin{align}\notag
-\frac{d}{dt}\check v_x(t)=-r\check v_x(t) + \gamma^*(x,\Delta_x \bar u^r)\cdot \Delta_x \check v(t) +A_x(t), \qquad \forall t\in [0,T],
\end{align} with $\check v(T) =0$. By Lemma \ref{lem:v}, there exist $c_3,C_3>0$ depending only on $\mathfrak{a}_l$ and $\mathfrak{a}_u$, such that:
\begin{align}\label{bound_check_v_T}
|\Deltadot \check v (t)|\le C_3\int_t^T e^{-(c_3+r)(s-t)}|A(s)|ds\le 
\frac{CC_3}{c_3 + \theta +r}e^{-\theta t}\leq Ce^{-\theta t},   \qquad \forall t\in[0,T],
\end{align} where we again used \eqref{cab}.
Now, $(v^A-\check v,m^A)$ solves \eqref{lin_MFG_r_around_stat} with $A\equiv 0$ and $\check B_x(t)=B_x(t)+\sum_{y\in[d]}\bar\mu^r_y\nabla_p\gamma^*_x(y,\Delta_y \bar u^r)\cdot\Delta_y \check v(t)$.
Trivially, $\check B(t)\in\calM_0$, recall from Proposition \ref{prop:statDisc_erg_mfg} that $\Delta_x \bar u^r$ are uniformly bounded for $r\in(0,r_0)$, and by assumption the second derivative of $H$ is bounded on any compact set and equals $\nabla_p \gamma^*$. Therefore by updating $C$ from \eqref{cab}, 
\begin{align}\label{eqn:checkB0_T}
|\check B(t)| \le C e^{-\theta t}. 
\end{align}

For any $b>0$, define the set of measurable functions: 
$$\calB_b:=\{B:\R_+\to\calM_0 \;|\;\forall t\in [0,T],\; | B(t)| \le  b e^{-\theta t}\}. $$
We prove that for any $C_4>0$ there exists $C_5>0$ independent of $C_4,r,m_0$, and $T$,  such that 
\begin{align}\label{bound_v_m_T}
\sup_{B\in\calB_{C_4}}\left\{|m^{(T)}(t)|+|\Deltadot v^{(T)}(t)|\right\}\le C_5C_4 e^{-(\lambda\wedge \gamma) t},\quad t\in [0,T], 
\end{align}
where $(v^{(T)},m^{(T)})$ solves \eqref{lin_MFG_r_around_stat} with $A\equiv 0$ and some $B\in\calB_{C_4}$.
Once \eqref{bound_v_m_T} is established, by recalling that $(v^A-\check v,m^A)$ solves \eqref{lin_MFG_r_around_stat} with $A\equiv 0$, then from \eqref{eqn:checkB0_T} and from \eqref{bound_v_m_T}, we obtain for all $t\in [0,T]$ that:
\begin{align*}
|m^A(t)|+|\Deltadot (v^A-\check v)(t)|\le C_5 C_4 e^{-(\lambda\wedge \gamma) t}.
\end{align*}
Together with \eqref{bound_check_v_T},  we obtain that:
\begin{align*}
|m^A(t)|+|\Deltadot v^A(t)|\le 2C_5 C_4 e^{-(\lambda\wedge \gamma) t},\qquad\forall t\in [0,T].
\end{align*}

Next, we establish \eqref{bound_v_m_T} (which by definition considers $A\equiv 0)$. Exploiting the linearity of the system with respect to $m_0$ and using \eqref{eqn:fin_hor_vm_dec}, we may establish \eqref{bound_v_m_T} for the case $m_0=0$. Using \eqref{bound:lin_A_B}, we use the fact that $B\in\calB_{C_4}$ to obtain that, for all $t\in [0,T]$:
\begin{align}\label{eq:mv_rt2_T}
    |m^{(T)}(t)|+|\Delta v^{(T)}(t)|\le C C_4 e^{rt/2}.
\end{align}

For any $b>0$ and for $t\in [0,T]$, set:
\[
\bar\rho^r(b,t) := b^{-1}\sup_{B\in\calB_b} e^{-rt} (|m^{(T)}(t)| + |\Delta v^{(T)}(t)|).
\] where $(v^{(T)},m^{(T)})$ solves \eqref{lin_MFG_r_around_stat} on the horizon $T$ with $A\equiv 0$, $m(0)=0$, and some $B\in\calB_{b}$. 

The structure of \eqref{lin_MFG_r_around_stat} implies that if $(\tilde v, \tilde m)$ is a solution with a given $B$, then $(\ell \tilde v, \ell \tilde m)$ is the solution associated to $\ell B$ for any $\ell \in\R$. Therefore, for any $b>0$, 
\begin{align}\label{eq:rho_c_T}
    \bar\rho^r(t):=\bar\rho^r(1,t)
    =\bar\rho^r(b,t),\qquad t\in [0,T].
\end{align} Fix $\tau >0$ and let $(v^{\tau,1},m^{\tau,1})$ be the solution of \eqref{lin_MFG_r_around_stat} on $[\tau,T]$ with $A\equiv B\equiv 0$ and $m^{\tau,1} (\tau) = m^{(T)}(\tau)$. Also let $(v^{\tau,0},m^{
\tau,0})$ be the solution to \eqref{lin_MFG_r_around_stat} on $[\tau,T]$ with $A\equiv 0$, $m^{\tau,0}(\tau)=0$, and the given $B$. Then on the time interval $[\tau,T]$, $(v,m)=(v^{\tau,1},m^{\tau,1})+(v^{\tau,0},m^{\tau,0})$. 

Note that $(\hat v^{\tau,0}(\cdot),\hat m^{\tau,0}(\cdot)):=(v^{\tau,0}(\tau+\cdot),m^{\tau,0}(\tau+\cdot))$ solves \eqref{lin_MFG_r_around_stat} on $[0,T-\tau]$. By the assumption on $B$, $B(\tau + t) \le C_4 e^{-\gamma t} e^{-\gamma \tau}$. Together with the definition of $\bar\rho^r(b,t)$ and the aforementioned linear property, we have
\[
C_4^{-1}e^{-rt} (|m^{\tau,0} (\tau + t)| + |\Delta v^{\tau ,0 }(\tau +t)|) = e^{-\gamma \tau}(C_4 e^{-\gamma\tau})^{-1}e^{-rt} (|\hat m^{\tau,0} ( t)| + |\Delta \hat v^{\tau ,0 }( t)|) \leq e^{-\gamma \tau}\bar\rho^r( t).
\] 
On the other hand, \eqref{eqn:fin_hor_vm_dec} and \eqref{eq:mv_rt2_T} imply that there exist $\lambda, C_6>0$, independent of $C_4,m_0$, such that for all $r\in (0,r_0)$ and for all $t\in [0,T-\tau]$,
\[
|m^{\tau,1} (\tau +t)| + |\Delta v^{\tau,1} (\tau+t)| \le C_6 e^{-\lambda t} |m(\tau)| \le C_6 C_4 e^{-\lambda t} e^{r\tau /2}. 
\] 
As a consequence,
\begin{align}\notag
C_4^{-1}e^{-r(t+\tau)}\left(|m(\tau+t)|+|\Deltadot v (\tau+t)|\right)
\le 
C_6 e^{-(\lambda+r)t}e^{-r\tau/2}+ e^{-(\gamma+r)\tau}\bar\rho^r(t).
\end{align}
Taking the supremum over all $B \in \calB_{C_4}$, one gets
\begin{align}\notag
\bar\rho^r(\tau+t)\le C_6 e^{-(\lambda+r)t}+e^{-(\gamma+r)\tau}\bar\rho^r(t).
\end{align} Multiply both sides by $e^{((\lambda\wedge\gamma) +r)(t+\tau)}$, one gets,
\begin{align}\notag
e^{((\lambda\wedge\gamma)+r)(t+\tau)}\bar\rho^r(\tau+t)\le C_6 e^{((\lambda\wedge\gamma)+r)\tau}+e^{((\lambda\wedge\gamma)+r)t}\bar\rho^r(t),
\end{align}
which, together with \eqref{eq:rho_c_T} implies the exponential decay in \eqref{bound_v_m_T}. Indeed, set $\theta:=\lambda\wedge \gamma$ and $\beta(t):=e^{(\theta+r)t}\bar \rho^r(t)$. Then, $\beta(t+\tau)\le Ce^{(\theta +r)\tau}+\beta(t)$. First, take $t$ to be a non-negative integer and $\tau=1$, then,
$$\beta(t)\le Cte^{\theta +r}+\beta(0)
.$$
Then, take $\tau\in(0,1)$ and:
$$
\beta(\tau+t)\le Ce^{(\theta+r)\tau}+ Cte^{\theta +r}+\beta(0)\le  C(t+1)e^{\theta +r}+\beta(0).
$$
So in general, for any $s\in[0,T]$,
$$\beta(s)\le C(s+1)e^{\theta+r}+\beta(0).
$$ Now, recall the definition of $\beta$ to get the desired decay. That is, we have improved the bound \eqref{eqn:fin_hor_vm_dec} for solutions to \eqref{lin_MFG_r_around_stat} on $[0,T]$ and with $v^{(T)}(T)=0$ so that:
\begin{align}\label{eqn:fin_hor_vm_dec_2}
    |\Delta v^{(T)}(t)| + |m^{(T)}(t)| \leq C (1+|m_0|+t) e^{-\lambda t}, \quad t\in [0,T], \quad A,B \in \calB,
\end{align} where $\calB := \bigcup_{b\in\N} \calB_b$.

{\bf Finishing the proof using the exponential decay from \eqref{eqn:fin_hor_vm_dec_2}:} Returning to \eqref{bound_dif_m} and \eqref{bound_dif_v}, we may use \eqref{eqn:fin_hor_vm_dec_2} (now established for general $A$ and $B$ with exponential decay), to obtain for any $T_0>0$:
\begin{align*}
    \sup_{t\in [0,T_0]} \Big( |m^{(T)}(t) - m^{(\TT)}(t)| + |\Delta (v^{(T)} - v^{(\TT)})(t)| \Big) \leq C(1+|m_0|+T) e^{rT_0/2} e^{-rT}\to 0,
\end{align*} as $T,\tilde T \to\iy$ since $C$ is independent of $T,\tilde T$. Moreover, integrating the equation that $e^{-rt}(v^{(T)} - v^{(\TT)})(t)$ satisfies, using \eqref{bound_dif_m}, \eqref{bound_dif_v}, and \eqref{eqn:fin_hor_vm_dec_2}: 
\begin{align*}
   e^{-rt}|(v_x^{(T)} - v_x^{(\TT)})(t)| &\leq \int_t^T \big[e^{-rs} |\gamma^*(x,\Delta_x \bar u^r) \cdot \Delta_x (v^{(T)}-v^{(\TT)})(s) \\
   &\qquad+ D^\eta_1 F(x,\bar\mu^r) \cdot (m^{(T)}-m^{(\TT)})(s)|\big]ds + e^{-rT} |v^{(\TT)} (T)| \\
   &\leq C e^{-rT} |\Delta v^{(\TT)}(T)| \int_t^T e^{-rs/2} ds \\
   &\leq Cr^{-1}
   e^{-rT} e^{-rt/2} |\Delta v^{(\TT)} (T)| \\
   &\leq C(1+|m_0|+T)r^{-1} 
   e^{-2rT},
\end{align*} where we emphasize the introduction of $r^{-1}$ because the generic constant $C$ is assumed to be independent of $r$. Therefore as $T,\TT \to\iy$, 
\[
\sup_{t\in [0,T_0]} |(v^{(T)}-v^{(\TT)})(t)| \to 0.
\] So, there is a sequence $\{T^m\}_{m\in\N}\to\iy$ such that for any $T_0>0$, $\{(m^{(T^m)},v^{(T^m)})\}_{m\in\N}$ is a Cauchy sequence in the sup-norm on $[0,T_0]$. Then there is a unique limit $(m,v)$, which is consistent on any interval of the form $[0,T_0]$. This limit then satisfies \eqref{lin_MFG_r_around_stat}. By a similar argument one can obtain uniqueness. That is, linearity of the ODE implies the difference of any two solutions is a solution with $A\equiv B \equiv m_0 = 0$. The argument is the same as in the following Lemma \ref{lem:lin_MFG_r_exp} that provides an explicit bound depending on the initial data and so implies uniqueness.
\end{proof}

\color{black}

We now present a few preliminary results for the linearized system \eqref{lin_MFG_r_around_stat}. Specifically, we obtain exponential bounds on the solutions. The proofs of these results follow the same lines as the previous proof and the proofs of Lemma 3.3 and Propositions 3.4 and 3.5 in \cite{car-por}. Specifically, the proofs rely on Lemmas \ref{lem:M0}, \ref{lem:v}, and \ref{lem:duality}. On a first reading, these proofs may be skipped.

\begin{lemma}\label{lem:lin_MFG_r_exp}
{\color{purple} }Let $(v,m)$ be the solution of \eqref{lin_MFG_r_around_stat} with $A=B\equiv0$. Then, for any $t\in\R_+$, 
\begin{align}\label{bound_0}
m(t)\cdot v(t)\ge 0,
\end{align}
and there exists a positive constant $C>0$, independent of $m_0$ and $r$, such that for any $t\in\R_+$,
\begin{align}\label{bound_m_Del_v}
|m(t)|+|\Deltadot v(t)|\le C|m_0|e^{rt/2}.
\end{align}
Moreover, there exist constants $r_0,C,\lambda>0$ such that for any $r\in(0,r_0)$,
\begin{align}\label{bound_m_Del_v_2}
|m(t)|+|\Deltadot v(t)|\le C|m_0|e^{-\lambda t}.
\end{align} 
\end{lemma}

\begin{proof} In the proof, we use $C>0$ as a generic constant that is independent of the time $t$, the discount factor $r$, and the initial data $m_0$ and may change from one line to the next.

{\bf Proof of \eqref{bound_0} and \eqref{bound_m_Del_v}:} Consider a solution $(v,m)$ to \eqref{lin_MFG_r_around_stat}. Define the constant  $V:=\sup_{(t,x)\in\R_+\times[d]}|v_x(t)|$. Note that $V$ may depend on $r$ through $v$.
From Lemmas \ref{lem:M0} and \ref{lem:duality},
\begin{align}\notag
|m(t)|
&\le 
Ce^{-ct}|m(0)|+C\Big(\int_0^t|\Deltadot v(s)|^2ds\Big)^{1/2}\\\notag
&\le
Ce^{-ct}|m(0)|+Ce^{rt/2}\Big(\int_0^te^{-rs}|\Deltadot v(s)|^2ds\Big)^{1/2}\\\notag
&\le
Ce^{-ct}|m(0)|+Ce^{rt/2}V\Big(|m(0)|^{1/2}+e^{-rt/2}|m(t)|^{1/2}\Big),
\end{align} where the positive constant $C$ in the above display is independent of $r$. Therefore, there exists a constant $C_{r,m_0,V}>0$ that depends on $r,m_0,V$ such that
\begin{align}\notag
|m(t)|\le C_{r,m_0,V}e^{rt/2}, \qquad t\in\R_+.
\end{align} We will prove the existence of a constant that is independent of these parameters, as desired. As a consequence of the previous assertion and since we assume $v$ is bounded, 
\begin{align}\label{lim_mv}
\lim_{t\to\iy}e^{-rt}m(t)\cdot v(t)=0,
\end{align}
which together with Lemma \ref{lem:duality} applied on $\R_+$, implies that
\begin{align}\label{int_dual}
\int_0^\iy e^{-rt}|\Deltadot v(t)|^2dt\le Cm_0\cdot v(0)\le C|m_0||\Deltadot v(0)|.
\end{align}
Applying Lemma \ref{lem:duality} on $[t_1,t_2]$ implies that $t\mapsto e^{-rt} m(t)\cdot v(t)$ is nonincreasing; together with \eqref{lim_mv}, we get \eqref{bound_0}. From \eqref{int_dual} and Lemma \ref{lem:M0}, we can reestimate $m(t)$ as follows.
\begin{align}\label{bound_m}
\begin{split}
|m(t)|
&\le
Ce^{-ct}|m(0)|+Ce^{rt/2}\Big(\int_0^te^{-rs}|\Deltadot v(s)|^2ds\Big)^{1/2}\\
&\le
Ce^{-ct}|m(0)|+Ce^{rt/2}|m(0)|^{1/2}|\Deltadot v(0)|^{1/2}.
\end{split}
\end{align}
From Lemma \ref{lem:v} together with the above, for any $t<t_1$,
\begin{align}\notag
|\Deltadot v(t)|
&\le 
Ce^{-c(t_1-t)}|\Deltadot v(t_1)|+C\int_t^{t_1} e^{-c(s-t)}|m(s)| ds\\\notag
&\le 
Ce^{-c(t_1-t)}|\Deltadot v(t_1)| + Ce^{-ct}|m_0|+Ce^{rt/2}|m(0)|^{1/2}|\Deltadot v(0)|^{1/2}.
\end{align}
Taking $t_1\to\iy$, one obtains
\begin{align}\label{bound_Del_v}
|\Deltadot v(t)|
&\le 
Ce^{-ct}|m_0|+Ce^{rt/2}|m_0|^{1/2}|\Deltadot v(0)|^{1/2}.
\end{align}
Substituting $t=0$, using Young's inequality, and rearranging,
\begin{align}\notag
|\Deltadot v(0)|
&\le C
|m_0|.
\end{align}
Plugging this into \eqref{bound_m} and \eqref{bound_Del_v}, one obtains \eqref{bound_m_Del_v}. 

{\bf Proof of \eqref{bound_m_Del_v_2}:} For every $r$, set the function, 
\begin{align}\notag
\rho^r(t):=\sup_{m_0\in\calM_0:|m_0|\le 1}e^{-rt}m(t)\cdot v(t),\qquad t\in\R_+,  
\end{align} where $(v,m)$ solves \eqref{lin_MFG_r_around_stat} with the initial condition $m(0)=m_0$. In what follows, we use multiple times the fact that for any $m_0\in \calM_0$ and any solution $(v,m)$ to \eqref{lin_MFG_r_around_stat}, one has
\begin{equation}\label{star}
    0\le e^{-rt} m(t) \cdot v(t) \le |m(0)|^2 \rho^r(t).
\end{equation} It follows by \eqref{bound_0}, the linearity of $(v,m)$ with respect to $m_0$ and the fact that $m(t)\cdot v(t)$ is homogeneous of degree $2$.

Lemma \ref{lem:duality} implies that for any given $r$, the function $\rho^r(\cdot)$ is non-increasing. From \eqref{bound_0} and \eqref{bound_m_Del_v}, for any $r_0$, $\{\rho^r(t)\}_{r\in(0,r_0), t\in\R_+}$ is bounded and nonnegative. Set,
\begin{align}\notag
\rho(t):=\limsup_{r\to 0^+}\rho^r(t),\qquad t\in\R_+.
\end{align} This function inherits the non-negativity and non-increasing properties. We establish \eqref{bound_m_Del_v_2} by the following steps: 
\begin{enumerate}[(a)]
\item showing that $\rho_{\iy}:=\lim_{t\to\iy}\rho(t)=0$;
\item showing that there exist $\gamma, C, \bar r_0>0$ independent of $m_0$ such that for any $r\in(0,\bar r_0)$ and $t\in\R_+$, 
\begin{align}\label{rho_exp_bound}
\rho^r(t)\le Ce^{-\gamma t};
\end{align} 
\item establishing the exponential bound: \eqref{bound_m_Del_v_2} holds.
\end{enumerate}

We start with proving (a). By the definition of $\rho_\iy$ as a double limit, let $t_n$ be such that for any $t\ge t_n$ there is $r_n:=r_n(t)>0$ such that for any $r\le r_n$, and every $m^n_0\in\calM_0$ satisfying $|m^n_0|\le 1$, one has,
\begin{align}\notag
e^{-r_nt_n} m^n(t_n)\cdot v^n(t_n)\ge \rho_\iy -\frac{1}{n},
\end{align}
where $(v^n,m^n)$ are associated with $m^n_0$. 
Define the functions 
\begin{align}\notag
\tilde v^n(s): = e^{-r_nt_n/2}(v^n(t_n+s)-\langle v^n(t_n)\rangle),\qquad 
\tilde m^n(s):=e^{-r_nt_n/2}m^n(t_n+s),\qquad s\in[-t_n,\iy).
\end{align}
From \eqref{bound_m_Del_v} and Lemma \ref{lem:b}, the functions $\tilde m^n$ and $ \tilde v^n$ are locally bounded. Hence, along a subsequence, which we relabel by $\{n\}$, we have the locally-uniform convergence $(\tilde v^n,\tilde m^n)\to (\tilde v, \tilde m)$, where  $(\tilde v, \tilde m)$ satisfies for $t\in (-\iy,0]$ and $x\in[d]$,
\begin{align}\notag
\begin{cases}
-\frac{d}{dt}\tilde v_x(t)= \gamma^*(x,\Delta_x \bar u^r)\cdot \Delta_x\tilde v(t) + D^\eta_1 F(x,\bar\mu^r)\cdot \tilde m(t),\\
\frac{d}{dt}\tilde m_x(t)=\sum_{y\in[d]} \tilde m_y(t)\gamma^*_x(y,\Delta_y \bar u^r) + \sum_{y\in[d]}\bar \mu^r_y \nabla_p\gamma^*_x(y,\Delta_y \bar u^r)\cdot \Delta_y \tilde v(t).
\end{cases}
\end{align} Keep in mind that $m^n(s)\in\calM_0$, hence $\langle v^n(s)\rangle\cdot m^n(s)=0$. Recalling the definition of $\rho^r$, it follows that for any $s\le 0$ and $\tau\ge 0$, and large enough $n$,
\begin{align}\notag
e^{-r_ns}\tilde m^n(s)\cdot \tilde v^n(s)=e^{-r_n(t_n+s)} m^n(t_n+s)\cdot  v^n(t_n+s)\le \rho^{r_n}(t_n+s)\le \rho^{r_n}(\tau).
\end{align}
Taking $n\to\iy$, we get $\tilde m(s)\cdot\tilde v(s)\le \rho(\tau)$. Now, send $\tau\to\iy$, we get for any $s\le 0$,
$$
\tilde m(s)\cdot\tilde v(s)\le \rho_\iy=\tilde m(0)\cdot\tilde v(0),
$$
where the equality follows by substituting $s=0$ and taking $n\to\iy$ in the definition of $(\tilde v^n,\tilde m^n)$. By Lemma \ref{lem:duality} applied to $(\tilde m,\tilde v)$, $\tilde m(s)\cdot\tilde v(s)$ is non-increasing with $s$. Together with the last bound, it follows that $
\tilde m(s)\cdot\tilde v(s)$ is constant, and so equals $\rho_\iy$. By duality again, it follows that $\Deltadot \tilde v(s)=0$ for any $s\le 0$, which together with $m(s)\in\calM_0$ implies that $\rho_\iy=\tilde m(s) \cdot \tilde v(s) = \tilde m(s) \cdot \Delta \tilde v(s) =0$.

We now turn to establishing (b). From part (a) we have that for any $\eps>0$, there is $T_0>0$ and $\bar r_0>0$ such that for any $t>T_0$ and $r\in(0,\bar r_0)$, $\rho^r(t) \leq \eps$.

Combining it with the duality lemma, \eqref{bound_0}, and \eqref{star}, we get that for any $t_2 > t_1 >T_0$ and $r\in (0,\bar r_0)$ and solutions $(v,m)$ to \eqref{lin_MFG_r_around_stat} (associated with $r$):
\begin{align}\notag
\int_{t_1}^{t_2}e^{-rs}|\Deltadot v(s)|^2ds\le C\eps|m_0|^2. 
\end{align}
Reestimating $m$ on $[T_0,\iy)$, using Lemma \ref{lem:M0}, \eqref{bound_m_Del_v}, and the above, we get
\begin{align*}
|m(T_0+t_1)|
&\le Ce^{-ct_1}|m(T_0)|+C\Big(\int_{T_0}^{T_0+t_1}|\Deltadot v(s)|^2ds\Big)^{1/2}\\\notag
&\le Ce^{-ct_1+rT_0/2}|m(0)|+Ce^{r(T_0+t_1)/2}\Big(\int_{T_0}^{T_0+t_1}e^{-rs}|\Deltadot v(s)|^2ds\Big)^{1/2}\\\notag
&\le Ce^{r(T_0+t_1)/2}|m(0)|\Big(e^{-(c+r/2)t_1}+\eps^{1/2}\Big).
\end{align*}
Take $t_1$ such that $Ce^{-ct_1}\le 1/4$, as $C$, it is independent of $\eps$ and $r$. Also, take $\eps$ such that $C\eps^{1/2}\le 1/4$. Denote $\tau:=T_0+t_1$, so we have
\begin{align}\label{bound_m_tau}
|m(\tau)|\le \frac{1}{2}|m_0|e^{r\tau /2}. 
\end{align}
Define, $(v^\tau(t),m^\tau(t)):=(v(\tau+t),m(\tau+t))$, $t\in\R_+$. This is a solution to \eqref{lin_MFG_r_around_stat} with the initial condition $m^\tau(0)=m(\tau)$. By \eqref{star},
\begin{align*}
e^{-rt} m^\tau (t)\cdot  v^\tau (t)\le |m(\tau)|^2\rho^r(t).
\end{align*}
Therefore, by \eqref{bound_m_tau},
\begin{align}\notag
e^{-r(t+\tau)}m(t+\tau)\cdot v(t+\tau)\le e^{-r\tau}|m(\tau)|^2\rho^r(t) \le \frac{1}{4} |m_0|^2 \rho^r(t).
\end{align}
Taking supremum over $\{m_0 : |m_0|\le 1\}$ on both sides, we get that for any $t\in\R_+$,
\begin{align}\notag
\rho^r(t+\tau)\le \frac{1}{4}\rho^r(t).
\end{align}
Since this can be established for any $\tau$ large enough (independently of $\eps$ and $r$), we obtain that (b) holds.

Finally, we turn to proving (c). We use Lemma \ref{lem:M0} on $[t/2,t]$, Lemma \ref{lem:duality}, and \eqref{rho_exp_bound}, we obtain that for any $t\in\R_+$,
\begin{align*}
|m(t)|
&\le 
Ce^{-ct/2}|m(t/2)|+C\big(\int_{t/2}^t|\Deltadot v(s)|^2ds\Big)^{1/2}\\
&\le
Ce^{-ct/2+rt/2}|m(0)|+Ce^{rt/4}\big(\int_{t/2}^te^{-rs}|\Deltadot v(s)|^2ds\Big)^{1/2}\\
&\le 
C|m(0)|\Big(e^{-ct/2+rt/2}+e^{rt/4-\gamma t/2}\Big).
\end{align*}
Take $r$ sufficiently small to obtain
\begin{align}\notag
|m(t)|\le C|m(0)|e^{-\lambda t},\qquad t\in\R_+,
\end{align}
for some $\lambda\in(0,c)$. Modifying $\lambda$ if necessary, by an application  of \eqref{bound_m_Del_v} on the interval $[t/2,\iy)$, one gets,
\begin{align}\notag
|\Deltadot v(t)|\le C|m(t/2)|e^{rt/4}\le C|m_0|e^{-\lambda t}.
\end{align}

\end{proof}

\begin{lemma}\label{lem:lin_MFG_r_exp_A_B}
Consider the system \eqref{lin_MFG_r_around_stat} and assume that there exist constants $C_1,\gamma>0$, such that for any $t\in\R_+$,
\begin{align}\label{bound_A_B}
|A(t)|+|B(t)|\le  C_1 e^{-\gamma t}.
\end{align}
Then, there exists a constant $C_2>0$ independent of $C_1, r,$ and $m_0$ such that for every $r\in(0,r_0)$ and every $t\in \R_+$,
\begin{align}\notag
|m(t)|+|\Deltadot v(t)|\le C_2 C_1 (1+|m_0|+t)e^{-(\lambda \wedge \gamma) t},
\end{align}
where $r_0,\lambda>0$ are the ones from Lemma \ref{lem:lin_MFG_r_exp}.
\end{lemma}

\begin{proof} We start with showing that one can reduce the problem to one with $A\equiv 0$. To this end,
let $(v^A,m^A)$ be the unique solution to \eqref{lin_MFG_r_around_stat} with some given $A$ and $B$ satisfying \eqref{bound_A_B}. Also, let $\check v$ be the unique (bounded) solution to 
\begin{align}\notag
-\frac{d}{dt}\check v_x(t)=-r\check v_x(t) + \gamma^*(x,\Delta_x \bar u^r)\cdot \Delta_x \check v(t) +A_x(t) \qquad \forall t\in\R_+.
\end{align}
By Lemma \ref{lem:v} with $T\to\iy$, there exist $c,C>0$ depending only on $\mathfrak{a}_l$ and $\mathfrak{a}_u$, such that
\begin{align}\label{bound_check_v}
|\Deltadot \check v (t)|\le C\int_t^\iy e^{-(c+r)(s-t)}|A(s)|ds\le 
\frac{C}{c+\gamma+ r}C_1e^{-\gamma t}   \qquad \forall t\in\R_+.
\end{align}
Now, $(v^A-\check v,m^A)$ solves \eqref{lin_MFG_r_around_stat} with $A\equiv 0$ and $\check B_x(t)=B_x(t)+\sum_{y\in[d]}\bar\mu^r_y\nabla_p\gamma^*_x(y,\Delta_y \bar u^r)\cdot\Delta_y \check v(t)$.
Trivially, $\check B(t)\in\calM_0$, recall from Proposition \ref{prop:statDisc_erg_mfg} that $\Delta_x \bar u^r$ are uniformly bounded for $r\in(0,r_0)$, and by assumption the second derivative of $H$ is bounded on any compact set and equals $\nabla_p \gamma^*$. Therefore by updating $C_1$, 
\begin{align}\label{eqn:checkB0}
|\check B(t)| \le C_1 e^{-\gamma t}. 
\end{align}

For any $b>0$, define the set of measurable functions: 
$$\calB_b:=\{B:\R_+\to\calM_0 \;|\;\forall t\; | B(t)| \le  b e^{-\gamma t}\}. $$
We prove that for any $c_1>0$ there exists $C_2>0$ independent of $c_1,r,m_0$, such that 
\begin{align}\label{bound_v_m}
\sup_{B\in\calB_{c_1}}\left\{|m(t)|+|\Deltadot v(t)|\right\}\le C_2c_1(1+|m_0|+t)e^{-(\lambda\wedge \gamma) t},
\end{align}
where $(v,m)$ solves \eqref{lin_MFG_r_around_stat} with $A\equiv 0$ and some $B\in\calB_{c_1}$. 
Recalling that $(v^A-\check v,m^A)$ solves \eqref{lin_MFG_r_around_stat} with $A\equiv 0$, and the constant $C_1$ from \eqref{eqn:checkB0}, then from \eqref{bound_v_m}, we obtain that 
\begin{align*}
|m^A(t)|+|\Deltadot (v^A-\check v)(t)|\le C_2 C_1(1+|m_0|+t)e^{-(\lambda\wedge \gamma) t},
\end{align*}
Together with \eqref{bound_check_v}, we obtain that 
\begin{align*}
|m^A(t)|+|\Deltadot v^A(t)|\le 2C_2 C_1(1+|m_0|+t)e^{-(\lambda\wedge \gamma) t},
\end{align*}

The rest of the proof is dedicated to establish \eqref{bound_v_m} (which by definition considers $A\equiv 0)$. Exploiting the linearity of the system with respect to $m_0$ and using \eqref{bound_m_Del_v_2}, we may establish \eqref{bound_v_m} for the case $m_0=0$.

We start with some preliminary estimates. From Lemma \ref{lem:duality}, there exists a constant $C_3>0$ (which can change from one line to the next) independent of $C_1,r,m_0$, such that for any $0\le t_1<t_2<\infty$, 
\begin{align*}
    C_3^{-1}\int_{t_1}^{t_2}e^{-rs}|\Delta v(s)|^2ds\le -C_3 \Big[e^{-rs}m(s)\cdot v(s)\Big]^{t_2}_{t_1}+C_3\int_{t_1}^{t_2}e^{-rs}|B(s)|^2ds.
\end{align*} The same arguments leading to \eqref{lim_mv} imply $\lim_{t\to\iy}e^{-rt}m(t)\cdot v(t)=0.$ Using the bound on $B\in\calB_{c_1}$ and since $m_0=0$, we get, 
\begin{align*}
    \int_{0}^{\infty}e^{-rs}|\Delta v(s)|^2ds\le C_3c_1.
\end{align*}
From Lemma \ref{lem:M0} and $m_0=0$,
\begin{align*}
    |m(t)|
    &\le C_3\Big(\int_{0}^{t}\left\{|\Delta v(s)|^2+|B(s)|^2\right\}ds\Big)^{1/2}\\
    &\le C_3e^{rt/2}\Big(\int_{0}^{t}e^{-rs}\left\{|\Delta v(s)|^2+|B(s)|^2\right\}ds\Big)^{1/2}\\
    &\le C_3c_1 e^{rt/2}.
\end{align*}
Applying Lemma \ref{lem:v} and taking $T\to\iy$, we obtain 
\begin{align*}
    e^{-rt}|\Delta v(t)|\le C_3\int_t^\iy e^{-c(s-t)}|m(s)|e^{-rs}ds\le C_3c_1e^{-rt/2}.
\end{align*}
Altogether, 
\begin{align}\label{eq:mv_rt2}
    |m(t)|+|\Delta v(t)|\le C_3c_1e^{rt/2}.
\end{align}

For any $b>0$, set 
\[
\bar\rho^r(b,t) := b^{-1}\sup_{B\in\calB_b} e^{-rt} (|m(t)| + |\Delta v(t)|).
\] where $(v,m)$ solves \eqref{lin_MFG_r_around_stat} with $A\equiv 0$, $m(0)=0$, and some $B\in\calB_{c_1}$. 
The structure of \eqref{lin_MFG_r_around_stat} implies that if $(\tilde v, \tilde m)$ is a solution with a given $B$, then $(\ell \tilde v, \ell \tilde m)$ is the solution associated to $\ell B$ for any $\ell \in\R$. Therefore, for any $b>0$, 
\begin{align}\label{eq:rho_c}
    \bar\rho^r(t):=\bar\rho^r(1,t)
    =\bar\rho^r(b,t),\qquad t\in\R_+ .
\end{align} Fix $\tau >0$ and let $(v^{\tau,1},m^{\tau,1})$ be the solution of \eqref{lin_MFG_r_around_stat} on $[\tau,\iy)$ with $A\equiv B\equiv 0$ and $m^{\tau,1} (\tau) = m(\tau)$. Also let $(v^{\tau,0},m^{
\tau,0})$ be the solution to \eqref{lin_MFG_r_around_stat} on $[\tau,\iy)$ with $A\equiv 0$, $m^{\tau,0}(\tau)=0$, and the given $B$. Then on the time interval $[\tau,\iy)$, $(v,m)=(v^{\tau,1},m^{\tau,1})+(v^{\tau,0},m^{\tau,0})$. 

Note that $(\hat v^{\tau,0}(\cdot),\hat m^{\tau,0}(\cdot)):=(v^{\tau,0}(\tau+\cdot),m^{\tau,0}(\tau+\cdot))$ solves \eqref{lin_MFG_r_around_stat} on $\R_+$. By the assumption on $B$, $B(\tau + t) \le c_1 e^{-\gamma t} e^{-\gamma \tau}$. Together with the definition of $\bar\rho^r(c,t)$ and the aforementioned linear property, we have
\[
c_1^{-1}e^{-rt} (|m^{\tau,0} (\tau + t)| + |\Delta v^{\tau ,0 }(\tau +t)|) = e^{-\gamma \tau}(c_1e^{-\gamma\tau})^{-1}e^{-rt} (|\hat m^{\tau,0} ( t)| + |\Delta \hat v^{\tau ,0 }( t)|) \leq e^{-\gamma \tau}\bar\rho^r( t).
\] 
On the other hand, Lemma \ref{lem:lin_MFG_r_exp} and \eqref{eq:mv_rt2} imply that there exist $\lambda, C_4>0$, independent of $c_1,r,m_0$, where $C_4$ may change from one line to the next, such that for all $r\in (0,r_0)$ and for all $t\in \R_+$,
\[
|m^{\tau,1} (\tau +t)| + |\Delta v^{\tau,1} (\tau+t)| \le C_4e^{-\lambda t} |m(\tau)| \le C_4c_1e^{-\lambda t} e^{r\tau /2}. 
\] 
As a consequence,
\begin{align}\notag
c_1^{-1}e^{-r(t+\tau)}\left(|m(\tau+t)|+|\Deltadot v (\tau+t)|\right)
\le 
C_4 e^{-(\lambda+r)t}e^{-r\tau/2}+ e^{-(\gamma+r)\tau}\bar\rho^r(t).
\end{align}
Taking the supremum over $B$ satisfying $|B(t)| \leq c_1 e^{-\gamma t}$, one gets
\begin{align}\notag
\bar\rho^r(\tau+t)\le C_4e^{-(\lambda+r)t}+e^{-(\gamma+r)\tau}\bar\rho^r(t).
\end{align} Multiply both sides by $e^{((\lambda\wedge\gamma) +r)(t+\tau)}$, one gets,
\begin{align}\notag
e^{((\lambda\wedge\gamma)+r)(t+\tau)}\bar\rho^r(\tau+t)\le Ce^{((\lambda\wedge\gamma)+r)\tau}+e^{((\lambda\wedge\gamma)+r)t}\bar\rho^r(t),
\end{align}
which, together with \eqref{eq:rho_c} implies the exponential decay in \eqref{bound_v_m}. Indeed, set $\theta:=\lambda\wedge \gamma$ and $\beta(t):=e^{(\theta+r)t}\bar \rho^r(t)$. Then, $\beta(t+\tau)\le Ce^{(\theta +r)\tau}+\beta(t)$. First, take $t$ to be a non-negative integer and $\tau=1$, then,
$$\beta(t)\le Cte^{\theta +r}+\beta(0)
.$$
Then, take $\tau\in(0,1)$ and:
$$
\beta(\tau+t)\le Ce^{(\theta+r)\tau}+ Cte^{\theta +r}+\beta(0)\le  C(t+1)e^{\theta +r}+\beta(0).
$$
So in general, for any $s\in\R_+$,
$$\beta(s)\le C(s+1)e^{\theta+r}+\beta(0).
$$ Now, recall the definition of $\beta$ to get the desired decay.
\end{proof}

\section{Proof of Proposition \ref{prop:MFG}}\label{sec:proof_prop_MFG}
\subsection{Proof of (i)}
Fix $0<T<\TT<\iy$. Before establishing the existence of a solution to \eqref{MFG^r} via the limit of \eqref{MFG^r_T} from Section \ref{sec:useful_lemmas}  as $T\to\iy$, we establish two bounds: one for $|\mu^T(t) - \mu^{\tilde T}(t)|$ and the other for $|\Delta (u^T(t) - u^{\tilde T}(t))|$, where $(u^T,\mu^T)$ and $(u^{\TT},\mu^{\TT})$ solve \eqref{MFG^r_T} on the finite horizons $T$ and $\TT$, with initial data $\mu_0$ and $\tilde\mu_0$, respectively. These bounds serve us in order to establish the existence of a solution to \eqref{MFG^r} by taking $\mu_0 = \tilde\mu_0$ and showing that $\{(u^{\TT},\mu^{\TT})\}_T$ is a Cauchy sequence. Later, we use these bounds for general $\mu_0$ and $\tilde \mu_0$ in order to prove \eqref{sol_ineq}. Throughout, $C>0$ is a constant that is independent of $r,$ $\mu_0$, $\tilde\mu_0$, and may change from one line to the next. 

The next sequence of inequalities are derived as follows. We apply Lemma \ref{lem:M0} to $\mu^T - \mu^{\TT}$; triangle inequality and Lipschitz continuity of $\gamma^*$; Jensen's inequality; the fact that $\mu^{\TT}_y \leq \mu^{\TT}_y + \mu^T_y$; the value-measure duality Lemma \ref{lem:val-meas-duality}, and finally the fact that $\mu-\tilde\mu \in\calM_0$ and Lemma \ref{lem:b} to get that for any $t\in [0,T]$:
\begin{align*}
    |(\mu^T - \mu^{\TT})(t)| &\leq Ce^{-ct}|\mu_0 - \tilde\mu_0| + C\int_0^t \big|\sum_{y\in [d]} \mu_y^{\TT} (s) [\gamma^*_x(y,\Delta_y u^T (s)) - \gamma^*_x(y,\Delta_y u^{\TT} (s))]\big|_1 ds \\ \notag
    &\leq Ce^{-ct} |\mu_0-\tilde\mu_0| + C\int_0^t \sum_{y\in [d]} \mu_y^{\TT} (s) |\Delta_y (u^T - u^{\TT})(s)|_1 ds \\\notag
    &\leq Ce^{-ct} |\mu_0-\tilde\mu_0| + Ce^{rt/2} \Big(\int_0^t e^{-rs} \sum_{y\in [d]} \mu_y^{\TT} (s) |\Delta_y (u^T - u^{\TT})(s)|^2 ds\Big)^{1/2} \\\notag
    &\leq Ce^{-ct} |\mu_0-\tilde\mu_0| + Ce^{rt/2} \Big(\int_0^t e^{-rs} \sum_{y\in [d]}(\mu^T_y(s)+ \mu_y^{\TT} (s)) |\Delta_y (u^T - u^{\TT})(s)|^2 ds\Big)^{1/2} \\\notag
    &\leq Ce^{-ct} |\mu_0-\tilde\mu_0| + Ce^{rt/2} \Big(\int_0^T e^{-rs} \sum_{y\in [d]}(\mu^T_y(s)+ \mu_y^{\TT} (s)) |\Delta_y (u^T - u^{\TT})(s)|^2 ds\Big)^{1/2} \\\notag
    &\leq Ce^{-ct} |\mu_0-\tilde\mu_0| + Ce^{rt/2}|\mu_0 - \tilde\mu_0|^{1/2} |\Delta (u^T -u^{\TT})(0)|^{1/2} \\\notag
    &\qquad + Ce^{r(t-T)/2}|\mu^T(T) - \mu^{\TT}(T)|^{1/2} |\Delta (u^T -u^{\TT})(T)|^{1/2}\numberthis\label{finhor_mu_bound_1}
\end{align*} By the mean value theorem, there exists a measurable function $q:[0,T] \to \R^d$ such that:
\begin{align}\label{eqn:dif_ode}
-\frac{d}{dt}(u^T_x - u^{\TT}_x)(t) &= -r(u^T_x - u^{\TT}_x)(t) + H(x,\Delta_x u^T (t)) - H(x,\Delta_x u^{\TT}(t))\\ \notag
&\qquad+ F(x,\mu^T (t)) - F(x,\mu^{\TT}(t))  \\ \notag
&= -r(u^T_x - u^{\TT}_x)(t) + \gamma^*(x,q(t))\cdot \Delta_x(u^T -u^{\TT})(t)\\
&\qquad + F(x,\mu^T (t)) - F(x,\mu^{\TT}(t)).  \notag
\end{align} So we can apply Lemma \ref{lem:v}, the Lipschitz continuity of $F$, and \eqref{finhor_mu_bound_1} to find that for $t\in [0,T]$:
\begin{align}\label{finhor_delu_bound_1}
    |\Delta &(u^T - u^{\TT}) (t)| \\\notag
    &\leq Ce^{-c(T-t)} |\Delta (u^{\TT} - u^{T})(T)| \\ \notag 
    &\qquad + Ce^{rt}e^{ct} \int_t^T e^{-rs} e^{-cs} |F(x,\mu^T(s)) - F(x,\mu^{\TT}(s))|ds \\ \notag
    &\leq Ce^{-c(T-t)} |\Delta (u^T - u^{\TT}) (T)| + Ce^{-ct}|\mu_0 - \tilde\mu_0| + Ce^{rt/2}|\mu_0-\tilde\mu_0|^{1/2} |\Delta (u^T - u^{\TT})(0)|^{1/2} \\ \notag
    &\qquad + Ce^{r(t-T)/2}|\mu^T(T)-\mu^{\TT}(T)|^{1/2} |\Delta (u^T - u^{\TT})(T)|^{1/2}.
\end{align}

We now split the analysis into two cases. In Case 1 we consider  $\mu_0=\tilde\mu_0$, while in case 2, we allow for these initial conditions to differ. The first case is used to establish the existence of a solution to \eqref{MFG^r}. The results in the second case are used in order to establish the  uniqueness of a solution to \eqref{MFG^r}. The reason why we consider different initial conditions is because we also use these results in order to prove part (iii).

{\bf Case 1: $\mu_0 = \tilde\mu_0$.} In this case, \eqref{finhor_mu_bound_1} and \eqref{finhor_delu_bound_1} become:
\begin{align*}
    |(\mu^T- \mu^{\TT})(t)| &\leq Ce^{r(t-T)/2} |\mu^T(T) - \mu^{\TT}(T)|^{1/2} |\Delta (u^T - u^{\TT}) (T)|^{1/2},\\
    |\Delta (u^T - u^{\TT}) (t)| &\leq Ce^{-c(T-t)}|\Delta (u^T - u^{\TT}) (T)| + Ce^{r(t-T)/2} |\mu^T (T) - \mu^{\TT}(T)|^{1/2} |\Delta (u^T - u^{\TT}) (T)|^{1/2}. 
\end{align*} Now, fix $T_0>0$. The above bounds together with \eqref{bound_u} imply that
\[
\sup_{t\in [0,T_0]}\big( |\Delta (u^T - u^{\TT})(t)| + |(\mu^T- \mu^{\TT})(t)|\big) \to 0 \quad\text{ as }\quad T,\TT\to\iy,
\] where we used $u^{(T)}(T)=0$ and the uniform (in $T$ and $\tilde T$) bound from $u^{(\tilde T)}(T)$ from \eqref{bound_u}. Using \eqref{eqn:dif_ode}, the dominated convergence theorem, and passing the limit, 
\[
\sup_{t\in [0,T_0]} \big( u^T(t) - u^{\TT}(t)\big)\to 0.
\] So the functions $u^T$ and $\mu^T$ converge uniformly to some limiting functions $u$ and $\mu$. By the structure of \eqref{MFG^r_T}, 
\[
    \mu_x^T(t) - \mu_x(0) = \int_0^t \sum_{y\in [d]} \mu_y^T(s) \gamma^*_x(y,\Delta_y u^T (s)) ds,\qquad t\in[0,T_0].
\] Passing the limit as $T\to\iy$ and using the dominated convergence theorem yields:
\[
    \mu_x(t) - \mu_x(0) = \int_0^t \sum_{y\in [d]} \mu_y(s) \gamma^*_x(y,\Delta_y u (s)) ds,\qquad t\in[0,T_0].
\] 
Since $T_0$ is arbitrary, we obtain that the above holds on $\R_+$. 
A similar computation for $u$ shows that $(u,\mu)$ satisfies \eqref{MFG^r}. 

{\bf Case 2: $\mu_0$ and  $\tilde\mu_0$ may differ.} Let $u$ be the limit of $u^T$ with initial data $\mu_0$ and $\tilde u$ the limit $u^{\TT}$ with initial data $\tilde\mu_0$. Recall once more that $\{|u^T(t)|\}_{t,T}$ is bounded by \eqref{bound_u}. Returning to \eqref{finhor_delu_bound_1},  put $t=0$, take the limit as $T,\TT\to\iy$, to get:
\begin{align*}
    |\Delta (u - \tilde u)(0)| &\leq C|\mu_0 - \tilde\mu_0| + C|\mu_0 -\tilde\mu_0|^{1/2} |\Delta (u - \tilde u)(0)|^{1/2}.
\end{align*} Applying Young's inequality to the rightmost term and rearranging terms,
\begin{align}\label{finhor_delu_0}
    |\Delta (u - \tilde u)(0)| &\leq C|\mu_0 - \tilde\mu_0|.
\end{align} Plugging \eqref{finhor_delu_0} into \eqref{finhor_mu_bound_1} and taking the limits $T,\TT\to\iy$, this implies
\begin{align}\label{mu_bound_good}
    |(\mu - \tilde\mu)(t)| \leq Ce^{rt/2}|\mu_0 - \tilde\mu_0|. 
\end{align} Similarly, applying \eqref{mu_bound_good} and \eqref{finhor_delu_0} to \eqref{finhor_delu_bound_1} after taking the limits $T,\TT\to\iy$ gives:
\begin{align}\label{delu_bound_good}
    |\Delta (u-\tilde u)(t)| \leq Ce^{rt/2}|\mu_0 - \tilde\mu_0|. 
\end{align} Note that the $C$ in \eqref{mu_bound_good} and \eqref{delu_bound_good} is still independent of $r>0$. Using \eqref{MFG^r}, we can integrate the ODE that $e^{-rt}(u-\tilde u)(t)$ satisfies, and use \eqref{mu_bound_good} and \eqref{delu_bound_good} to get: 
\[
|(u-\tilde u)(t)| \leq Cr^{-1} e^{rt/2} |\mu_0 - \tilde\mu_0|.
\] The above display and \eqref{mu_bound_good} together imply uniqueness.

\subsection{Proof of (ii).}
By the previous section, given some initial data $\mu_0$, \eqref{MFG^r} has a unique solution denoted by $(u^r,\mu^r)$. We now provide an exponential rate of convergence of the discounted MFG system \eqref{MFG^r} to the stationary discounted MFG system \eqref{stat_MFG_r}. This result is equivalent to \cite[Theorem 3.7]{car-por}.

Let $(u^r,\mu^r)\in \calC^1(\R_+,\R^d) \times \calC^1(\R_+,\calP([d]))$ solve \eqref{MFG^r} with the initial data $\mu^r(0) = \mu_0^r \in\calP([d])$ and let $(\bar u^r, \bar\mu^r)\in\R^d \times\calP([d])$ be the solution to \eqref{stat_MFG_r}. For any $k,L>0$, denote the sets of functions:
\begin{align*}
E_{k} &:= \{(v,m):\R_+ \to \R^d\times \calM_0 \mid |\Delta v(t)| + |m(t)| \le ke^{-\gamma t}\quad  \forall t\in \R_+ ; \quad m(0) = \mu_0^r - \bar\mu_0^r\},
\end{align*}
\begin{align*}
    E_{k,L} := \{(v,m) \in E_k \mid |(d/dt) v(t)| + |(d/dt) m(t)| \leq L \quad \forall t\in\R_+\}.
\end{align*} Recall from Proposition \ref{prop:statDisc_erg_mfg} that $\bar\mu^r$ is bounded away from $0$, uniformly in $r\in (0,r_0)$. Choose $k\in (0,1)$ satisfying $0<k<\bar\mu_y^r$ for all $y\in [d]$, and for now we assume that $\tilde m_0 := \mu_0^r - \bar\mu_0^r$ is such that $|\tilde m_0| \le k^2$. The latter condition will eventually be relaxed. Fix $(v,m) \in E_{k,L}$. The lower bound $k$ implies that $\bar\mu_y^r + m_y(t) \geq 0$ for all $y\in [d]$, which together with $m(t) \in \calM_0$ implies $\bar\mu_y^r + m(t) \in\calP([d])$. 

Consider the solution $(\tilde v,\tilde m)\in\calC^1(\R_+,\R^d)\times \calC^1(\R_+,\calM_0)$ to the linear system \eqref{lin_MFG_r_around_stat} with the data: 
\begin{align*}
A_x(t) &:= H(x,\Delta_x \bar{u}^r) - H(x,\Delta_x (\bar{u}^r+v)) + \gamma^*(x,\Delta_x \bar{u}^r) \cdot \Delta_x v\\ 
&\qquad+ F(x, \bar{\mu}^r + m) - F(x,\bar{\mu}^r) - D^\eta_1 F(x,\bar{\mu}^r)\cdot m,\\
B_x(t) &:= \sum_{y\in [d]}(\bar{\mu}^r +m)_y\Big[\gamma^*_x(y,\Delta_y(\bar{u}^r + v)) - \gamma^*_x(y,\Delta_y \bar{u}^r)\Big] - \bar{\mu}^r_y\nabla_p \gamma^*_x(y, \Delta_y \bar{u}^r)\cdot \Delta_y v,
\end{align*} where the time parameter $t$ is omitted for brevity. Due to this choice of $(v,m)$, Taylor's expansion, the conditions on $H,\gamma^*$ (recall \eqref{gamma_H}), and $F$, and the definition of $E_k$, there exists $C>0$ depending on $\gamma^*, H$, and $d$ such that for any $x\in [d]$ and $t\in\R_+$,
\[
|A_x(t)| + |B_x(t)| \le Ck^2 e^{-2\gamma t},
\] So we apply Lemma \ref{lem:lin_MFG_r_exp_A_B} to get:
\[
|\tilde m(t)| + |\Delta \tilde v(t)|\le Ck^2 (1+|\tilde m_0|+t)e^{-\theta t}, 
\] where $\theta:=2\gamma \wedge \lambda$ and with $\lambda$ as in Lemma \ref{lem:lin_MFG_r_exp}. Since $|\tilde m_0|\le k^2\leq 1$, we can choose $\gamma \in(0,\lambda)$ so that for all $t\in\R_+$:
\[
|\tilde m(t)| +|\Delta \tilde v(t)| \le Ck^2 e^{-\gamma t}. 
\] Given this decay on $(\tilde v, \tilde m)$, we can define a map $\iota : E_k \to E_{Ck^2}$, which sends a choice of $(v,m)$ to the associated solution $(\tilde v, \tilde m)$. Note that there is a constant $\tilde L>0$ such that for all $t\in\R_+$: 
\[
\left|\frac{d}{dt} \tilde v(t)\right| + \left|\frac{d}{dt} \tilde m(t)\right| \le \tilde L.
\] Therefore, the map $\iota : E_{k,\tilde L} \to E_{Ck^2,\tilde L}$ is compact and by the Schauder fixed point theorem implies that, taking $k$ small enough, it has a fixed point, which we denote by $(v^r,m^r)$. Then, $(u^r,\mu^r) := (\bar{u}^r, \bar{\mu}^r) + (v^r,m^r)$ satisfies \eqref{MFG^r} and $(v^r,m^r)$ satisfies the exponential decay condition. This means that for all $t\in\R_+$:
\[
|\Delta (u^r(t) - \bar{u}^r)| + |\mu^r(t) - \bar{\mu}^r| \leq C e^{-\gamma t}.
\]

In the following section we will show exponential decay for $|u^r(t) -\hat u^r(t)|$, the difference of two solutions with different initial data. In particular, this implies exponential decay $|u^r(t) - \bar u^r|$ (by taking $\tilde\mu_0=\bar\mu$), a stronger assertion than what is proved in this section, but what appears in the proposition. Note that the following section does not require this stronger bound and the bound is only updated after the proof in the following section; hence, we avoid circularity.

We now relax the restriction on $\tilde m_0$. We need to show that there exists $T>0$ such that for all $\mu_0^r \in \calP([d])$, the associated solution $(u^r,\mu^r)$ to \eqref{MFG^r} satisfies $|\mu^r(t) - \bar\mu^r| \le k^2<1$. Then, we can apply the previous result on $[T,\iy)$. Fix $\delta>0$ such that $\bar\mu^r_x > \delta$ for all $x\in [d]$ and $r\in (0,r_0)$. By Lemma \ref{lem:val-meas-duality}, the bound for $\bar\mu^r_x>\delta>0$, and since $\mu_0^r - \bar\mu^r\in\calM_0$,
\begin{align*}
C_{2,H} \int_0^\iy e^{-rt} |\Delta (u^r(t) -\bar u^r)|^2 dt &\leq C_{2,H}\delta^{-1} \int_0^\iy e^{-rt} \sum_{x\in [d]}|\Delta_x (u^r(t) -\bar u^r)|^2(\mu_x^r(t) + \bar\mu^r_x) dt\\
&\le \delta^{-1}(\mu_0^r - \bar\mu^r) \cdot (u^r(0) -\bar u^r)\\
&=  \delta^{-1}(\mu_0^r - \bar\mu^r) \cdot (u^r(0) - \langle u^r(0)\rangle -\bar u^r + \langle \bar u^r\rangle).
\end{align*} By Lemmas \ref{lem:b} and \ref{lem:v}, and using the previous line, there exists $C_1>0$ independent of $r$, such that:
\begin{align}\label{bound_dual}
\delta C_{2,H} \int_0^\iy e^{-rt} |\Delta (u^r(t) -\bar u^r)|^2 dt \leq C_1.
\end{align} Set $m^r(\cdot) := \mu^r(\cdot)-\bar\mu^r$ and note that $m^r$ satisfies:
\[
\frac{d}{dt} m_x^r(t) = \sum_{y\in [d]} m^r_y(t) \gamma^*_x(y,\Delta_y u^r(t)) + \hat B_x(t),
\] where
\[
\hat B_x(t) := \sum_{y\in [d]} \bar\mu^r_y [\gamma^*_x(y,\Delta_y u^r(t))-\gamma^*_x(y,\Delta_y \bar u^r)].
\] Fix $t\geq t_1 > 0$. By Lemma \ref{lem:M0}, there exist $c,C_2>0$, depending only on $\mathfrak{a}_l$ and $\mathfrak{a}_u$, such that:
\[
|\mu^r(t) - \bar\mu^r| \le C_2 e^{-c(t-t_1)}|\mu^r(t_1) - \bar\mu^r| +C_2 e^{rt/2} \Big[\int_{t_1}^t e^{-rs} |\Delta (u^r(s) - \bar u^r)|^2\Big]^{1/2}.
\] 
By \eqref{bound_dual} there exist a sequence $\{T^i\}_{i\in\N}\to\iy$, and $t_1^i\in [0,T^i]$ and $t_2^i\in [3 T^i,4T^i]$, such that: 
\[
e^{-rt_j^i} |\Delta (u^r(t_j^i)-\bar u^r)|^2 \le \frac{C_1}{T^i},\quad j=1,2,\quad i\in\N.
\] Otherwise, the integral would be greater than $2C_1/\sqrt{\delta C_{2,H}}$ and so would contradict \eqref{bound_dual}. Fix (for now) an arbitrary $i\in\N$ and denote $T=T^i$ and $t_j=t^i_j$, $j=1,2$. From Lemma \ref{lem:duality}, we get: 
\begin{align*}
    \delta C_{2,H} \int_{t_1}^{t_2} e^{-rt} |\Delta (u^r(t) -\bar u^r)|^2 dt &\le e^{-rt_1} |\Delta (u^r(t_1) -\bar u^r)| |\mu^r (t_1) -\bar\mu^r| \\
    &\qquad + e^{-rt_2} |\Delta (u^r(t_2) -\bar u^r)| |\mu^r (t_2) -\bar\mu^r| \\
    &\le \frac{C}{\sqrt{T}}.
\end{align*} As $t_1\le T< 3T \le t_2 \le 4T$, we get for all $t\in [2T,t_2]$,
\begin{align*}
    |\mu^r(t) - \bar\mu^r| &\le Ce^{-c(2T-t_1)} |\mu^r(t_1) -\bar\mu^r| + Ce^{rt_2/2} \Big[\int_{t_1}^{t_2} e^{-rs} |\Delta (u^r(s) - \bar u^r)|^2 ds \Big]^{1/2} \\
    &\le Ce^{-cT} + Ce^{2rT} T^{-1/4}.
\end{align*} Choosing $T$ large and $r$ sufficiently small therefore yields $|\mu^r(t) - \bar\mu^r|$ sufficiently small for any $t\in [2T,3T]$. 
\qed

\subsection{Proof of (iii).}\label{sec:prop223}
Set $v:= u^r - \tilde u^r$ and set $m:= \mu^r -\tilde\mu^r$ where $(u^r, \mu^r)$ and $(\tilde u^r, \tilde \mu^r)$ are solutions to \eqref{MFG^r} with the initial data $\mu_0$ and $\tilde\mu_0$, respectively. By the mean value theorem, there exist measurable functions $q:=q^{u^r,\tilde u^r}$, $z:= z^{\mu^r, \tilde \mu^r}$, and $\ell := \ell^{u^r,\tilde u^r}$, such that:
\begin{align*}
&H(x, \Delta_x u^r(t)) - H(x,\Delta_x \tilde u^r (t)) = \gamma^*(x, q(t)) \cdot \Delta_x v(t), \\
&F(x,\mu^r(t)) - F(x,\tilde \mu^r(t)) = D^\eta_1 F(x, z(t)) \cdot m(t), \\
&\gamma^*_x(y,\Delta_y u^r(t)) - \gamma^*_x(y,\Delta_y \tilde u^r(t)) = \nabla_p \gamma_x(y,\ell(t)) \cdot \Delta_y v(t).
\end{align*} 
We note that for any fixed time $t_0 \in \R_+$, the mean value theorem asserts that there exists $\lambda_{t_0} \in [0,1]$ such that:
\begin{align*}
q(t_0) &= \lambda_{t_0} [\Delta_x u^r(t_0)] + (1-\lambda_{t_0})[\Delta_x \tilde u^r (t_0)]\\
&\in \Big[[\Delta_x u^r(t_0)]_y \wedge [\Delta_x \tilde u^r (t_0)]_y , [\Delta_x u^r (t_0)]_y \vee [\Delta_x \tilde u^r (t_0)]_y \Big]_{y\in [d]}. 
\end{align*} 

Part (ii) above implies that:
\begin{align}\label{eqn:pls_work}
\sum_{x\in[d]}\left(|q_x(t) - \Delta_x \bar u^r|+|z_x(t)-\bar\mu^r_x|+|\ell_x(t)-\Delta_x\bar u^r|\right) \leq Ce^{-\gamma t}.
\end{align} Note that $(v,m)$ solves the system:
\begin{align*}
\begin{cases}
-\frac{d}{dt} v_x(t)= -rv_x(t) + \gamma^*(x,q(t))\cdot \Delta_x v + D^\eta_1 F(x,z(t))\cdot m(t), \\
\frac{d}{dt} m_x(t)=\sum_{y\in[d]}  m_y(t)\gamma^*_x(y,\Delta_y u^r (t)) + \sum_{y\in[d]} \tilde \mu_y^r \nabla_p \gamma^*_x(y,\ell(t))\cdot \Delta_y v(t),
\end{cases} 
\end{align*} with the initial data $m(0) = \mu_0 -\tilde \mu_0$. That is, $(v,m)$ solves \eqref{lin_MFG_r_around_stat} with:
\begin{align}\notag
A_x(t)&:=\Big( \gamma^*(x,q(t))- \gamma^*(x,\Delta_x  \bar u^r)\Big) \cdot \Delta_x v(t) +\Big( D^\eta_1 F(x,z(t))- D^\eta_1 F(x,\bar\mu^r)\Big)\cdot m(t),\\\notag
B_x(t)&:=\sum_{y\in[d]} m_y(t)\big(\gamma^*_x(y,\Delta_y  u^r(t))-\gamma^*_x(y,\Delta_y  \bar u^r)\big)
\\\notag
&\qquad+\Big( \sum_{y\in[d]} \tilde \mu^r_y(t) \nabla_p\gamma^*_x(y,\ell(t))- \sum_{y\in[d]} \bar\mu^r_y \nabla_p\gamma^*_x(y,\Delta_y  \bar u^r)\Big)\cdot \Delta_y v(t).
\end{align} Since $\gamma^*$, $\nabla_p\gamma^*$, and $D^\eta_1 F$ are Lipschitz, we may use the exponential bound \eqref{eqn:pls_work} along with \eqref{mu_bound_good} and $\eqref{delu_bound_good}$, to get that:
\begin{align*}
    |A(t)| + |B(t)| \leq Ce^{(r/2 - \gamma)t} |\mu_0 - \tilde\mu_0|.
\end{align*} By Lemma \ref{lem:lin_MFG_r_exp_A_B} then, we obtain: 
\begin{align}\label{eqn:good_dog}
    |\Delta (u^r-\tilde u^r)(t)| + |(\mu^r - \tilde\mu^r)(t)| \leq Ce^{(r/2 - \gamma)t} |\mu_0 -\tilde\mu_0|.
\end{align} Integrating the ODE \eqref{eqn:dif_ode}, using the Lipschitz continuity of $H$ and of $F$, and applying \eqref{eqn:good_dog} after taking the limits $T,\TT\to\iy$ yields:
\begin{align*}
    e^{-rt}|u^r(t) - \tilde u^r(t)| \leq C|\mu_0 - \tilde\mu_0| \int_t^\iy e^{-rs/2}e^{-\gamma s} ds
\end{align*} which implies
\begin{align}\label{u_bound_ok}
    |u^r(t) - \tilde u^r (t)| \leq Ce^{(r/2-\gamma)t}|\mu_0-\tilde\mu_0|.
\end{align} Together, \eqref{eqn:good_dog} and \eqref{u_bound_ok} finish the proof of the bound. Note that with this result, we can improve the bound from part (ii) of this proposition from: 
\[
    |\Delta (u^r(t) - \bar u^r)| \leq Ce^{-\gamma t} |\mu_0 - \bar\mu^r|,
\] to:
\[
    |u^r(t) - \bar u^r| \leq Ce^{-\gamma t} |\mu_0 - \bar\mu^r|.
\]

\subsection{Proof of (iv)} Now, we show the relationship between the solution to \eqref{MFG^r} and the discounted MFE.
Let $(u^r,\mu^r)$ be the solution to \eqref{MFG^r} where $\mu_0^r = e_{x_0}$, $x_0\in [d]$. Define $\al^r_{xy}(t):= \gamma^*_y(x,\Delta_x u^r(t))$ and $X_t := X^{\al^r}_t$, with $X_0 = x_0\in [d]$ fixed. We note that by Kolmogorov's equation, 
\[
\frac{d}{dt} \PP_{x_0}(X_t = x)=\sum_{y\in [d]} \PP_{x_0}(X_t = y) \al^r_{yx}(t),\qquad t\in \R_+,
\] and so by the uniqueness of it, $\PP_{x_0}(X_t = x) = \mu^r_x(t)$. Moreover, as with deriving \eqref{eqn:u_true_form}, we obtain that:
\[
u^r_x(0) = \EE\Big[ \int_0^\iy e^{-rt} \sum_{y\in [d]} \1_{\{X_{t} = y\}}(f(y,\al^r_y(t)) + F(y,\mu^r(t))) dt \Big] = J_r(\al^r,\mu^r),
\] with equality because $\al^r_{xy}$ is defined to be $\gamma^*_y(x,\Delta_x u^r(t))$, and in addition, for all $\al\in\calA$,
\[
J_r(\al^r,\mu^r) \leq J_r(\al,\mu^r).
\] This implies:
\[
\argmin_{\al\in\calA} J_r(\al,\mu^r) = J_r(\al^r,\mu^r). 
\]

 Conversely, suppose that $(\tilde \al^r,\tilde\mu^r)$ is a discounted MFE and define:
\[
\tilde u^r_x(t) := e^{rt} J_r(t,x,\tilde \al^r,\tilde\mu^r) := e^{rt} \EE\Big[ \int_t^\iy e^{-rs} \sum_{y\in [d]} \1_{\{X_{s^-} = y\}}(f(y,\tilde \al^r_y(s)) + F(y,\tilde \mu^r(s))) ds \Big| X_t = x\Big].
\] 
Note that $J_r$ is the value function of the MFG. Hence, $\tilde u$ satisfies,

\begin{align*}
    -\frac{d}{dt} \tilde u^r_x(t) &= -r \tilde u^r_x(t)+\tilde\al^r_x (t) \cdot \Delta_x \tilde u^r(t) +f(x,\tilde\al^r(t)) + F(x,\tilde\mu^r(t)),\qquad t\in\R_+.
\end{align*} 
To end the proof, it suffices to show that $\tilde \al^r_{xy}(t) = \gamma^*_y(x,\Delta_x \tilde u^r(t))$. Indeed, in this case, the previous display shows that $\tilde u^r$ satisfies the value equation in \eqref{MFG^r}; meanwhile, the fact that $\tilde \mu^r_x(t) = \PP(X_t = x)$ implies: 
\begin{equation}\label{eqn:MFE_tilde_u_value}
\frac{d}{dt} \tilde \mu^r_x(t) = \sum_{y\in [d]} \tilde \mu^r_y(t) \al^r_{yx} (t),\qquad t\in\R_+,  
\end{equation}
and so we would have that $\tilde\mu^r$ satisfies the Kolmogorov equation in \eqref{MFG^r}. Hence, $(\tilde u^r,\tilde\mu^r)$ satisfies \eqref{MFG^r}. Thus, we turn to showing $\tilde \al^r_{xy}(t) = \gamma^*_y(x,\Delta_x \tilde u^r(t))$.

By way of contradiction, assume that there exists $z\in [d]$ such that $\tilde \al_z^r(t) \neq \gamma^*(z,\Delta_z \tilde u^r(t))$ for some positive Lebesgue-measure set of values of $t\in\R_+$. Denote this set by $\calT$. By uniqueness of the minimizer $\gamma^*$ for the Hamiltonian $H$ and by \eqref{eqn:MFE_tilde_u_value}, there exists $\eps>0$ such that for any $t\in\calT$:
\[
-\frac{d}{dt}\tilde u^r_z(t)+r\tilde u^r_z(t) - \eps \geq  \gamma^*(z,\Delta_z \tilde u^r(t)) \cdot \Delta_z \tilde u^r +f(z,\gamma^*(z,\Delta_z \tilde u^r(t)) + F(z,\tilde\mu^r(t)),
\] and for any $x\in[d]$ and any $t\in\R_+$,
\[
-\frac{d}{dt}\tilde u^r_x(t)+r\tilde u^r_x(t) \geq  \gamma^*(x,\Delta_x \tilde u^r(t)) \cdot \Delta_x \tilde u^r +f(x,\gamma^*(x,\Delta_x \tilde u^r(t)) + F(x,\tilde\mu^r(t)).
\] Let $X^\gamma_t$ be a jump process on $[d]$ with rates given by $\gamma_y^*(x,\Delta_x \tilde u^r(t))$ and with starting state $ X^\gamma_0 = x_0\in [d]$.

Applying It\^o's lemma to $e^{-rt}\tilde u^r_{ X^\gamma_t}(t)$, taking expectations, and using the previous inequalities,
\begin{align*}
    -\tilde u^r_{x_0}(0) &= \EE \int_0^\iy e^{-rt} \sum_{y,y\neq z} \1_{\{ X^\gamma_t=y\}} \Big(\gamma^*(y,\Delta_y \tilde u^r(t))\cdot \Delta_y \tilde u^r(t) -r\tilde u^r_y(t) + \frac{d}{dt} \tilde u^r_y(t)\Big) dt \\
    &\quad + \EE \int_0^\iy e^{-rt}  \1_{\{ X^\gamma_t=z\}} \Big(\gamma^*(z,\Delta_z \tilde u^r(t))\cdot \Delta_z \tilde u^r(t) -r\tilde u^r_z(t) + \frac{d}{dt} \tilde u^r_z(t)\Big) dt \\
    &\leq  -\eps \int_{\calT} e^{-rt} \EE[\1_{\{ X^\gamma_t = z\}}]dt \\
    &\quad - \EE \int_0^\iy e^{-rt}  \Big( f( X^\gamma_t,\gamma^*( X^\gamma_t,\Delta_{ X^\gamma_t}\tilde u^r(t))) + F( X^\gamma_t,\tilde\mu^r(t)) \Big) dt\\
    &=  -\eps \int_{\calT} e^{-rt} \EE[\1_{\{ X^\gamma_t = z\}}]dt - J_r(\gamma^*(\cdot,\Delta_{\cdot} \tilde u^r),\tilde\mu^r).
\end{align*} Since the rates of the chain $(X^\gamma_t)_{t\geq 0}$ are in the compact set $\A$ that is bounded away from zero, 
\[
\int_{\calT} e^{-rt} \EE[\1_{\{ X^\gamma_t = z\}}] dt >0,
\] and recalling the definition of $\tilde u^r$, this implies that:
\[
J_r(\gamma^*(\cdot,\Delta_{\cdot} \tilde u^r),\tilde\mu^r) < \tilde u^r_{x_0}(0)= J_r(\tilde \al^r , \tilde\mu^r).
\] However, this contradicts the minimality of $\tilde\al^r$ and so we must instead have that $\tilde \al_{xy}^r(t) =\gamma^*_y(x,\Delta_x \tilde u^r(t))$ for all $x,y\in [d]$ and $t\in\R_+$. 
\qed

\section{Proof of Proposition \ref{prop:MEr}}\label{sec:proof_MEr}

In order to prove the regularity of the derivative of the master equation, we need a linearized system around the discounted MFG system, which we analyze in the following subsection.

\subsection{Preliminary Results: Linearized System around the discounted MFG System} We now introduce a linearized system \eqref{lin_MFG_r_around_MFG^r} centered around the discounted MFG system \eqref{MFG^r} and establish a sensitivity result via Lemma \ref{lemma:lipschitz_grad}. This lemma serves us in the sequel in establishing the regularity of the solution of the discounted master equation.

Given the solution $(u^r,\mu^r)$ to \eqref{MFG^r} with $\mu^r(0) = \mu_0$ and $m_0\in\calM_0$, consider the linearized system around $( u^r,\mu^r)$, defined for $t\in\R_+$:
\begin{align}
\label{lin_MFG_r_around_MFG^r}
\begin{cases}
-\frac{d}{dt}v_x(t)=-rv_x(t) + \gamma^*(x,\Delta_x  u^r(t))\cdot \Delta_x v(t) + D^\eta_1 F(x,\mu^r(t))\cdot m(t),\\
\frac{d}{dt}m_x(t)=\sum_{y\in[d]} m_y(t)\gamma^*_x(y,\Delta_y  u^r(t)) + \sum_{y\in[d]} \mu^r_y(t) \nabla_p\gamma^*_x(y,\Delta_y  u^r(t))\cdot \Delta_y v(t),\\
m(0)=m_{0},\qquad\text{$v$ is bounded}.
\end{cases}
\end{align}
To obtain the existence and uniqueness of a solution to the latter system, we use the nice manipulation observed by \cite{car-por} that \eqref{lin_MFG_r_around_MFG^r} actually is equal to \eqref{lin_MFG_r_around_stat} with the choice:
\begin{align}\notag
A_x(t)&:=-\Big( \gamma^*(x,\Delta_x  u^r(t))- \gamma^*(x,\Delta_x  \bar u^r)\Big)\cdot \Delta_x v(t) +\Big( D^\eta_1 F(x,\mu^r(t))- D^\eta_1 F(x,\bar\mu^r)\Big)\cdot m(t),\\\notag
B_x(t)&:=-\sum_{y\in[d]} m_y(t)\big(\gamma^*_x(y,\Delta_y  u^r(t))-\gamma^*_x(y,\Delta_y  \bar u^r)\big)
\\\notag
&\qquad-\Big( \sum_{y\in[d]} \mu^r_y(t) \nabla_p\gamma^*_x(y,\Delta_y  u^r(t))- \sum_{y\in[d]} \bar\mu^r_y \nabla_p\gamma^*_x(y,\Delta_y  \bar u^r)\Big)\cdot \Delta_y v(t).
\end{align}
In fact, We can show that these $A$ and $B$ have exponential decay and so Proposition \ref{prop:lin_MFG_r_exi} holds. Indeed, as in the proof of \eqref{bound_m_Del_v}, one may show here as well that there are constants $C, r_0>0$, independent of $m_0$, such that for any $t\in\R_+$ and any $r\in (0,r_0)$,
\begin{align}\notag
|m(t)|+|\Deltadot v(t)|\le C|m_0|e^{r t/2}.
\end{align}

With this in mind, we can replicate the proof in Section \ref{sec:prop223} in order to obtain the following corollary. The corollary will be helpful in bounding the solutions of the linearized system \eqref{lin_MFG_r_around_MFG^r} by an exponential decay and the initial data. In particular, the independence from the discount $r$ will be useful in the proof of Proposition \ref{prop:MEr}.

\begin{corollary}\label{cor:MFG^r2}
The system \eqref{lin_MFG_r_around_MFG^r} admits a unique solution. Moreover, there exist $\theta,r_0>0$, and $C>0$, independent of $m_0$, such that for any $r\in(0,r_0)$ and $t\in\R_+$,
\begin{align}\notag
|m(t)|+|\Deltadot v(t)| + |v(t)| \le Ce^{-\theta t}|m_0|.
\end{align}
\end{corollary}

We now begin to make more concrete the relationship between the solution to the linearized system \eqref{lin_MFG_r_around_MFG^r} and differences of solutions to the MFG system \eqref{MFG^r}. A preliminary result that establishes a second-order bound follows.

\begin{lemma}\label{lemma:lipschitz_grad}
Let $(u^r,\mu^r)$ and $(\hat u^r,\hat \mu^r)$ be solutions of \eqref{MFG^r} with the initial data $\mu^r(0)=\mu^r_0$ and $\hat \mu^r(0)=\hat \mu^r_0$, respectively. Also, let $(v,m)$ be the solution of \eqref{lin_MFG_r_around_MFG^r}, with the initial state $m(0)=\mu^r_0-\hat\mu^r_0$. 
Then for any $r >0$, there is a constant $C_r>0$ that may depend on $r$, but is independent of the initial data $\mu^r_0$ and $\hat\mu^r_0$, such that:
\begin{align}\label{u_mu_lipschitz_deriv}
\sup_{t\in \R_+}\left(|\hat\mu^r(t)-\mu^r(t)-m(t)|+|\hat u^r(t)-u^r(t)-v(t)|\right)\le C_r|\mu^r_0-\hat\mu^r_0|^2.
\end{align}
\end{lemma}

For the proof of this lemma as well as for future reference, we need the following linearized system.
\begin{align}
\begin{cases}
-\frac{d}{dt}w_x(t)=-rw_x(t) + \gamma^*(x,\Delta_x u^r(t))\cdot \Delta_x w(t) + D^\eta_1 F(x,\mu^r(t))\cdot \rho(t) + A_x(t),\\
\frac{d}{dt}\rho_x(t)=\sum_{y\in[d]} \rho_y(t)\gamma^*_x(y,\Delta_y  u^r(t)) + \sum_{y\in[d]} \mu^r_y(t) \nabla_p\gamma^*_x(y,\Delta_y  u^r(t))\cdot \Delta_y w(t)+B_x(t),\\
\rho(0)=\rho_{0},\qquad\text{$w$ is bounded}. \label{lin_MFG_r_around_u}
\end{cases} 
\end{align} 

\begin{proof}[Proof of Lemma \ref{lemma:lipschitz_grad}]
Arithmetic implies that $w(t):=\hat{u}^r(t) - u^r(t) - v(t)$ and $\rho(t):= \hat{\mu}^r(t)-\mu^r(t)-m(t)$  solve the above linearized system \eqref{lin_MFG_r_around_u} 
with the values:
\begin{align*}
A_x(t) := & H(x,\Delta_x\hat{u}^r(t)) - H(x,\Delta_x u^r(t)) - \gamma^*(x,\Delta_x u^r(t))\cdot \Delta_x(\hat{u}^r(t)-u^r(t))\\
&+F(x,\hat{\mu}^r(t)) - F(x,\mu^r(t)) - D^\eta_1 F(x,\mu^r(t))\cdot (\hat{\mu}^r(t) - \mu^r(t)),
\end{align*} and
\begin{align*}
    B_x(t) := \sum_{y\in [d]} \Big[ \hat{\mu}^r_y(t)[&\gamma^*_x(y,\Delta_y \hat{u}^r(t))-\gamma^*_x(y,\Delta_y u^r(t)) - \nabla_p \gamma^*_x(y,\Delta_y u^r(t))\cdot \Delta_y (\hat u^r(t) - u^r(t))]\\
    &+ (\hat\mu_y(t)-\mu_y^r(t)) [\nabla_p \gamma^*_x(y,\Delta_y u^r(t)) \cdot \Delta_y (\hat{u}^r(t) - u^r(t))]\Big].
\end{align*} Fix $0\leq t_1<t_2\leq \iy$. We briefly consider $|B(t)|$ by itself. Using Taylor's theorem to generate a second order bound for $\gamma^*$, Young's inequality, and with $C>0$ from Proposition \ref{prop:MFG},
\begin{align*}
    |B(t)|  &\leq C_{L,\gamma} \Big(\sum_{y\in [d]} |\Delta_y (\hat u^r(t) - u^r(t))|^2  
     +\sum_{y\in [d]} |\hat\mu_y(t) - \mu_y(t)||\Delta_y (\hat u^r(t) - u^r(t))|\Big) \\\notag
    &\le C_{L,\gamma} \Big( \frac{3}{2} \sum_{y\in [d]} |\Delta_y (\hat u^r(t) - u^r(t))|^2   +\frac{1}{2} \sum_{y\in [d]} |\hat\mu_y(t) - \mu_y(t)|^2 \Big) \\\notag
    &\le 2C_{L,\gamma} C |\hat\mu_0 - \mu_0|^2. \numberthis{\label{B_integral}}
\end{align*} Throughout the proof $C>0$ may change from one line to the next, but will remain independent of $r>0$, and the data $\mu_0^r$ and $\hat\mu_0^r$. We can similarly bound $A$ using a second-order bound and Proposition \ref{prop:MFG}:
\begin{align*}
    |A(t)| &\le (C_{L,H}+C_{L,F}) \big[\max_{x\in [d]}|\Delta_x (\hat u^r - u^r)(t)|^2 + |\hat\mu(t) - \mu(t)|^2 \big] \\\notag
    &\le C (C_{L,H}+C_{L,F}) |\hat \mu_0 - \mu_0|^2. \numberthis{\label{A_integral}}
\end{align*}

Applying Lemma \ref{lem:v} to $w(t,x)$ on $[t_1,t_2]$, using the boundedness of $D^\eta F$, taking a supremum bound for $e^{-rt}\rho(t)$, applying \eqref{A_integral}, and taking $t_2\to\iy$ gives:
\[
e^{-rt_1} |\Delta w(t_1)|  \leq C\big(\sup_{t\in [t_1,\iy)} e^{-rt}|\rho(t)| + e^{-rt_1}|m_0|^2\big).
\] By Proposition \ref{prop:MFG} (ii) and Corollary \ref{cor:MFG^r2}, $|\rho(t)| \leq Ce^{-\theta t}$ for some $\theta>0$. So there exists $T=T_r>0$ so that, \begin{align}\label{eq:rho}
    |\rho(T)| = \sup_{t\in\R_+} |\rho (t)|.
\end{align} Note that this $T=T_r$ will appear in the final constant and this is where the dependence on $r>0$ comes from. Combining this with the above display and multiplying by $e^{-rt_1}$ makes the right hand side independent of $t_1$. Then, taking the supremum over all $t_1\in\R_+$ yields:
\begin{align}\label{sq_bound:sup_w}
    \sup_{t_1\in\R_+}|\Delta w(t_1)| \leq C\big( |\rho (T)| + |m_0|^2\big).
\end{align}

Let $\hat w$ be the solution to:
\[
-\frac{d}{dt} \hat w(t,x) = -r\hat w(t,x) + \Delta_x \hat w(t) \cdot \gamma^*(x,\Delta_x u^r(t)), \quad\text{  } t\in [t_1,t_2], \quad \hat w(t_2,x) = \xi_x,
\] where $\xi \in\R^d$ is arbitrary. Applying Lemma \ref{lem:v} to $\hat w$ gives:
\begin{align*}
    |\Delta \hat w(t)| \leq Ce^{-c(t_2-t)} |\Delta \xi|.
\end{align*} Taking the supremum,
\begin{align}\label{lem:46_hat}
    \sup_{t\in [t_1,t_2]}|\Delta \hat w(t)| \leq C|\Delta \xi|,
\end{align} where $C$ is still independent of $t_1, t_2$. We can differentiate $e^{-rt}\hat w(t)\cdot \rho(t)$ and using \eqref{lin_MFG_r_around_u},
\begin{align*}
    \frac{d}{dt} \Big[e^{-rt} &\sum_{x\in [d]} \hat w_x(t) \rho_x(t) \Big] \\ \notag
    &= e^{-rt} \Big[\sum_{x\in [d]} \rho_x(t)\big(-\Delta_x \hat w(t) \cdot \gamma^*(x,\Delta_x u^r(t))\big)\\ \notag
    &\qquad +\hat w_x(t)\big(\sum_{y\in [d]} \rho_y(t) \gamma^*_x(y,\Delta_y u^r(t)) + \mu_y^r(t)\nabla_p\gamma^*_x(y,\Delta_y u^r(t)) \cdot \Delta_y w(t) +B_x(t)\big)\Big].
\end{align*} We can use the fact that $\sum_{x\in [d]} \gamma^*_x(y,\Delta_y u^r(t))=0$ to see that the first sums in each line after the equality have an opposite sign and so cancel with one another. We can then integrate both sides from $t_1$ to $t_2$, and rearrange terms,
\begin{align*}
    \int_{t_1}^{t_2} e^{-rt} &\sum_{x,y\in [d]} \hat w_x(t) \mu_y(t) \nabla_p \gamma^*_x(y,\Delta_y u^r(t))\cdot \Delta_y w(t) dt \\
    &= \sum_{x\in [d]} e^{-rt_2}  \xi_x\rho_x(t_2) - e^{-rt_1}\hat w_x(t_1)\rho_x(t_1) - \sum_{x\in [d]} \int_{t_1}^{t_2} e^{-rt} [\hat w_x(t)-\langle \hat w(t)\rangle_x] B_x(t) dt,
\end{align*} where for the last term we used the fact that $B(t)\in\calM_0$ in order to insert $\langle\hat w(t)\rangle$. Denote $G_{x,z}(y):= -\pl_{p_z} \gamma^*_x(y,\Delta_y u^r(t)) = -\pl^2_{p_x,p_z} H(y,\Delta_y u^r(t))$ and let $G(y)$ be the $d\times d$ matrix with entries $[G_{x,z}(y)]_{x,z\in [d]}$. We note that by the assumptions on $H$, $G(y)$ is symmetric and positive definite. We note that the integrand on the left-hand side of the above display can be written:
\[
-\sum_{y\in [d]} \mu^r_y(t) (\Delta_y \hat w(t))^T G(y) \Delta_y w(t).
\] With $t_1=0$, $\rho(t_1)=0$. Rearranging, using these observations, \eqref{lem:46_hat}, and \eqref{B_integral},
\begin{align*}
    e^{-rt_2} |\rho(t_2)\cdot \xi| &\leq C|\xi| \int_0^{t_2} e^{-rt} \sum_{y\in [d]} \mu^r_y(t) |G(y)\Delta_y w(t)| dt + C|\xi||\mu_0 - \hat\mu_0|^2
\end{align*} Using Jensen's inequality (for the product measure $\mu^r(t) \otimes e^{-rt} dt$ and the function $x\mapsto x^2$), boundedness of $G$, and taking the supremum over all $|\xi| =1$,
\begin{align}\label{sq_bound:rho_like}
    e^{-rt_2} |\rho(t_2)| &\leq C\Big(\int_0^{t_2} e^{-rt} \sum_{y\in [d]} \mu^r_y(t) |\Delta_y w(t)|^2 dt\Big)^{1/2} + C|\mu_0 - \hat\mu_0|^2.
\end{align} In order to treat the integral on the right-hand side, we mimic the proof of Lemma \ref{lem:duality}, this time keeping $A$, and find that:
\begin{align*}
&C^{-1} \int_{t_1}^{t_2} e^{-rt} \sum_{y\in [d]} \mu_y^r(t) |\Delta_y w(t)|^2 dt\\
&\quad\leq -e^{-rt} \rho(t) \cdot w(t)\Big|_{t_1}^{t_2} + \sum_{x\in [d]} \int_{t_1}^{t_2} e^{-rt} (w_x(t)B_x(t) - \rho_x(t) A_x(t)) dt. 
\end{align*} 
Take $t_1=0$ and $t_2 =\iy$, use \eqref{B_integral}, \eqref{A_integral}, the facts $\rho(t),B(t)\in\calM_0$, and Lemma \ref{lem:b}, then:
\begin{align}\label{sq_bound:duality}
C^{-1} \int_{0}^{\iy} e^{-rt} \sum_{y\in [d]} \mu_y^r(t) |\Delta_y w(t)|^2 dt \leq \sup_{t\in\R_+} |\Delta w(t)| |\mu_0-\hat\mu_0|^2 + |\rho (T)||\mu_0-\hat\mu_0|^2  \end{align} 
where we recall that $T$ is chosen such that $|\rho(T)|$ attains the supremum in \eqref{eq:rho}. In \eqref{sq_bound:rho_like}, we note that the right-hand side is bounded above by the same quantity, except that here the integral upper bound is $\iy$; so, plugging in \eqref{sq_bound:duality} to \eqref{sq_bound:rho_like} gives:
\begin{align*}
    e^{-rt_2} |\rho (t_2)| \leq C\big[\big(\sup_{t\in\R_+}|\Delta w(t)| + |\rho(T)| \big)^{1/2}|\mu_0-\hat\mu_0| + |\mu_0 -\hat\mu_0|^2\big].
\end{align*} In particular, this holds for $t_2 = T$. Using the fact that $\sqrt{a+b} \leq \sqrt{a} + \sqrt{b}$ for $a,b\geq 0$ and Young's inequality:
\begin{align}\label{sq_bound:final_rho}
    |\rho(T)| \leq Ce^{2rT} |\mu_0 -\hat\mu_0|^2.
\end{align} Note that the constant is not independent of $r$ since we could possibly have $e^{2rT_r}\to\iy$ as $r\to 0^+$. Plugging in \eqref{sq_bound:final_rho} to \eqref{sq_bound:sup_w}:
\begin{align}\label{sq_bound:final_delta_w}
    \sup_{t\in\R_+} |\Delta w(t)| \leq Ce^{2rT} |\mu_0 -\hat\mu_0|^2.
\end{align} Using the ODE $w$ satisfies, integrating, using \eqref{A_integral}, and applying \eqref{sq_bound:final_rho} and \eqref{sq_bound:final_delta_w} gives:
\begin{align*}
    e^{-rt}|w_x(t)| &\leq \int_{t}^{\iy} e^{-rs} |\gamma^*(x,\Delta_x u^r(s))\cdot \Delta_x w(s) + D^\eta_1 F(x,\mu^r(s))\cdot \rho(s) + A_x(s)| ds \\
    &\leq Cr^{-1} e^{-rt}(\sup_{t_1\in\R_+} |\Delta w(t_1)| + |\rho(T)| + |\mu_0 -\hat\mu_0|^2) \\
    &\leq Cr^{-1} e^{-rt} e^{rT} |\mu_0 - \hat\mu_0|^2.
\end{align*} Multiplying both sides by $e^{-rt}$ finishes the proof.
\end{proof}

Using the previous lemma, we will show that $D^\eta_1 U_r$ is Lipschitz in $\mu_0$ with Lipschitz constant $C_r$, which depends on $r$. Here, we begin to establish the regularity of $U_r$ that is uniform in $r$. The result in the next proposition allows for a uniform Lipschitz constant for the family $\{D^\eta_1 U_r(x,\eta)\}_{r\in (0,r_0)}$. Once more we go through the linearized system in order to establish uniform-in-$r$ regularity of $U_r$.

\begin{proposition}\label{prop:linearized_lipschitz}
Let $(u^r,\mu^r)$ and $(\hat u^r, \hat \mu^r)$ be solutions to \eqref{MFG^r} with initial data $\mu^r(0) = \mu_0$ and $\hat\mu^r(0) = \hat\mu_0$. Let $(v,m)$ and $(\hat v, \hat m)$ solve \eqref{lin_MFG_r_around_MFG^r} with $(u^r,\mu^r)$ and $(\hat u^r, \hat \mu^r)$, respectively, and with $m(0)= \hat m(0)=m_0\in\calM_0$. Then, there exist $C,r_0,\theta>0$ independent of the initial data, such that for all $r\in (0,r_0)$ and all $t\in \R_+$,
\[
 |m(t) -\hat m(t)| + |v(t) - \hat v(t)| \le C(1+|\mu_0-\hat\mu_0||m_0|+t)e^{-\theta t}|m_0||\mu_0 - \hat\mu_0|. 
\]
\end{proposition}
\begin{proof}
With $w:= v- \hat v$ and $\rho := m-\hat m$, we note that $(w,\rho)$ solves \eqref{lin_MFG_r_around_u} with $\rho(0)=0$,
\begin{align*}
    A_x(t) &:= \big[\gamma^* (x,\Delta_x u^r(t)) -\gamma^* (x,\Delta_x \hat u^r(t))\big] \cdot \Delta_x \hat v(t) \\
    &\qquad + \big[ D^\eta_1 F(x,\mu^r(t)) -D^\eta_1 F(x,\hat\mu^r(t)) \big] \cdot \hat m(t),
\end{align*} and
\begin{align*}
    B_x(t) &:= \sum_{y\in [d]} \Big\{ \hat m_y(t) \big[ \gamma^*_x (y,\Delta_y u^r(t)) -\gamma^*_x (y,\Delta_y \hat u^r(t)) \big] \\
    &\qquad\qquad  + \big[ \mu_y^r(t) \nabla_p \gamma^*_x(y,\Delta_y u^r(t)) - \hat \mu_y^r(t) \nabla_p \gamma^*_x(y,\Delta_y \hat u^r(t))\big] \cdot \Delta_y \hat v(t)\Big\}.
\end{align*} 

Note that from Proposition \ref{prop:MFG}, together with Corollary \ref{cor:MFG^r2}, there are $C,\theta,r_0>0$ such that for any $t\in \R_+$ and $r\in (0,r_0)$,
\begin{equation}\label{master:AB_decay}
    |A(t)| + |B(t)| \leq C|\mu_0 - \hat\mu_0| |m_0|e^{-\theta t} .
\end{equation} For the rest of the proof, $C$ is a positive constant, independent of $r$, the time variables, and the initial data, which may change from one line to the next.
The rest of the proof here is adapted from that of Lemma \ref{lem:lin_MFG_r_exp_A_B}. We start with showing that one can reduce the problem to one with $A\equiv 0$. To this end,
let $(w^A,\rho^A)$ be the unique solution to \eqref{lin_MFG_r_around_u} with some given $A$ and $B$ satisfying \eqref{master:AB_decay}. Also, let $\check w$ be the unique (bounded) solution to: 
\begin{align}\notag
-\frac{d}{dt}\check w_x(t)=-r\check w_x(t) + \gamma^*(x,\Delta_x u^r(t))\cdot \Delta_x \check w(t) + A_x(t), \qquad t\in\R_+.
\end{align}
By Lemma \ref{lem:v} with $T\to\iy$, there exist $c,C>0$ depending only on $\A$, such that:
\begin{align}\label{bound_delta_check_w}
|\Deltadot \check w (t)|\le Ce^{(c+r)t} \int_t^\iy e^{-(c+r)s}|A(s)|ds \le \frac{Ce^{-\theta t}}{c+r+\theta} |\mu_0 - \hat\mu_0||m_0|,   \qquad  t\in\R_+.
\end{align}
Now, $(w^A-\check w,\rho^A)$ solves \eqref{lin_MFG_r_around_u} with $A\equiv 0$ and $\check B_x(t)=B_x(t)+\sum_{y\in[d]}\mu^r_y\nabla_p\gamma^*_x(y,\Delta_y u^r)\cdot\Delta_y \check w(t)$.
Still, $\check B(t)\in\calM_0$. Recall that by the vanishing discount proof in Proposition \ref{prop:statDisc_erg_mfg}, $|\Delta_x \bar u^r|$ is uniformly bounded for $r$ small enough. So, by Proposition \ref{prop:MFG} and this observation, we can modify $C,r_0>0$, so that for all $r\in (0,r_0)$:
\[
|\Delta u^r| \leq C|\Delta \bar u^r| + C|\mu_0 - \bar\mu^r| \leq \hat C,
\] where, still $\hat C$ is independent of $r$ and the initial data. So, $\Delta_x u^r$ is uniformly bounded for $r\in(0,r_0)$ and for any $x\in [d]$. By assumption the second derivative of $H$ is bounded on any compact set and equals $\nabla_p \gamma^*$. Therefore by \eqref{master:AB_decay} and \eqref{bound_delta_check_w}, \begin{align}\label{eqn:checkB2}
|\check B(t)| \le C_1 e^{-\theta t}|\mu_0-\hat\mu_0||m_0|, 
\end{align}
where $C_1$ depends only on $\A$, and the bounds on $C_{L,\gamma^*}$, $C_{L,\nabla_p\gamma^*}$, and $C_{\nabla_p\gamma^*}$.

For any $b>0$, define the set of measurable functions: 
$$\calB_b:=\{B:\R_+\to\calM_0 \mid |B(t)| \le  b e^{-\theta t},\quad \forall t\in\R_+\}. $$
We prove that for any $c_1>0$ there exists $C_2>0$ independent of $c_1,r,\mu_0,\hat\mu_0$, and $m_0$, such that 
\begin{align}\label{bound_w_rho}
\sup_{B\in\calB_{c_1}}\left\{|\rho (t)|+|\Deltadot w(t)|\right\}\le C_2 c_1
(1+|\mu_0-\hat\mu_0||m_0|+t)e^{- \theta t},
\end{align} where $(\rho,w)$ above solves \eqref{lin_MFG_r_around_u} with $A\equiv 0$ and some $B\in\calB_{c_1}$. Recall that $(w^A-\check w,\rho^A)$ solves \eqref{lin_MFG_r_around_u} with $A\equiv 0$ and recall also the parameter $C_1>0$ from \eqref{eqn:checkB2}. Then, \eqref{bound_w_rho} implies,
\begin{align*}
|\rho^A (t)|+|\Deltadot (w^A-\check w)(t)|\le C_1C_0|\mu_0-\hat\mu_0||m_0|
(1+|\mu_0-\hat\mu_0||m_0|+t)e^{- \theta t}.
\end{align*}
Together with \eqref{bound_delta_check_w}, we obtain that
\begin{align*}
|\rho^A (t)|+|\Deltadot w^A(t)|\le 2C_1C_0|\mu_0-\hat\mu_0||m_0|
(1+|\mu_0 - \hat\mu_0||m_0| + t)e^{- \theta t}.
\end{align*} The rest of the proof is dedicated to establish \eqref{bound_w_rho} (which by definition considers $A\equiv 0)$.

Since we are able to assume $A\equiv 0$, we may follow the proof of Lemma \ref{lem:duality} to obtain that: 
\begin{align*}
    \int_{t_1}^{t_2} e^{-rs} \sum_{y\in [d]} \mu_y^r(s) |\Delta w(s)|^2 ds \le -C \Big[e^{-rt} \rho(t) \cdot w(t)\Big]_{t_1}^{t_2} + C\int_{t_1}^{t_2} e^{-rs}|B(s)|^2 ds,
\end{align*} for $0\leq t_1<t_2\leq \iy$. Since $\rho(0)=0$ and $\lim_{t\to\iy} e^{-rt} \rho(t)\cdot w(t) =0$, 
\begin{align*}
    \int_{0}^{\iy} e^{-rs} \sum_{y\in [d]} \mu_y^r(s) |\Delta w(s)|^2 ds \le  C\int_{0}^{\iy} e^{-rs}|B(s)|^2 ds. 
\end{align*} The bound for $B$ then yields:
\begin{align*}
    \int_0^\iy e^{-rs} \sum_{y\in [d]} \mu_y^r(s) |\Delta w(s)|^2 ds \le \frac{C}{r+\theta}|\mu_0- \hat\mu_0|^2|m_0|^2.
\end{align*}

Using this fact and Lemma \ref{lem:M0},
\begin{align*}
        |\rho (t)| &\le Ce^{rt/2}\Big(\int_0^t e^{-rs}  \sum_{y\in [d]} \mu_y^r(s) |\Delta w(s)|^2 ds\Big)^{1/2} +C\int_0^t e^{-c(t-s)} |B(s)|_1 ds \leq Ce^{rt/2}|\mu_0 - \hat\mu_0||m_0|.
\end{align*} At the same time, Lemma \ref{lem:v} and boundedness of $D^\eta F$ gives:
\begin{align*}
    |\Delta w(t)| &\le Ce^{(r+c)t} \int_t^\iy e^{-(r+c)s} |D^\eta_1 F(x,\mu^r(s))||\rho(s)| ds \le C e^{rt/2}|\mu_0 - \hat\mu_0| |m_0|. 
\end{align*} Hence,
\begin{align*}
    |\Delta w(t)| + |\rho(t)| \le Ce^{rt/2} |\mu_0 - \hat\mu_0||m_0|. \numberthis{\label{eqn:what_we_need}}
\end{align*}

For any $b>0$, define:
\[
\bar g^r(b,t) := b^{-1}\sup_{B\in\calB_b} e^{-rt} (|\rho (t)| + |\Delta w(t)|),
\]
where $(w,\rho)$ solves \eqref{lin_MFG_r_around_u} with $A\equiv 0$, $\rho(0)=\rho_0=0$, and some $B\in\calB_{b}$. 
The structure of \eqref{lin_MFG_r_around_u} implies that if $(\tilde w, \tilde \rho)$ is a solution with a given $B$, then $(\ell \tilde w, \ell \tilde \rho)$ is the solution associated with $\ell B$ for any $\ell \in\R$. Therefore, for any $b>0$, 
\begin{align}\notag
    \bar g^r(t):=\bar g^r( 1,t)
    =\bar g^r( b,t),\qquad t\in\R_+ .
\end{align} Fix $\tau >0$ and let $(w^{\tau,1},\rho^{\tau,1})$ be the solution of \eqref{lin_MFG_r_around_u} on $[\tau,\iy)$ with $A\equiv B\equiv 0$ and $\rho^{\tau,1} (\tau) = \rho(\tau)$. Also let $(w^{\tau,0},\rho^{\tau,0})$ be the solution to \eqref{lin_MFG_r_around_u} on $[\tau,\iy)$ with $A\equiv 0$, $\rho^{\tau,0}(\tau)=0$, and the given $B$. Then on the time interval $[\tau,\iy)$, $(w,\rho)=(w^{\tau,1},\rho^{\tau,1})+(w^{\tau,0},\rho^{\tau,0})$. 

Note that $(\hat w^{\tau,0}(\cdot),\hat \rho^{\tau,0}(\cdot)):=(w^{\tau,0}(\tau+\cdot),\rho^{\tau,0}(\tau+\cdot))$ solves \eqref{lin_MFG_r_around_u} on $\R_+$. Recall that we assume $|B(t)|\le c_1e^{-\theta t}$. So, $B(\tau + t) \le c_1 e^{-(t+\tau)\theta}$. Together with the definition of $\bar g^r$ and the aforementioned linear property, we have:
\[
c_1^{-1}e^{-rt} (|\rho^{\tau,0} (\tau + t)| + |\Delta w^{\tau ,0 }(\tau +t)|) = e^{-\theta \tau}(c_1 e^{-\theta\tau})^{-1}e^{-rt} (|\hat \rho^{\tau,0} ( t)| + |\Delta \hat w^{\tau ,0 }( t)|) \leq e^{-\theta \tau}\bar g^r( t).
\] Note that $(w^{\tau,1}, \rho^{\tau,1})$ solves \eqref{lin_MFG_r_around_MFG^r} and so we have from Corollary \ref{cor:MFG^r2} that there exist $C_3,\theta>0$ independent of $r,\tau,t,m_0,\mu_0,\hat\mu_0$ such that:
\begin{align*}
    |\Delta w^{\tau,1}(\tau+t)| + |\rho^{\tau,1}(\tau+t)| \le C_3e^{-\theta t}|\rho(\tau)|. 
\end{align*} As a consequence of the previous two displays and then using \eqref{eqn:what_we_need},
\begin{align*}
c_1^{-1} (|\rho(\tau + t)| + |\Delta w(\tau +t)|) &\leq c_1^{-1} C_3e^{-\theta t} |\rho(\tau)| + e^{-\theta \tau} e^{rt} \bar g^r(t) \\
&\leq Ce^{-\theta t}e^{r\tau/2}  + e^{-\theta \tau} e^{rt}\bar g^r(t).
\end{align*} 
Multiplying both sides by $e^{-r(\tau + t)}$ and taking the supremum over $B$ satisfying $|B(t)| \leq c_1 e^{-\gamma t}$, one gets:
\[
\bar g^r (\tau +t) \leq Ce^{-(\theta+r) t} e^{-r\tau/2} + e^{-(\theta +r) \tau} \bar g^r(t) \le Ce^{-(\theta+r) t}  + e^{-(\theta +r) \tau} \bar g^r(t).
\] Multiplying both sides by $e^{(\theta +r)(t+\tau)}$, 
\begin{align}\notag
e^{(\theta+r)(t+\tau)}\bar g^r(\tau+t)\le Ce^{\theta \tau}+e^{(\theta+r)t}\bar g^r(t),
\end{align}
which, together with the fact that $\bar g^r(t) = \bar g^r(b,t)$ for any $b>0$, implies the exponential decay in \eqref{bound_w_rho}.
\end{proof}

\subsection{Proof of (i)} For any $t_0\in\R_+$ and $\mu_0\in\calP([d])$, let $(u^{r,t_0,\mu_0},\mu^{r,t_0,\mu_0})$ be the solution to \eqref{MFG^r} on $[t_0,\iy)$ with the initial data $\mu^r(t_0)=\mu_0$. Define $\bar U_r:\R_+\times[d]\times\calP([d])\to\R$, by  
$$\bar{U}_r(t,x,\eta) := e^{-r t} u^{r,t,\eta}(t,x).$$ 
Clearly, $(u^{r,t,\mu_0}(t+\cdot),\mu^{r,t,\mu_0}(t+\cdot))$ solves \eqref{MFG^r} on $\R_+$ with the initial condition $\mu^{r,t,\mu_0}(0)=\mu_0$. By uniqueness of the solution to \eqref{MFG^r}, we get that $(u^{r,t,\mu_0}(t+\cdot),\mu^{r,t,\mu_0}(t+\cdot))=(u^{r,0,\mu_0}(\cdot),\mu^{r,0,\mu_0}(\cdot))$ and as a result 
\begin{align*}
\bar U_r(t,x,\mu_0) =e^{-rt}u^{r,t,\mu_0}(t,x)=e^{-rt}u^{r,0,\mu_0}(0,x)= e^{-rt} \bar U_r(0,x,\mu_0),
\end{align*}
where here and in the sequel we often use the notation $u(t,x)$ for $u_x(t)$. Also, let $v^{r,t_0,\mu_0,m_0}(t,x)$ be the solution of  \eqref{lin_MFG_r_around_MFG^r} on $[t_0,\iy)$ with the data $(u^{r,t_0,\mu_0},\mu^{r,t_0,\mu_0})$ and $m(t_0) = m_0$. Let $\bar D: \R_+ \times [d] \times \calP([d]) \to \R^d$ be given by $\bar D = (\bar D_z)_{z\in [d]}$ where 
$$\bar D_z(t_0,x,\mu_0) := v^{r,t_0,\mu_0,e_{1z}}(t_0,x).$$ By linearity of \eqref{lin_MFG_r_around_MFG^r}, it follows that 
\begin{align}\label{master:linear_deriv}
    v^{r,t_0,\mu_0,m_0}(t_0,x) = \bar D(t_0,x,\mu_0)\cdot m_0.
\end{align} By Lemma \ref{lemma:lipschitz_grad}, there exists $C_r>0$ such that for all $t_0 \in \R_+$ and all $\mu_0,\tilde \mu_0\in\calP([d])$:
\[
|\bar{U}_r(t_0,x,\mu_0) - \bar{U}_r(t_0,x,\tilde\mu_0) - \bar D(t_0,x,\mu_0)\cdot (\mu_0 - \tilde\mu_0)| \leq C_r|\mu_0 - \tilde \mu_0|^2.
\] So, we have \begin{align}\label{master:bar_U_D}
D^\eta_1 \bar U_r(t_0,x,\mu_0) = \bar D(t_0,x,\mu_0)
\end{align} and $\bar D$ is Lipschitz in $\mu_0$ with Lipschitz constant $C_r$, a constant depending on $r$ for now. Let $\alpha: [t_0,\iy) \times \calP([d])\to \calQ$ be defined by $\alpha_{xy}(s,\eta):= \gamma^*_y(x,\Delta_y u^{r,0,\mu_0}(t_0))$. Also, denote $\mu(t)=\mu^{r,0,\mu_0}(t)$, $t\in\R_+$, and define $\mu^{(s)}:= (1-s)\mu(t_0) + s\mu(t_0 +h)$. Since $D^\eta \bar U_r$ exists, we can write for $h>0$ that: 
\begin{align*}
    \bar U_r(t_0 + h, x,\mu(t_0+h)) -\bar U_r (t_0 &+h,x,\mu(t_0)) \\ \numberthis \label{eqn:star_35}
    &= \int_0^1 \Big[\frac{\pl}{\pl(\mu(t_0+h) - \mu(t_0))} \bar U_r(t_0 + h,x, \mu^{(s)})\Big]_1 ds \\
    &= \int_0^1 \int_{t_0}^{t_0+h} \sum_{z,y\in [d]} \mu_y(t) \al_{yz} (t,\mu(t)) D^\eta_{1z} \bar U_r (t_0 + h,x,\mu^{(s)}) dtds \\
    &= \int_0^1 \int_{t_0}^{t_0+h} \sum_{z,y\in [d]} \mu_y(t) \al_{yz} (t,\mu(t)) D^\eta_{yz} \bar U_r (t_0 + h,x,\mu^{(s)}) dtds \\
    &= \int_0^1 \int_{t_0}^{t_0+h} \sum_{y\in [d]} \mu_y(t) \al_{y} (t,\mu(t))\cdot D^\eta_y \bar U_r (t_0 + h,x,\mu^{(s)}) dtds,
\end{align*} where recall that $A_y$ is the $y$-th row of the matrix $A$ and where for the second and third equalities we used \eqref{simplex_witchcraft} since $\mu_y(t) \al_{y}(t,\mu(t)) \in\calM_0$. 


Since $D^\eta \bar U_r$ is Lipschitz,
\begin{align*}
    \lim_{h\to 0^+} \frac{1}{h} \big[\bar{U}_r(t_0 +h,x,\mu(t_0+h)) - \bar{U}_r(t_0+h,x&,\mu(t_0))\big] \\
    &= \sum_{y\in [d]}\mu_y(t_0) \alpha_y(t_0,\mu(t_0))D^\eta_y \bar{U}_r(t_0,x,\mu(t_0)).
\end{align*} Therefore,
\begin{align*}
    \lim_{h\to 0^+} \frac{1}{h} \big[ \bar U_r(t_0 + h,x,\mu(t_0+h)) -\bar U_r&(t_0,x,\mu(t_0))\big] \\ 
    &= \lim_{h\to 0}\Big\{ \frac{1}{h} \big[ \bar U_r(t_0 + h,x,\mu(t_0+h)) -\bar U_r(t_0+h,x,\mu(t_0))\big] \\
    &\qquad \quad + \frac{1}{h} \big[ \bar U_r(t_0 + h,x,\mu(t_0)) -\bar U_r(t_0,x,\mu(t_0))\big]\Big\}\\
    &= \sum_{y\in [d]}\mu_y(t_0) \alpha_y(t_0,\mu(t_0))D^\eta_y \bar{U}_r(t_0,x,\mu(t_0)) \\
    &\qquad \quad +\partial_t \bar U_r (t_0,x,\mu(t_0)). 
\end{align*} On the other hand, by the definition of $\bar U_r$ and since $u^{r,t,\mu(t)}(t,x)=u^{r,0,\mu(0)}(t,x)$,
\begin{align*}
    \lim_{h\to 0^+} \frac{1}{h}\big[\bar{U}_r(t_0+h,x,\mu(t_0&+h)) - \bar{U}_r(t_0,x,\mu(t_0))\big] \\
    &= \lim_{h\to 0^+} \frac{1}{h}\big[ e^{-r(t_0+h)}u^{r,t_0+h,\mu(t_0+h)}(t_0+h,x) - e^{-rt_0}u^{r,t_0,\mu(t_0)}(t_0,x)\big]\\
    &= e^{-rt_0} \lim_{h\to 0^+} \frac{1}{h} \big[e^{-rh} u^{r,0,\mu_0} (t_0+h,x) - u^{r,0,\mu_0}(t_0,x) \big]\\
    &= e^{-rt_0}\big[\frac{d}{dt} u^{r,0,\mu_0}(t_0,x) - ru^{r,0,\mu_0}(t_0,x)\big].
\end{align*} Putting these two observations together and using the definition of $\al$, we obtain that:
\begin{align*}
\partial_t \bar{U}_r(t_0,x,\mu(t_0)) &= e^{-rt_0}\left[\frac{d}{dt} u^{r,0,\mu_0}(t_0,x) - ru^{r,0,\mu_0}(t_0,x)\right] \\
&\qquad - \sum_{y\in [d]} \mu_y(t_0) \gamma^*(y,\Delta_y u^{r,0,\mu_0}(t_0,\cdot)) \cdot D^\eta_y \bar U_r (t_0,x,\mu(t_0)).
\end{align*} Now, using the facts that $u^r$ solves $\eqref{MFG^r}$ and that $ \bar{U}_r(t_0,x,\mu(t_0)) = e^{-rt_0}\bar{U}_r(0,x,\mu(t_0))$, we obtain:
\begin{align*}
re^{-rt_0} \bar U_r(0,x,\mu(t_0)) &= e^{-rt_0}\left[ H(x,\Delta_x u^{r,0,\mu_0}(t_0,\cdot)) + F(x,\mu(t_0))\right]\\
&\qquad + e^{-rt_0}\sum_{y\in [d]} \mu_y(t_0) \gamma^*(y,\Delta_y u^{r,0,\mu_0}(t_0,\cdot)) \cdot D^\eta_y \bar U_r (0,x,\mu(t_0)). \notag
\end{align*} Take $t_0=0$ and define: \begin{align}\label{eq:Ur}
U_r(x,\eta):= \bar{U}_r(0,x,\eta) = u^{r,0,\eta}(0,x).
\end{align}
Expand the dot product to obtain that for any $\eta\in\calP([d])$: 
\[
rU_r(x,\eta)= H(x,\Delta_x U_r(\cdot,\eta))+\sum_{y,z\in[d]}\eta_y D^\eta_{yz} U_r(x,\eta)\gamma^*_z(y,\Delta_y U_r(\cdot,\eta))+F(x,\eta),
\] which is exactly \eqref{ME^r}. 

\subsection{Proof of (ii)} Fix arbitrary $t_0>0$ and $\mu_0\in\calP([d])$. Let $\tilde U_r$ be a solution to \eqref{ME^r}. Fix $\mu_0 \in\calP([d])$ and let $\tilde\mu^r :[t_0,\iy) \to\calP([d])$ be the unique solution to: 
\[
\begin{cases}
\frac{d}{dt} \tilde\mu_x^r(t) = \sum_{y\in [d]} \tilde\mu_y^r(t) \gamma_x^*(y,\Delta_y \tilde U_r(\cdot,\tilde \mu^r(t))),\\
\tilde\mu^r(t_0) = \mu_0.
\end{cases} 
\] Define a function $\tilde u^r : [t_0,\iy) \times [d]\to \R$ by $\tilde u^r(t,x) := \tilde U_r(x,\tilde\mu(t))$. By the definitions of $\tilde \mu^r$, $\tilde u^r$, and $\tilde U_r$, we get by a similar computation to \eqref{eqn:star_35} that:
\begin{align*}
    -\frac{d}{dt} \tilde u^r(t,x) &= \sum_{y,z\in [d]} \tilde\mu^r_y(t) D^\eta_{yz} \tilde U_r(x,\tilde\mu^r(t))\gamma_z^*(y,\Delta_y \tilde U_r(\cdot,\tilde\mu^r(t))) \\
    &= -r\tilde U_r(x,\tilde\mu^r(t)) + H(x,\Delta_x \tilde U_r(\cdot,\tilde \mu^r(t))) + F(x,\tilde\mu^r(t)) \\ 
    &= -r\tilde u^r(t,x) + H(x,\Delta_x \tilde u^r(t,\cdot)) + F(x,\tilde\mu^r(t)).
\end{align*} So $(\tilde u^r, \tilde\mu^r)$ solves \eqref{MFG^r} and by uniqueness for that system, $\tilde u^r=u^r$, which leads to the equality $\tilde U_r(t,\tilde\mu(t)) = U_r(t,\mu(t))$ for all $t\in[t_0,\iy)$. Specifically, for $t=t_0$, 
$$\tilde U_r(t_0,\mu_0)=\tilde U_r(t_0,\tilde\mu(t_0)) = U_r(t_0,\mu(t_0))=U_r(t_0,\mu_0).$$
Since $t_0$ and $\mu_0$ are arbitrary, we obtain $\tilde U_r=U_r$.

\subsection{Proof of (iii)} By \eqref{master:linear_deriv} and \eqref{master:bar_U_D}, $D^\eta_{1y} U_r(x,\eta) = v^{r,0,\mu_0,e_{1y}}(0,x)$, and so by Corollary \ref{cor:MFG^r2}, there exist $C,r_0>0$ such that for all $x,y\in[d]$, $\eta \in\calP([d])$, $r\in (0,r_0)$, and $m\in\calM_0$,
\begin{align}\label{eq:D_eta_bound}
\Big| \frac{\pl}{\pl m} [U_r(x,\eta)]_1\Big| = |m\cdot D^\eta_1 U_r(x,\eta)| = |v^{r,0,\eta,m}(0,x)| \leq C|m|.
\end{align} In particular, \eqref{eq:D_eta_bound} holds for $m=e_{1z}.$ By \eqref{bound_u}, we have $|rU_r(x,\eta)| = |ru^r(0,x)|\leq  C_{f+F}$. Together, \eqref{eq:D_eta_bound} and this fact imply the first part of (iii).

We now turn to establish the uniform Lipschitz constant for the family $\{D^\eta_1 U_r(x,\eta)\}_{r\in (0,r_0)}$.  Let $v^{r,t_0,\mu_0,m_0}(t,x)$ be the solution of \eqref{lin_MFG_r_around_MFG^r} on $[t_0,\iy)$ with parameters $\mu(t_0)=\mu_0$ and $m(t_0) = m_0$. Recall from \eqref{master:linear_deriv} that $\bar D: \R_+ \times [d] \times \calP([d]) \to \R^d$ given by $\bar D = (\bar D_z)_{z\in [d]}$ satisfies 
$$m\cdot \bar D(t_0,x,\mu_0) = v^{r,t_0,\mu_0,m}(t_0,x).$$ By Lemma \ref{lemma:lipschitz_grad}, there exists $C_r>0$ such that for all $\eta_1,\eta_2 \in\calP([d])$:
\[
|U_r(x,\eta_1) - U_r(x,\eta_2) - \bar D(0,x,\eta_1)\cdot (\eta_1-\eta_2)| \leq C_r|\eta_1-\eta_2|^2.
\] From \eqref{master:bar_U_D} and \eqref{eq:Ur}, we have $D^\eta_1 U_r(x,\eta) = \bar D(0,x,\eta)$ for all $\eta\in\calP([d])$ and $\bar D$ is Lipschitz with Lipschitz constant $C_r$ depending on $r$. We can now upgrade the Lipschitz constant for $\bar D$ through $v$ into one that is uniform in $r\in (0,r_0)$. Namely, by \eqref{master:linear_deriv} and Proposition \ref{prop:linearized_lipschitz}, there exists $C>0$ such that for any $\eta_1,\eta_2\in\calP([d])$:
\begin{align*}
\frac{|D^\eta_1 U_r(x,\eta_1) - D^\eta_1 U_r(x,\eta_2)|}{|\eta_1 - \eta_2|} &\leq \sup_{m_0\in\calM_0}
\frac{|(D^\eta_1 U_r(x,\eta_1) - D^\eta_1 U_r(x,\eta_2))\cdot m_0|}{|\eta_1 - \eta_2||m_0|}\\
&=\sup_{m_0 \in\calM_0}\frac{|v^{r,0,\eta_1,m_0}(0,x) - v^{r,0,\eta_2,m_0}(0,x)|}{|\eta_1 - \eta_2||m_0|} \\ 
&\leq C .
\end{align*} So for any $x\in [d]$,
\[
\sup_{r\in (0,r_0)}\sup_{\eta_1\neq \eta_2} \frac{|D^\eta_1 U_r(x,\eta_1) - D^\eta_1 U_r(x,\eta_2)|}{|\eta_1 - \eta_2|} \leq C.
\] 

\subsection{ Proof of (iv):} Fix $\eta,\hat\eta\in\calP([d])$ and let $(u^r,\mu^r)$ and $(\hat u^r,\hat\mu^r)$ solve \eqref{MFG^r} with the initial data $\mu^r(0)=\eta$ and $\hat\mu^r(0)= \hat\eta$, respectively. By the definition of $U_r$ and Lemma \ref{lem:val-meas-duality},
\begin{align*}
    \sum_{x\in [d]}(U_r(x,\eta) - U_r(x,\hat\eta))(\eta_x-\hat\eta_x) = \sum_{x\in [d]} (u^r(0,x)-\hat u^r(0,x))(\eta_x-\hat\eta_x) \geq 0.
\end{align*}
\qed

\section{Proof of Theorem \ref{prop:ME}}\label{sec:pf_thm}
\subsection{Proof of (i)} 
Let $U_r$ be the solution to $\eqref{ME^r}$ and let $(\bar u^r,\bar\mu^r)$ be the stationary solution to the discounted MFG system \eqref{stat_MFG_r}. By the definition of $U_r$ from \eqref{eq:Ur} and Proposition \ref{prop:statDisc_erg_mfg}, there exist $C,r_0>0$ such that for all $r\in (0,r_0)$, and $x\in [d]$, $|\Delta_x U^r(\cdot,\bar\mu^r)| = |\Delta_x \bar u^r| \leq C\sqrt{r} + |\Delta_x \bar u|$ where $\bar u$ comes from \eqref{erg_MFG}. Since $D^\eta_1 U_r$ is Lipschitz uniformly in $r\in (0,r_0)$, we can say the same for $U_r$. Combining this with Proposition \ref{prop:MEr} (iii), we may update $C>0$ such that for any $\eta\in\calP([d])$, $x\in [d]$, and any $r\in (0,r_0)$,
\begin{align*}
 |U_r(x,\eta) - U_r(1,\bar\mu^r)| &\le |U_r(x,\eta) - U_r(x,\bar\mu^r)| + |U_r(x,\bar\mu^r) - U_r(1,\bar\mu^r)| \\
 &\le C |\eta- \bar\mu^r| + |U_r(x,\bar\mu^r) - U_r(1,\bar\mu^r)| \\
 &\le 2C + C\sqrt{r_0} + |\Delta \bar u|.
\end{align*} Then by passing the suprema,
\[
\sup_{r\in (0,r_0)}\sup_{\eta \in \calP([d])} |U_r (x,\eta) - U_r(1,\bar\mu^r)| \leq 2C + C\sqrt{r_0} + |\Delta \bar u|.
\] By the second part of Proposition \ref{prop:MEr} (iii), we have for all $x\in [d]$:
\[
\sup_{r\in (0,r_0)}\Big(\sup_{\eta \in \calP([d])}|D^\eta_1 U_{r}(x,\eta)| + \sup_{\eta_1 \neq \eta_2}\frac{|D^\eta_1 U_r(x,\eta_1)- D^\eta_1 U_r(x,\eta_2)|}{|\eta_1-\eta_2|}\Big)\leq C_2.
\] Therefore, for any $x\in[d]$, $(\|U_r(x,\cdot) - U_r(1,\bar\mu^r)\|_{\calC^1_{\text{Lip}}(\calP([d]))})_{r\in(0,r_0)}$ is bounded, where for any differentiable $g:\calP([d])\to \R$, we define: 
\begin{align}
&\|g(\cdot)\|_{\calC^1_{\text{Lip}}(\calP([d]))} := \sup_{\eta \in \calP([d])} |g(\eta)| +  \sup_{\eta \in \calP([d])} |D^\eta_1 g(\eta)|+\sup_{\eta_1\ne\eta_2}\frac{|D^\eta_1 g(\eta_1)-D^\eta_1 g(\eta_2)|}{|\eta_1-\eta_2|}.\notag
\end{align} Now, we make use of the fact that the bound is uniform in $r$ and employ the vanishing discount approach. By compactness, for any sequence of vanishing discounts $(r_n)_n\to 0$, there exists a subsequence, which by abuse of notation we also refer to as $(r_n)_n$, for which there exists $U_0(x,\cdot) \in \calC^1(\calP([d]))$ with $
\|U_0(x,\cdot)\|_{\calC^1(\calP([d]))} \leq C$ and where:
\[
U_{r_n}(x,\cdot) - U_{r_n}(1,\bar\mu^{r_n}) \to U_0(x,\cdot) \quad \text{ and } \quad D^\eta U_{r_n}(x,\cdot) \to D^\eta U_0(x,\cdot), 
\] with convergence in the uniform sense. Again by passing to a further subsequence, the real-valued sequence $\{r_n U_{r_n}(1,\bar\mu^{r_n})\}_{n\in\N}$  has a limit, which we denote by $\varrho_1$. In view of the master equation \eqref{ME^r}, using the Lipschitz continuity of $H$ and $\gamma^*$, and the vanishing discount:
\begin{align*}
\varrho_1 &=  \lim_{n\to\infty} [r_n U_{r_n}(1,\bar\mu^{r_n}) + r_n(U_{r_n}(x,\eta) - U_{r_n}(1,\bar\mu^{r_n}))] \\
&= \lim_{n\to\infty} r_n U_{r_n}(x,\eta)\\
&= \lim_{n\to\infty} \Big[ H(x,\Delta_x U_{r_n}(\cdot,\eta))+\sum_{y,z\in[d]}D^\eta_{yz} U_{r_n}(x,\eta)\eta_y\gamma^*_z(y,\Delta_y U_{r_n}(\cdot, \eta))+F(x,\eta)\Big] \\
&= H(x,\Delta_x U_0(\cdot,\eta))+\sum_{y,z\in[d]} D^\eta_{yz} U_0(x,\eta) \eta_y\gamma^*_z(y,\Delta_y U_0(\cdot,\eta))+F(x,\eta).
\end{align*}

\subsection{Proof of (ii)} This follows from the uniform constant used in the vanishing discount. 

\subsection{Proof of (iii)} From Proposition \ref{prop:MEr}, $U_r$ is monotone and so for any $\eta,\hat\eta \in\calP([d])$:
\begin{align*}
    0&\leq \sum_{x\in [d]} (\eta_x - \hat\eta_x)(U_r(x,\eta)-U_r(x,\hat\eta)) \\
    &=\sum_{x\in [d]} (\eta_x - \hat\eta_x)(U_r(x,\eta)-U_r(1,\bar\mu^r) + U_r(1,\bar\mu^r)-U_r(x,\hat\eta)).
\end{align*} Taking the limit along the subsequence as in the vanishing discount we obtain
\[
0\leq \sum_{x\in [d]} (\eta_x - \hat\eta_x)(U_0(x,\eta) - U_0(x,\hat\eta)).
\]

\subsection{Proof of (iv)}From the definition of $\vr_1$ above, the definition of $U_r(x,\bar\mu^r)$ (see \eqref{eq:Ur}, and Proposition \ref{prop:statDisc_erg_mfg}, 
\begin{equation}\label{erg_value_verif}
    \varrho_1 = \limn r_n U_{r_n}(1,\bar\mu^{r_n}) = \limn r_n \bar{u}^{r_n}_1 = \bar{\varrho}.
\end{equation} In this way, we eliminate the dependence on the choice of state $1$ for $\varrho$, the constant in \eqref{ME}. Moreover, since $U_r(x,\bar\mu^r)=\bar u^r_x$ and by \eqref{stat_MFG_r}, the vanishing discount yields:
\[
\sum_{y,z\in[d]} D^\eta_{yz} U_0(x,\bar\mu) \bar\mu_y\gamma^*_z(y,\Delta_y U_0(\cdot,\bar\mu)) = \lim_{n\to\infty} \Big[\sum_{y,z\in[d]}D^\eta_{yz} U_{r_n}(x,\bar\mu^{r_n}) \bar\mu^{r_n}_y\gamma^*_z(y,\Delta_y U_{r_n}(\cdot,\bar\mu^{r_n}))\Big]=0.
\] So for $\eta=\bar\mu$ we have
\[
\varrho = H(x,\Delta_x U_0(\cdot,\bar\mu))+F(x,\bar\mu).
\] Since $\varrho=\bar\varrho$, we see that $(\varrho, U_0(\cdot,\bar\mu),\bar\mu)$ solves the ergodic MFG \eqref{erg_MFG} and hence $U_0(\cdot,\bar\mu)=\bar u_{\cdot}$ up to a constant.

{\bf Uniqueness up to a constant: } On one hand, we note that for any solution $(\varrho,U)$ to \eqref{ME}, we can obtain another solution $(\varrho,U + c\Vec{1})$ where $\Vec{1}\in\R^d$ is the $d$-vector of all ones. This follows from the fact that the shift by $c\vec{1}$ does not change the derivatives $D^\eta_{yz} U(x,\eta)$ and since $\Delta_x (U+c\vec{1})(\cdot,\eta) = \Delta_x U(\cdot,\eta)$.

We now turn to prove the other direction. 
Let $(\varrho,U_0^i)$, $i=1,2$,  be two solutions to the ergodic master equation \eqref{ME}. 
Define $\mu^i:\R_+ \to \calP([d])$ as the unique solution to:
\[
\frac{d}{dt} \mu^i_x(t) = \sum_{y\in [d]} \mu^i_y(t) \gamma^*_x(y,\Delta_y U^i_0(\cdot,\mu^i(t))),
\] with $\mu^i(0) = \mu_0$. By the chain rule,
\[
\frac{d}{dt} U^i_0 (x,\mu^i(t)) = \sum_{y\in [d]} \mu^i_y(t) \gamma^*(y,\Delta_y U^i_0(\cdot,\mu^i(t))) \cdot D^\eta_y U^i_0 (x,\mu^i(t)).
\] Since $U^i_0$ satisfies \eqref{ME}, and setting $u^i_x(t):=U^i_0(x,\mu^i(t))$, we find that $(u^i,\mu^i)$ satisfies the following system:
\begin{align}
\label{sys:finite_erg_Z}
\begin{cases}
-\frac{d}{dt}  u^i_x(t) +\bar\vr= H(x,\Delta_x  u^i(t)) + F(x, \mu^i(t)),\\
\frac{d}{dt}  \mu^i_x (t) = \sum_{y\in [d]}  \mu^i_y(t) \gamma^*_x(y,\Delta_y  u^i(t)), \\
\mu^i(0) = \mu_0, 
\quad t\in\R_+.\end{cases}
\end{align} 
Uniqueness for this system follows by the same arguments given earlier in the paper.

A similar computation to Lemma \ref{lem:val-meas-duality}, along with the fact that $\mu^1(0)=\mu^2(0)=\mu_0$ implies that for any $T>0$,
\begin{align*}
    \int_0^T \sum_{x\in [d]} |\Delta_x (u^1(t)-u^2(t))|^2 (\mu^1_x(t) + \mu^2_x(t)) dt \leq C|\Delta (u^1-u^2)(T)||(\mu^1-\mu^2)(T)| .
\end{align*} 
From \cite[Theorem 3]{Gomes2013}\footnote{It is worth mentioning that this theorem from \cite{Gomes2013} relies on the additional contraction assumptions of that work. However, their proof only uses the contractive assumptions to bound their potential functions $u(t)$ in a sub-additive (non-norm) function $\|\cdot\|_{\sharp}$; here, $\|u(t)\|_{\sharp}$ measures the maximum difference between potentials of any two states from $[d]$ at time $t\in\R_+$ and so is equivalent to $|\Delta u^i(t)|$ in our case. Namely, for any measurable function $b:\R_+\to \R^d$, the quantity $\|b(t)\|_{\sharp}$ is bounded for all $t\in\R_+$ if and only if $|\Delta b(t)|$ is bounded for all $t\in\R_+$. Recall that we found that $\Delta_x u^i(t) = \Delta_x U^i_0(\cdot,\mu^i(t))$ and because $U^i_0$ is Lipschitz, $|\Delta_x U^i_0(\cdot,\mu^i(t))| \leq C(|\Delta_x \bar u| + |\mu^i(t) - \bar\mu|)$. So, we obtain boundedness independently of the contractive assumptions and may otherwise reuse the proof from \cite[Theorem~3]{Gomes2013}.}, we have that $\mu^1(t),\mu^2(t) \to \bar\mu$ as $t\to\iy$, where $(\vr,\bar u ,\bar\mu)$ is the unique solution to \eqref{erg_MFG}. Therefore,
\begin{align*}
    \lim_{T\to\iy} \int_0^T \sum_{x\in [d]} |\Delta_x (u^1- u^2)(t)|^2 ( \mu_x^1(t) + \mu_x^2(t)) dt \leq 0.
\end{align*} Note that if $ \mu^i_x(0) = 0$, then
\[
\frac{d}{dt} \mu^i_x(0) = \sum_{y,y\neq x} \mu^i_y(0) \gamma^*_x(y,\Delta_y u^i(0)) > \mathfrak{a}_l>0.
\] So, for any given $\eps>0$, on $[\eps,T]$, the process $ \mu^1+ \mu^2$ is bounded away from zero and hence $\Delta_x u^1(t) = \Delta_x  u^2(t)$ for $t\in [\eps,T]$ and for all $x\in [d]$. Taking $\eps\to 0+$, $\Delta_x  u^1(t) = \Delta_x  u^2(t)$ for all $t\in [0,T]$ and $x\in [d]$. Using uniqueness of the Kolmogorov's equation from \eqref{sys:finite_erg_Z}, we obtain that $\mu^1 = \mu^2$. Combining these results implies that $(d/dt)  u^1_x(t) = (d/dt)  u^2_x(t)$, and so, there exists a constant $c\in\R$ (that may depend on $\mu_0$), such that for any $(t,x)\in\R_+\times[d]$, $u^1(t,x) = u^2(t,x) + c$. As a result, $U^1_0(x,\mu(t)) = U^2_0 (x,\mu(t)) + c$ for all $(t,x)\in\R_+\times[d]$. 

Now, we are going to show that actually $c$ is independent of $\mu_0$. Take $\mu_0 = \bar\mu$, which in turn implies $\mu^1(t) = \mu^2(t) = \bar\mu$ and denote the corresponding constant $c$ that arises from the same computation as we had previously as $\bar c \in \R$. Using the Lipschitz continuity of $U^i_0$,
\begin{align*}
    |c- \bar c| &= \lim_{t\to\iy} |U^1_0(x,\mu^1(t)) - U^2_0(x,\mu^2(t)) - U^1_0(x,\bar\mu) + U^2_0 (x,\bar\mu)| \\
    &\leq \lim_{t\to\iy}\big\{|U^1_0(x,\mu^1(t)) - U^1_0(x,\bar\mu)| + |U^2_0(x,\mu^2(t)) - U^2_0 (x,\bar\mu)|\big\} \\
    &\leq C \lim_{t\to\iy} \max_{i=1,2}|\mu^i(t)  - \bar\mu| \\
    &= 0.
\end{align*} So, $c=\bar c$ and the constant is independent of the choice of $\mu_0$. Consequently, $U^1_0(x,\mu_0) = U^2_0 (x,\mu_0) + \bar c$ for all $(x,\mu_0)\in[d]\times\calP([d])$.
\qed

\vspace{1cm}
{\bf Acknowledgement.} We thank the anonymous AE and the  referees for their suggestions, which helped us improve our paper.

\color{black}
\footnotesize
\bibliographystyle{abbrv} 
\bibliography{bib_Asaf_IMS} 

\begin{thebibliography}{10}

\bibitem{Achdou_book}
Y.~Achdou, P.~Cardaliaguet, F.~Delarue, A.~Porretta, and F.~Santambrogio.
\newblock {\em Mean Field Games}.
\newblock Springer, 2019.

\bibitem{MR3597374}
A.~Arapostathis, A.~Biswas, and J.~Carroll.
\newblock On solutions of mean field games with ergodic cost.
\newblock {\em J. Math. Pures Appl. (9)}, 107(2):205--251, 2017.

\bibitem{BCCD-arxiv}
E.~Bayraktar, A.~Cecchin, A.~Cohen, and F.~Delarue.
\newblock {Finite state mean field games with Wright Fisher common noise as
  limits of N-player weighted games}.
\newblock {\em Math. Oper. Res., to appear}, 2021.

\bibitem{BCCD}
E.~Bayraktar, A.~Cecchin, A.~Cohen, and F.~Delarue.
\newblock {Finite state MFG{s} with Wright--Fisher common noise}.
\newblock {\em J. de Math. Pures et App}, 2021.

\bibitem{bay-coh2019}
E.~Bayraktar and A.~Cohen.
\newblock Analysis of a finite state many player game using its master
  equation.
\newblock {\em SIAM Journal on Control and Optimization}, 56(5):3538--3568,
  2018.

\bibitem{Bensoussan2013}
A.~Bensoussan, J.~Frehse, and P.~Yam.
\newblock {\em Mean field games and mean field type control theory}.
\newblock SpringerBriefs in Mathematics. Springer, New York, 2013.

\bibitem{ben-jam-yam2019}
A.~{Bensoussan}, P.~{Jameson Graber}, and S.~C.~P. {Yam}.
\newblock Stochastic control on space of random variables.
\newblock {\em arXiv e-prints}, page arXiv:1903.12602, Mar. 2019.

\bibitem{JEP_2021__8__1099_0}
C.~Bertucci.
\newblock Monotone solutions for mean field games master equations: finite
  state space and optimal stopping.
\newblock {\em Journal de l{\textquoteright}\'Ecole polytechnique {\textemdash}
  Math\'ematiques}, 8:1099--1132, 2021.

\bibitem{buc-li-pen-rai2017}
R.~Buckdahn, J.~Li, S.~Peng, and C.~Rainer.
\newblock Mean-field stochastic differential equations and associated {PDE}s.
\newblock {\em Ann. Probab.}, 45(2):824--878, 2017.

\bibitem{bur-ign-rep2020}
M.~Burzoni, V.~Ignazio, A.~M. Reppen, and H.~M. Soner.
\newblock Viscosity solutions for controlled {M}c{K}ean-{V}lasov
  jump-diffusions.
\newblock {\em SIAM J. Control Optim.}, 58(3):1676--1699, 2020.

\bibitem{CardaliaguetDelarueLasryLions}
P.~Cardaliaguet, F.~Delarue, J.-M. Lasry, and P.-L. Lions.
\newblock {\em The master equation and the convergence problem in mean field
  games}, volume 201 of {\em Annals of Mathematics Studies}.
\newblock Princeton University Press, Princeton, NJ, 2019.

\bibitem{car-por}
P.~Cardaliaguet and A.~Porretta.
\newblock Long time behavior of the master equation in mean-field game theory.
\newblock {\em Analysis \& PDE}, 2017.

\bibitem{CarmonaDelarue_book_I}
R.~Carmona and F.~Delarue.
\newblock {\em Probabilistic theory of mean field games with applications.
  {I}}, volume~83 of {\em Probability Theory and Stochastic Modelling}.
\newblock Springer, Cham, 2018.
\newblock Mean field FBSDEs, control, and games.

\bibitem{CarmonaDelarue_book_II}
R.~Carmona and F.~Delarue.
\newblock {\em Probabilistic theory of mean field games with applications.
  {II}}, volume~84 of {\em Probability Theory and Stochastic Modelling}.
\newblock Springer, Cham, 2018.
\newblock Mean field games with common noise and master equations.

\bibitem{cec-pel2019}
A.~Cecchin and G.~Pelino.
\newblock Convergence, fluctuations and large deviations for finite state mean
  field games via the master equation.
\newblock {\em Stochastic Processes and their Applications}, 129(11):4510 --
  4555, 2019.

\bibitem{cha-cri-del2019}
J.-F. Chassagneux, D.~Crisan, and F.~Delarue.
\newblock Numerical method for {FBSDE}s of {M}c{K}ean-{V}lasov type.
\newblock {\em Ann. Appl. Probab.}, 29(3):1640--1684, 2019.

\bibitem{del-lac-ram2019}
F.~Delarue, D.~Lacker, and K.~Ramanan.
\newblock From the master equation to mean field game limit theory: a central
  limit theorem.
\newblock {\em Electron. J. Probab.}, 24:Paper No. 51, 54, 2019.

\bibitem{del-lac-ram2020}
F.~Delarue, D.~Lacker, and K.~Ramanan.
\newblock From the master equation to mean field game limit theory: large
  deviations and concentration of measure.
\newblock {\em Ann. Probab.}, 48(1):211--263, 2020.

\bibitem{Doob}
J.~L. Doob.
\newblock {\em Stochastic Processes}.
\newblock Wiley Classics Library. A Wiley Interscience Publication, 1990.
\newblock Reprint of the 1953 original.

\bibitem{MR3127148}
E.~Feleqi.
\newblock The derivation of ergodic mean field game equations for several
  populations of players.
\newblock {\em Dyn. Games Appl.}, 3(4):523--536, 2013.

\bibitem{gan-mes-mou-zha2022}
W.~Gangbo, A.~R. M\'{e}sz\'{a}ros, C.~Mou, and J.~Zhang.
\newblock Mean field games master equations with nonseparable {H}amiltonians
  and displacement monotonicity.
\newblock {\em Ann. Probab.}, 50(6):2178--2217, 2022.

\bibitem{gan-swi2015}
W.~Gangbo and A.~\'{S}wi\c ech.
\newblock Existence of a solution to an equation arising from the theory of
  mean field games.
\newblock {\em J. Differential Equations}, 259(11):6573--6643, 2015.

\bibitem{Gomes2013}
D.~A. Gomes, J.~Mohr, and R.~R. Souza.
\newblock Continuous time finite state mean field games.
\newblock {\em Appl. Math. Optim.}, 68(1):99--143, 2013.

\bibitem{MR3059616}
D.~A. Gomes, J.~Mohr, and R.~R.~a. Souza.
\newblock Discrete time, finite state space mean field games.
\newblock In {\em Dynamics, games and science. {I}}, volume~1 of {\em Springer
  Proc. Math.}, pages 385--389. Springer, Heidelberg, 2011.

\bibitem{MR2554588}
X.~Guo and O.~Hern\'{a}ndez-Lerma.
\newblock {\em Continuous-time {M}arkov decision processes}, volume~62 of {\em
  Stochastic Modelling and Applied Probability}.
\newblock Springer-Verlag, Berlin, 2009.
\newblock Theory and applications.

\bibitem{Huang2006}
M.~Huang, R.~P. Malham{\'e}, and P.~E. Caines.
\newblock Large population stochastic dynamic games: Closed-loop
  {M}c{K}ean-{V}lasov systems and the {N}ash certainty equivalence principle.
\newblock {\em Commun. Inf. Syst.}, 6(3):221--251, 2006.

\bibitem{Lasry2006}
J.-M. Lasry and P.-L. Lions.
\newblock Jeux \`a champ moyen. {I}. {L}e cas stationnaire.
\newblock {\em C. R. Math. Acad. Sci. Paris}, 343(9):619--625, 2006.

\bibitem{LasryLions2}
J.-M. Lasry and P.-L. Lions.
\newblock Jeux \`a champ moyen. {II}. {H}orizon fini et contr\^ole optimal.
\newblock {\em C. R. Math. Acad. Sci. Paris}, 343(10):679--684, 2006.

\bibitem{LasryLions}
J.-M. Lasry and P.-L. Lions.
\newblock Mean field games.
\newblock {\em Jpn. J. Math.}, 2(1):229--260, 2007.

\bibitem{mou-zha2019}
C.~{Mou} and J.~{Zhang}.
\newblock Wellposedness of second order master equations for mean field games
  with nonsmooth data.
\newblock {\em arXiv e-prints}, page arXiv:1903.09907, Mar. 2019.

\bibitem{2022arXiv220110762M}
C.~{Mou} and J.~{Zhang}.
\newblock {Mean Field Game Master Equations with Anti-monotonicity Conditions}.
\newblock {\em arXiv e-prints}, page arXiv:2201.10762, Jan. 2022.

\bibitem{MR4064670}
P.~Wi\c{e}cek.
\newblock Discrete-time ergodic mean-field games with average reward on compact
  spaces.
\newblock {\em Dyn. Games Appl.}, 10(1):222--256, 2020.

\bibitem{wu-zha2020}
C.~Wu and J.~Zhang.
\newblock Viscosity solutions to parabolic master equations and
  {M}c{K}ean-{V}lasov {SDE}s with closed-loop controls.
\newblock {\em Ann. Appl. Probab.}, 30(2):936--986, 2020.

\end{thebibliography}

\end{document}